\documentclass[12pt, reqno]{amsart}
\usepackage{amssymb, mathrsfs}
\usepackage{latexsym}
\usepackage{amsmath}
\usepackage{amsthm}
\usepackage{amsfonts}
\usepackage{dsfont, upgreek}
\usepackage{mathtools}
\usepackage{epsfig}
\usepackage{amscd}
\usepackage{graphicx}

\usepackage{overpic}
\usepackage{tikz-cd }

\usepackage{tikz}
\usetikzlibrary{decorations.markings}
\usetikzlibrary{arrows.meta}
\usetikzlibrary{calc, fadings}
\usepackage{pgfplots}
\pgfplotsset{compat=1.18}

\usepackage{tikz-3dplot}

\usepackage{enumitem}
\setlist[enumerate]{itemsep=0.5ex}

\usepackage{stmaryrd}

\usepackage{color}

\usepackage[left=1.4in,right=1.4in,top=1.2in,bottom=1.2in]{geometry}

\usetikzlibrary{decorations.markings}
\usetikzlibrary{arrows.meta}

\usepackage{extarrows}

\usepackage{adjustbox}
\usepackage{slashed}

\usepackage[colorlinks=true,linkcolor=blue,citecolor=blue]{hyperref}

\usepackage[colorinlistoftodos,prependcaption,textsize=tiny]{todonotes}  

\usepackage[all,cmtip]{xy}
\theoremstyle{plain}
\newtheorem{theorem}{Theorem}[section]
\newtheorem{proposition}[theorem]{Proposition}
\newtheorem{lemma}[theorem]{Lemma}
\newtheorem{corollary}[theorem]{Corollary}

\newtheorem{conjecture}[theorem]{Conjecture}

\theoremstyle{definition} 
\newtheorem{definition}[theorem]{Definition}
\newtheorem{example}[theorem]{Example}
\newtheorem*{claim*}{Claim}
\newtheorem{claim}[theorem]{Claim}

\theoremstyle{remark} 
\newtheorem{remark}[theorem]{Remark}

\numberwithin{equation}{section}

\newcommand{\Sc}{\mathrm{Sc}}

\newcommand{\Bigwedge}{\mathord{\adjustbox{raise=.4ex, totalheight=.7\baselineskip}{$\bigwedge$}}}

\newcommand{\fiber}{\mathbb F}
\newcommand{\link}{\mathbb L}

\newcommand{\ind}{\textup{Ind}}
\newcommand{\id}{\mathrm{id}}

\newcommand{\dist}{\mathrm{dist}}

\newcommand{\R}{\mathbb{R}}
\newcommand{\tr}{\mathrm{tr}}

\newcommand{\Z}{\mathbb{Z}}

\newcommand{\dR}{\mathrm{dR}}
\newcommand{\dom}{\mathrm{dom}}
\newcommand{\sph}{\mathbb{S}}
\newcommand{\normal}{\mathbf{n}}

\newcommand{\tancone}{\mathcal C}
\newcommand{\normalcone}{\mathcal C^\perp}
\newcommand{\domain}{\overbar M}
\newcommand{\target}{M}
\newcommand{\hsurface}{\Sigma}
\newcommand{\sector}{\mathbb G}

\newcommand{\interior}[1]{%
	{\kern0pt#1}^{\mathrm{\,o}}%
}

\makeatletter
\let\save@mathaccent\mathaccent
\newcommand*\if@single[3]{%
	\setbox0\hbox{${\mathaccent"0362{#1}}^H$}%
	\setbox2\hbox{${\mathaccent"0362{\kern0pt#1}}^H$}%
	\ifdim\ht0=\ht2 #3\else #2\fi
}
\newcommand*\rel@kern[1]{\kern#1\dimexpr\macc@kerna}
\newcommand*\wideaccent[2]{\@ifnextchar^{{\wide@accent{#1}{#2}{0}}}{\wide@accent{#1}{#2}{1}}}
\newcommand*\wide@accent[3]{\if@single{#2}{\wide@accent@{#1}{#2}{#3}{1}}{\wide@accent@{#1}{#2}{#3}{2}}}
\newcommand*\wide@accent@[4]{%
	\begingroup
	\def\mathaccent##1##2{%
		\let\mathaccent\save@mathaccent
		\if#42 \let\macc@nucleus\first@char \fi
		\setbox\z@\hbox{$\macc@style{\macc@nucleus}_{}$}%
		\setbox\tw@\hbox{$\macc@style{\macc@nucleus}{}_{}$}%
		\dimen@\wd\tw@
		\advance\dimen@-\wd\z@
		\divide\dimen@ 3
		\@tempdima\wd\tw@
		\advance\@tempdima-\scriptspace
		\divide\@tempdima 10
		\advance\dimen@-\@tempdima
		\ifdim\dimen@>\z@ \dimen@0pt\fi
		\rel@kern{0.6}\kern-\dimen@
		\if#41
		#1{\rel@kern{-0.6}\kern\dimen@\macc@nucleus\rel@kern{0.4}\kern\dimen@}%
		\advance\dimen@0.4\dimexpr\macc@kerna
		\let\final@kern#3%
		\ifdim\dimen@<\z@ \let\final@kern1\fi
		\if\final@kern1 \kern-\dimen@\fi
		\else
		#1{\rel@kern{-0.6}\kern\dimen@#2}%
		\fi
	}%
	\macc@depth\@ne
	\let\math@bgroup\@empty \let\math@egroup\macc@set@skewchar
	\mathsurround\z@ \frozen@everymath{\mathgroup\macc@group\relax}%
	\macc@set@skewchar\relax
	\let\mathaccentV\macc@nested@a
	\if#41
	\macc@nested@a\relax111{#2}%
	\else
	\def\gobble@till@marker##1\endmarker{}%
	\futurelet\first@char\gobble@till@marker#2\endmarker
	\ifcat\noexpand\first@char A\else
	\def\first@char{}%
	\fi
	\macc@nested@a\relax111{\first@char}%
	\fi
	\endgroup
}
\makeatother

\newcommand\overbar{\wideaccent\overline}

\makeatletter
\newcommand*{\transpose}{%
	{\mathpalette\@transpose{}}%
}
\newcommand*{\@transpose}[2]{%
	\raisebox{\depth}{$\m@th#1\intercal$}%
}
\makeatother

\begin{document}		
	
	\title{Gromov's dihedral rigidity conjecture in  dimension three}
	
	\author{Jinmin Wang}
	\address[Jinmin Wang]{State Key Laboratory of Mathematical Sciences, Academy of
		Mathematics and Systems Science, Chinese Academy of Sciences}
	\email{jinmin@amss.ac.cn}
	\thanks{}
	\author{Zhizhang Xie}
	\address[Zhizhang Xie]{ Department of Mathematics, Texas A\&M University }
	\email{xie@tamu.edu}
	\thanks{}
	\author{Guoliang Yu}
	\address[Guoliang Yu]{ Department of
		Mathematics, Texas A\&M University}
	\email{guoliangyu@tamu.edu}
	\thanks{}
	
	\begin{abstract}
		
 In this article, we present a self-contained proof of Gromov’s dihedral rigidity conjecture on scalar curvature in the three-dimensional case. The proof avoids many of the technical complications that arise in higher dimensions, while still illustrating the essential ideas of the general approach developed in \cite{Wang:2021tq} and \cite{Wang:2022vf}. It is significantly shorter than the proof of the general case and is intended to be more accessible.
		
%

	\end{abstract}
	\maketitle

\section{Introduction}\label{sec:intro}
Gromov's dihedral rigidity conjecture states that \begin{conjecture}[{\cite{GromovDiracandPlateau, MR3822551}}]\label{conj:dihedralrigidity}
	Let $P$ be a convex polyhedron in $\R^n$ and $g_0$ the Euclidean metric on $P$. If $g$ is a smooth Riemannian metric on $P$ such that the scalar curvature, mean curvature, and dihedral angles satisfy
	\begin{enumerate}[label=$(\arabic*)$]
		\item $\Sc(g)\geq \Sc(g_0) = 0$,
		\item $H_g(F_i)\geq H_{g_0}(F_i) = 0$ for each codimension one face $F_i$ of $P$, and
		\item $\theta_{ij}(g)\leq \theta_{ij}(g_0)$ on $F_{ij} = F_i\cap F_j$ for each pair of adjacent  codimension one faces $F_i$ and $F_j$,	
	\end{enumerate}
	then we have 
	\[ \Sc(g)=0, H_g(F_i) = 0 \textup{ and }  \theta_{ij}(g) =  \theta_{ij}(g_0)\]
	for all $i$ and all $j\neq i$. Moreover, $(P, g)$ is flat.
\end{conjecture}	
Here a pair of codimension one faces of $P$ are called adjacent if their intersection is a nonempty codimension two face of $P$.

In dimension three, Li  proved the dihedral rigidity conjecture for cone-type or prism type polyhedra under additional angle restrictions \cite{Lichaocomparison}. We also would like to point out  Lott  established an analogous rigidity theorem for even-dimensional manifolds with smooth boundaries (in which case dihedral angles do not appear) \cite{Lottboundary}.

In \cite{Wang:2021tq} and \cite{Wang:2022vf}, the authors proved the above conjecture--and in fact, established stronger generalizations of the conjecture--in all dimensions. The proof in the higher-dimensional case is rather technical. Following Gromov's suggestions, we present in this paper a self-contained and significantly shorter proof of Gromov's dihedral rigidity conjecture in the three-dimensional case. By restricting our attention to the three-dimensional case, we bypass many of the technical complications arising in the general setting, while still illustrating the essential ideas of the general approach.

%
	
More precisely, in this paper we give a self-contained proof of  the following strengthened version of Gromov's dihedral rigidity conjecture in dimension $3$. 
	
\begin{theorem}\label{thm:main2}
		Let $(\target, g)$ be a convex polyhedron in the Euclidean space $\R^3$, where $g$ is the Euclidean metric. Let $(\domain,\overbar g)$ be a connected spin polyhedral manifold of dimension three and $f\colon \domain \to \target$ be a polyhedral map with non-zero degree, such that the scalar curvature, mean curvature, and dihedral angles satisfy 
		\begin{enumerate}[label=$(\arabic*)$]
			\item $\Sc(\overbar g)\geq \Sc(g) = 0$,
			\item $H_{\overbar g}(\overbar F_i)\geq  0$ for each codimension one face $\overbar F_i$ of $\domain$, and
			\item $\theta_{ij}(\overbar g)\leq f^\ast\theta_{ij}(g)$ on $\overbar F_{ij} = \overbar F_i\cap \overbar F_j$ for each pair of adjacent  codimension one faces $\overbar F_i$ and $\overbar F_j$,	
		\end{enumerate}  then $\Sc_{\overbar g}=0$, $H_{\overbar g}=0$, and $\theta_{\overbar g}=f^*\theta_{g}$. Moreover, $(\domain,\overbar g)$ is flat.
\end{theorem}
The precise definitions of polyhedral manifolds and polyhedral maps are given in Section \ref{sec:polymanifold}. In this setting, the map $f$ induces a homomorphism between the top relative homology groups:$$f_* \colon H_n(\domain,  \partial \domain; \mathbb{Z}) \to H_n(\target, \partial \target; \mathbb{Z}).$$The degree of $f$ is defined as the unique integer $\deg(f)$ such that$$f_*([\domain, \partial \domain]) = \deg(f)[\target, \partial \target].$$

Our strategy for proving Theorem \ref{thm:main2} utilizes twisted Dirac operators equipped with suitable elliptic boundary conditions. A key step is to construct a solution to an appropriate twisted Dirac operator, subject to a natural  boundary condition determined by the geometric setup, so that the classical Bochner-Lichnerowicz-Weitzenb\"ock formula can be applied to obtain estimates for the scalar curvature and mean curvature (cf. Proposition \ref{prop:smooth>=}). A further computation using this solution also yields estimates for the  dihedral angles (cf. Lemma \ref{lemma:angleRigidity}).
However, the presence of dihedral angles and higher-codimensional singularities introduces substantial analytical difficulties. In particular, these singularities make it considerably more challenging to formulate the relevant index theory rigorously and to carry out the associated index computations. Overcoming these difficulties constitutes the main contribution of our work.

   Let us give a brief overview of the paper. Since $f\colon \domain \to \target$ is a spin map, the bundle $T\domain \oplus f^\ast T\target$ admits a spinor bundle $E = S(T\domain \oplus f^\ast T\target)$. Let $D$ be the (twisted) Dirac operator associated with the bundle $E$. There is a natural boundary condition $B$ (cf. Definition \ref{def:boundaryCondition}) determined by the unit inner normal vectors of the codimension-one faces of $\domain$ and their corresponding codimension-one faces in $\target$. 

A main ingredient of our proof is developing an index theory for this type of Dirac operator $D_B$ subject to these boundary conditions.

\textbf{(1).} First, we prove that the Dirac operator under consideration is Fredholm, provided that the dihedral angles of $\domain$ are strictly less than $\pi$ and satisfy the prescribed comparison condition (Theorem~\ref{thm:ess-saVector}). We begin by establishing essential self-adjointness. Because $\domain$ has polyhedral corners, we study $D_B$ within Cheeger's framework for conic-type operators (cf.~\cite{MR530173,MR573430,JC83}). A standard partition-of-unity argument shows that it is enough to prove that $D_B$ is locally essentially self-adjoint near every point of $\domain$ (cf.~Definition~\ref{def:ess-saLocal}). At points in the interior of $\domain$ or in the  interior of a codimension-one face, local essential self-adjointness follows from classical elliptic regularity theory.
	
	At points on higher-codimension faces, we establish a key reduction theorem. Under the relevant comparison condition on the dihedral angles, $D_B$ is locally essentially self-adjoint near a point $x \in \domain$ if and only if the corresponding model operator on the tangent cone at $x$, equipped with the induced boundary condition, is essentially self-adjoint (Theorem~\ref{thm:reduction}). Thus, the general problem reduces to model problems on Euclidean polyhedral cones.
	
	 If $x$ lies in the  interior of a codimension-two face (that is, an edge) of $\domain$, Theorem~\ref{thm:reduction} reduces the analysis to two-dimensional sectors in $\mathbb R^2$ whose angles satisfy the required inequality. The relevant essential self-adjointness results are established in Lemmas~\ref{lemma:essensa-jumpanglewithf} and~\ref{lemma:ess-sa-innerproductcomparison}. Similarly, if $x$ lies in a codimension-three face (that is, a vertex) of $\domain$, the problem reduces to three-dimensional polyhedral cones in $\mathbb R^3$. This case is treated in Lemmas~\ref{lemma:ess-saSphericalLink} and~\ref{lemma:spectrumSphericalLink} and Corollary~\ref{corollary:ess-saThreeDimension}.
	
	Once $D_B$ is known to be essentially self-adjoint, its domain is the Sobolev space $H^1(\domain, E; B)$ of $H^1$ sections of $E$ over $\domain$ satisfying the boundary condition $B$. By the Rellich lemma, the inclusion $H^1(\domain, E; B) \hookrightarrow L^2(\domain, E)$ is compact. A standard argument then shows that $D_B$ is Fredholm.
	
\textbf{(2).} Second, we compute the Fredholm index $\ind(D_B)$ (Theorem~\ref{thm:indexTheoremVector}). A direct computation on $\domain$ is difficult because of its singularities. Instead, we use a cutting-and-pasting argument to reduce the problem to manifolds with smooth boundary, where the index computation is classical~\cite{AtiyahBott}; see also~\cite{Lottboundary}.

	\begin{figure}[h]
		\centering	
		\begin{tikzpicture}
			\node[anchor=south west,inner sep=0] (image) at (0,0) {\includegraphics[width=.9\textwidth]{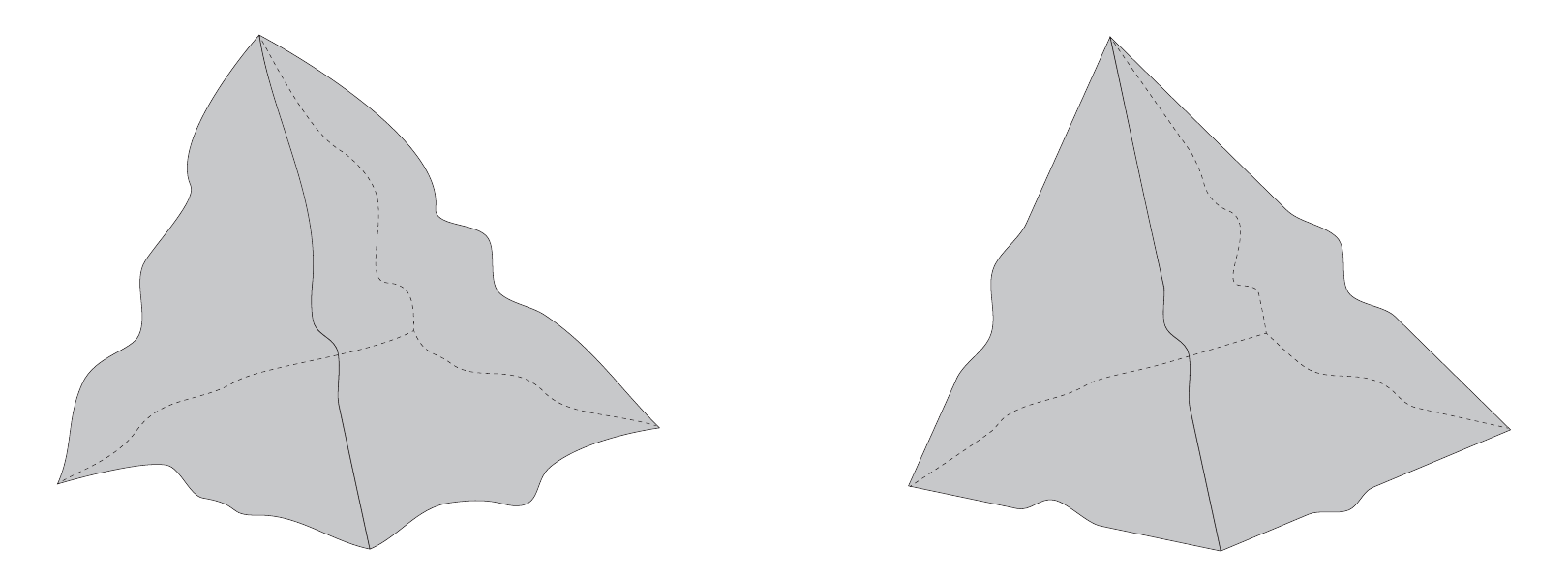}};
			\begin{scope}[x={(image.south east)},y={(image.north west)}]
				\draw[
				thick,
				decorate,
				decoration={snake, amplitude=0.6mm, segment length=2mm}
				]
				(0.45,0.62) -- (0.51,0.62);
				
				\draw[
				thick,
				-{Stealth[length=3mm,width=2mm]}
				]
				(0.51,0.62) -- (0.53,0.62);
				
				
			\end{scope}
		\end{tikzpicture}		
		\caption{Deformation of small neighborhoods of the vertices of $\protect\domain$ that makes both the metric and the adjacent faces flat.}
		\label{fig:deformation-near-vertices}
	\end{figure}

	We first deform the boundary conditions (Proposition~\ref{thm:algebraicDeformation}) and the metric in small neighborhoods of the vertices of $\domain$ so that both the metric and the faces become flat near each vertex (see Figure~\ref{fig:deformation-near-vertices}), without changing the Fredholm index. For technical reasons, we work with boundary conditions more general than those naturally induced by the  unit inner normal vector fields on the codimension-one faces of $\domain$ and $\target$ (cf.~Lemma~\ref{lemma:algAngle} and Proposition~\ref{thm:algebraicDeformation}). We then apply the gluing formula of Theorem~\ref{thm:gluing}. Suppose that $\domain$ is cut into two pieces, $\domain_1$ and $\domain_2$, along a hypersurface $\hsurface$ that meets every codimension-one face it intersects orthogonally, and that the corresponding boundary condition is imposed along $\hsurface$. The formula states that the Fredholm index of $D_B$ on $\domain$ is the sum of the Fredholm indices of the corresponding Dirac operators on $\domain_1$ and $\domain_2$. It therefore allows us to cut off a small polyhedral neighborhood of each vertex along a spherical link (see Figure~\ref{fig:cutting-and-pasting-vertices}).\footnote{More precisely, we cut along a convex hypersurface that is a slight modification of the spherical link; see Example~\ref{example:hypersurface}.} A crucial point is to show that the Fredholm index on each cut-off piece vanishes. We prove this by combining a deformation argument with the approximation result in Lemma~\ref{lemma:approximation} and Corollary~\ref{coro:approximation}.
	
%
		\begin{figure}[h]
		\centering	
		\begin{tikzpicture}
			\node[anchor=south west,inner sep=0] (image) at (0,0) {\includegraphics[width=.9\textwidth]{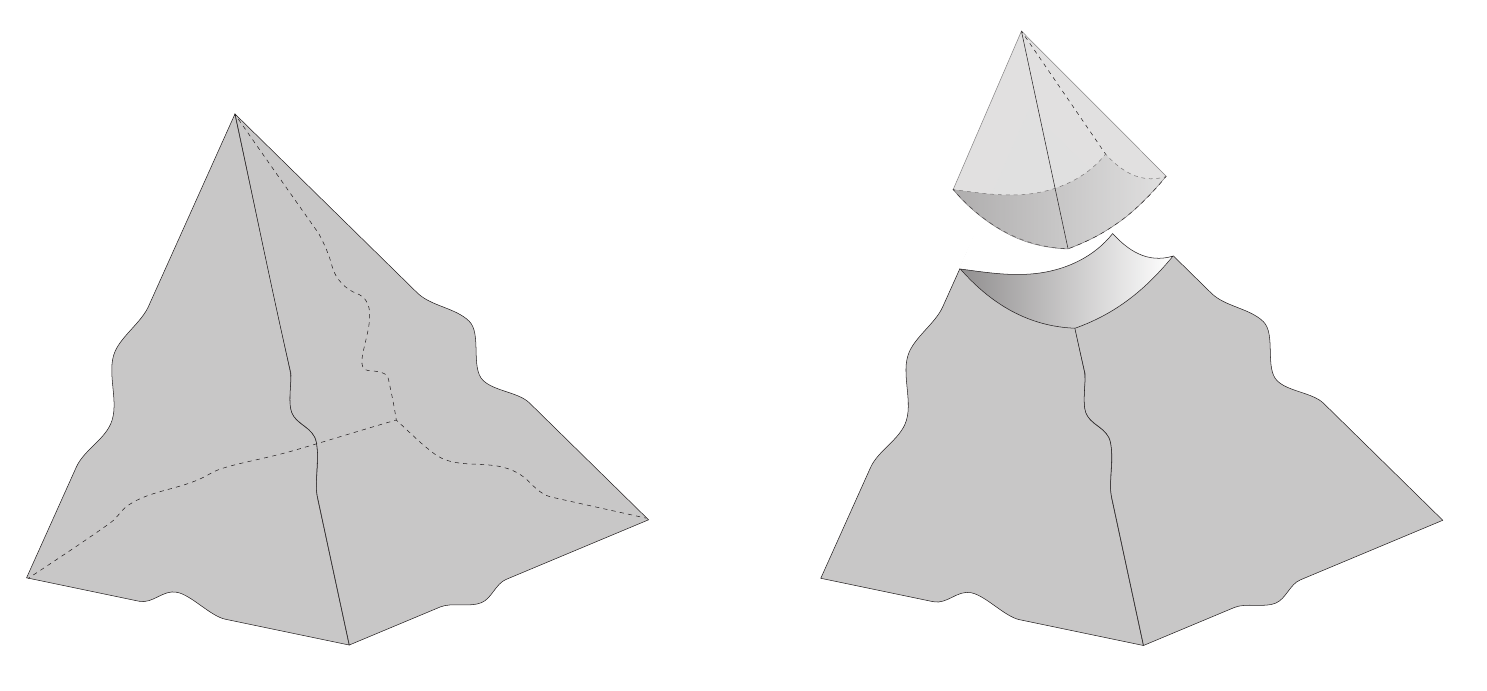}};
			\begin{scope}[x={(image.south east)},y={(image.north west)}]
				\draw[
				thick,
				decorate,
				decoration={snake, amplitude=0.6mm, segment length=2mm}
				]
				(0.45,0.62) -- (0.51,0.62);
				
				\draw[
				thick,
				-{Stealth[length=3mm,width=2mm]}
				]
				(0.51,0.62) -- (0.53,0.62);
				
				
			\end{scope}
		\end{tikzpicture}		
		\caption{Cutting off a small neighborhood of each vertex using the gluing formula (Theorem~\ref{thm:gluing}). In the right-hand figure, some hidden dashed edges are omitted.}
		\label{fig:cutting-and-pasting-vertices}
	\end{figure} 
	
	Applying this procedure at every vertex of $\domain$ yields a new polyhedral manifold $\domain'$ (see Figure~\ref{fig:cutting-and-pasting-vertices-2}). Removing a small neighborhood of a vertex $v \in \domain$ creates new vertices and edges. To distinguish the original edges from the newly created ones, let $\mathcal E_{\domain}$ denote the set of edges of the original manifold $\domain$. By a slight abuse of notation, after the truncation we continue to use $\mathcal E_{\domain}$ for the remaining segments of those edges. Thus, the new edges of $\domain'$ created by the cutting procedure are not included in $\mathcal E_{\domain}$.
	
	Let $\Gamma$ be an edge of $\domain'$ so that $\Gamma \in \mathcal E_{\domain}$. We deform the boundary conditions and the metric near $\Gamma$ so that the metric and the adjacent faces become flat (see Figure~\ref{fig:deform-near-edge}), and then excise a small neighborhood by cutting along a surface $\Sigma_\Gamma$ (see Figure~\ref{fig:cutting-and-pasting-edge}). Here $\Sigma_\Gamma$ is the product of $\Gamma$ with a circular arc of small radius. Geometrically, the cut-off region is the product $\mathbb G \times I$ of a sector $\mathbb G \subset \mathbb R^2$ and an interval $I$ (see the left-hand figure of Figure~\ref{fig:edgecut}). We show that the Fredholm index of the associated Dirac operator on this region vanishes by combining a deformation argument with the approximation result in Lemma~\ref{lemma:approximation} and Corollary~\ref{coro:approximation}.

	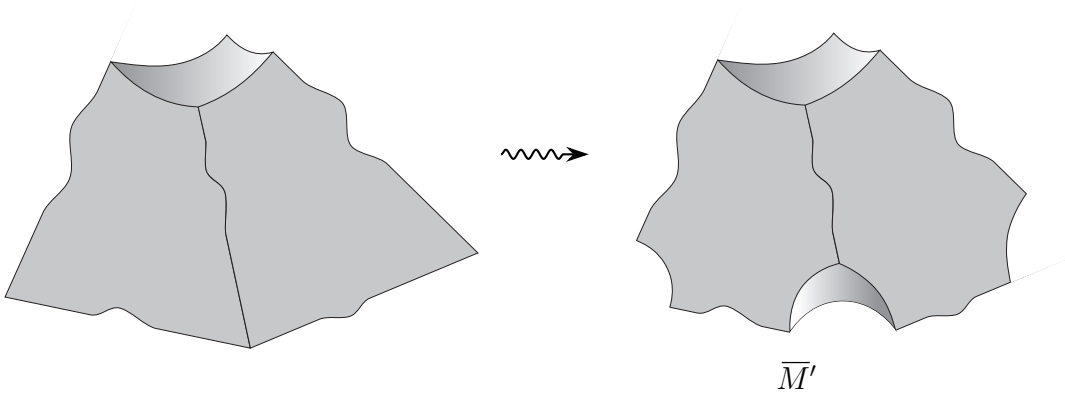
\begin{figure}[h]		
		\centering	
\begin{tikzpicture}
	\node[anchor=south west,inner sep=0] (image) at (0,0) {\includegraphics[width=1\textwidth]{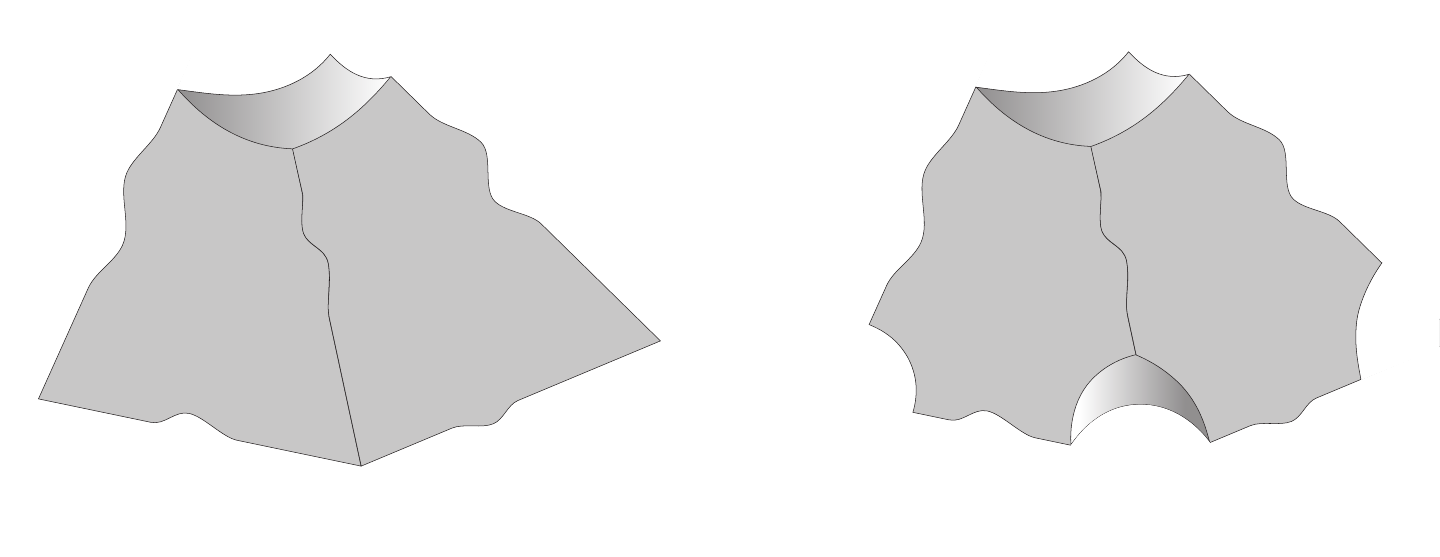}};
	\begin{scope}[x={(image.south east)},y={(image.north west)}]
		\draw[
		thick,
		decorate,
		decoration={snake, amplitude=0.6mm, segment length=2mm}
		]
		(0.48,0.62) -- (0.54,0.62);
		
		\draw[
		thick,
		-{Stealth[length=3mm,width=2mm]}
		]
		(0.54,0.62) -- (0.56,0.62);
		
		
		\node at (0.75,0.1) {$\domain'$};
	\end{scope}
\end{tikzpicture}		
		\caption{The manifold $\protect\domain'$ obtained by applying the vertex-cutting procedure at every vertex of $\protect\domain$.}
		\label{fig:cutting-and-pasting-vertices-2}
	\end{figure}

		\begin{figure}[h]		
		\centering	
		\begin{tikzpicture}
			\node[anchor=south west,inner sep=0] (image) at (0,0) {\includegraphics[width=1\textwidth]{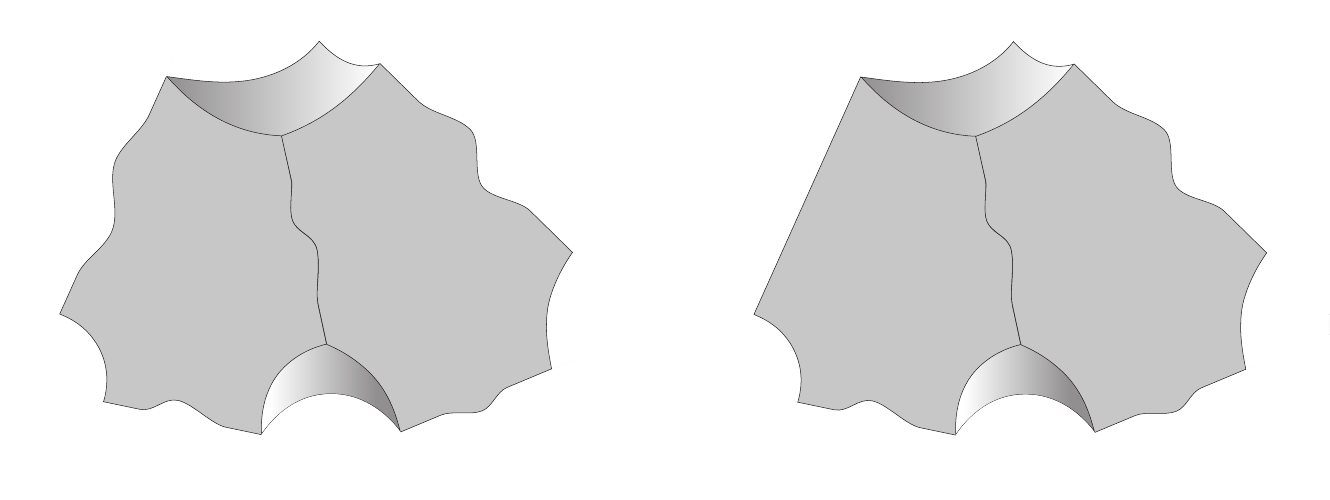}};
			\begin{scope}[x={(image.south east)},y={(image.north west)}]
				\draw[
				thick,
				decorate,
				decoration={snake, amplitude=0.6mm, segment length=2mm}
				]
				(0.47,0.54) -- (0.53,0.54);
				
				\draw[
				thick,
				-{Stealth[length=3mm,width=2mm]}
				]
				(0.53,0.54) -- (0.55,0.54);
				
				\node at (0.06, 0.59) {$\Gamma$};
				
				\node at (0.59,0.59) {$\Gamma$};
			\end{scope}
		\end{tikzpicture}		
	\caption{Deformation near $\Gamma$ that makes both the metric and the adjacent faces flat.}
		\label{fig:deform-near-edge}
	\end{figure}
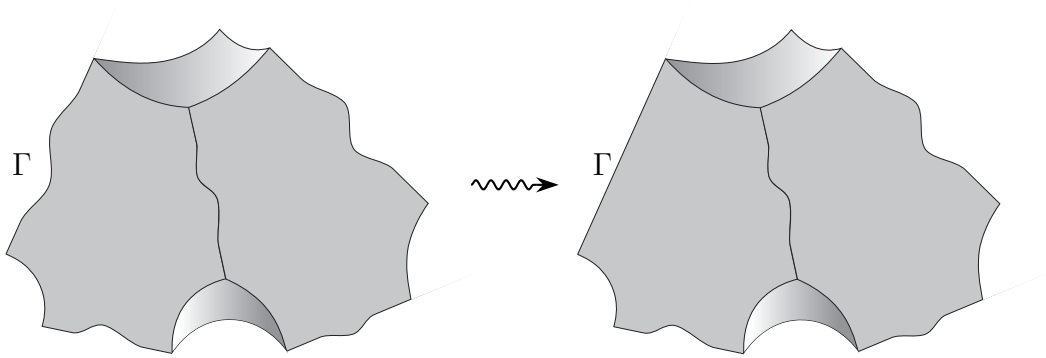
	\begin{figure}[h]		
	\centering	
	\begin{tikzpicture}
		\node[anchor=south west,inner sep=0] (image) at (0,0) {\includegraphics[width=1\textwidth]{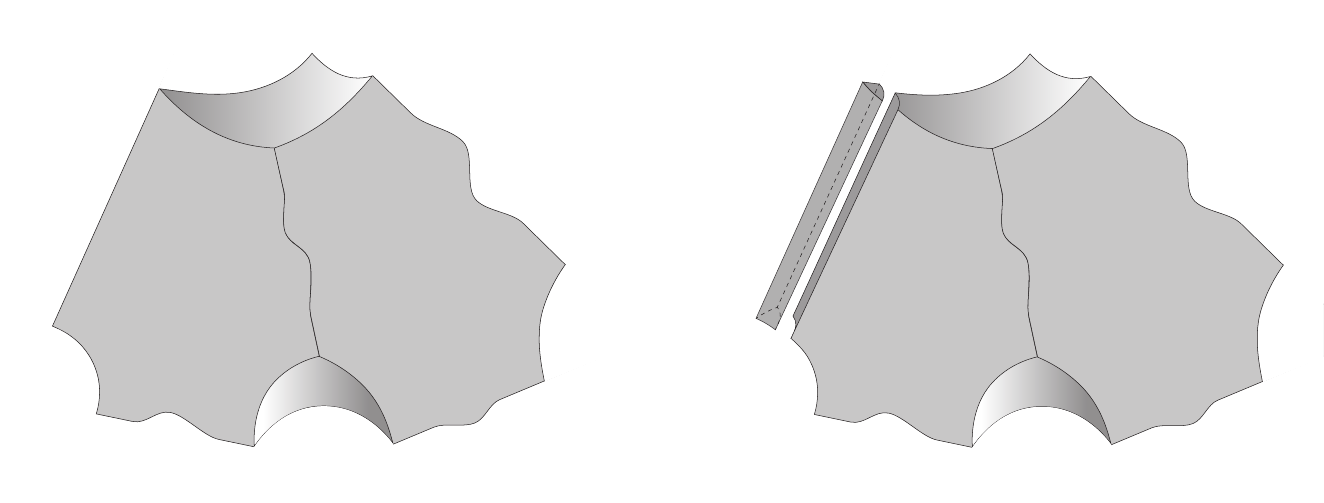}};
		\begin{scope}[x={(image.south east)},y={(image.north west)}]
			\draw[
			thick,
			decorate,
			decoration={snake, amplitude=0.6mm, segment length=2mm}
			]
			(0.47,0.5) -- (0.53,0.5);
			
			\draw[
			thick,
			-{Stealth[length=3mm,width=2mm]}
			]
			(0.53,0.5) -- (0.55,0.5);
			
%
		\end{scope}
	\end{tikzpicture}		
	\caption{Cutting off a small neighborhood of the edge $\Gamma$.}
	
	\label{fig:cutting-and-pasting-edge}
\end{figure}

	\begin{figure}[h]		
		\centering	
		\begin{tikzpicture}
			\node[anchor=south west,inner sep=0] (image) at (0,0) {\includegraphics[width=0.9\textwidth]{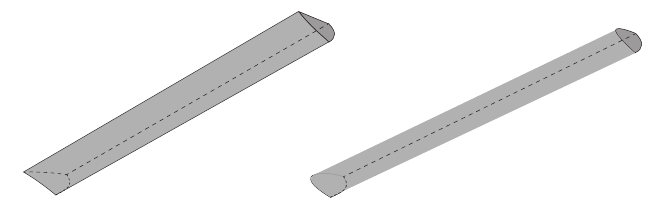}};
			\begin{scope}[x={(image.south east)},y={(image.north west)}]
				\draw[
				thick,
				decorate,
				decoration={snake, amplitude=0.6mm, segment length=2mm}
				]
				(0.47,0.5) -- (0.53,0.5);
				
				\draw[
				thick,
				-{Stealth[length=3mm,width=2mm]}
				]
				(0.53,0.5) -- (0.55,0.5);
				
				%
			\end{scope}
		\end{tikzpicture}		
	\caption{Replacing the neighborhood $\protect\mathbb G \times I$ of $\Gamma$ by $U \times I$.}
		\label{fig:edgecut}
	\end{figure}
	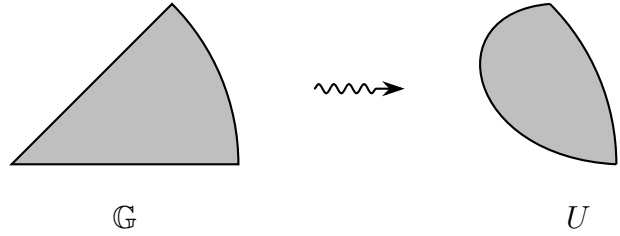
\begin{figure}[h]
		\centering
		\begin{tikzpicture}
			
			\begin{scope}
				\def\r{3}
				
				\fill[gray!50]
				(0,0) -- (\r,0)
				arc[start angle=0,end angle=45,radius=\r]
				-- cycle;
				
				\draw[thick]
				(0,0) -- (\r,0)
				arc[start angle=0,end angle=45,radius=\r]
				-- cycle;
				
				
				\node at (1.5,-0.7) {$\mathbb G$};
			\end{scope}
			
			\draw[
			thick,
			decorate,
			decoration={snake, amplitude=0.6mm, segment length=2mm}
			]
			(4,1) -- (4.9,1);
			
			\draw[
			thick,
			-{Stealth[length=3mm,width=2mm]}
			]
			(4.9,1) -- (5.2,1);
			
			\begin{scope}[xshift=5cm]
				
				\def\r{3}
				
				\fill[gray!50]
				({\r*cos(45)},{\r*sin(45)})
				arc[start angle=45,end angle=0,radius=\r]
				.. controls (1,0.1) and (0.6,2.0) ..
				({\r*cos(45)},{\r*sin(45)});
				
				\draw[thick]
				({\r*cos(45)},{\r*sin(45)})
				arc[start angle=45,end angle=0,radius=\r];
				
				\draw[thick]
				(\r,0)
				.. controls (1,0.1) and (0.6,2.0) ..
				({\r*cos(45)},{\r*sin(45)});
				
				\node at (2.5,-0.7) {$U$};

			\end{scope}
			
		\end{tikzpicture}
\caption{Replacing the sector $\mathbb G$ by the smooth convex region $U$.}
		\label{fig:GtoU}
	\end{figure}
	The boundary of the sector $\mathbb G \subset \mathbb R^2$ consists of three pieces: two radial line segments emanating from the origin and a circular arc of small radius $\varepsilon$ (see the left-hand figure of Figure~\ref{fig:GtoU}). We replace the two radial segments by a smooth convex curve that meets the circular arc orthogonally at both endpoints, thereby obtaining a convex region $U \subset \mathbb R^2$ (see the right-hand figure  of Figure~\ref{fig:GtoU}). We then form the product $U \times I$ (see the right-hand figure of Figure~\ref{fig:edgecut}). We likewise show that the Fredholm index of the associated Dirac operator on $U \times I$, equipped with the corresponding boundary condition, is zero. Gluing $U \times I$ to $\domain'$ along $\Sigma_\Gamma$ preserves the Fredholm index by Theorem~\ref{thm:gluing} (see Figure~\ref{fig:smoothedge}). Geometrically, this operation smooths the edge $\Gamma$; in particular, it simultaneously smooths the vertices at the two endpoints of $\Gamma$.

	  	\begin{figure}[h]		
	  	\centering	
	  	\begin{tikzpicture}
	  		\node[anchor=south west,inner sep=0] (image) at (0,0) {\includegraphics[width=1\textwidth]{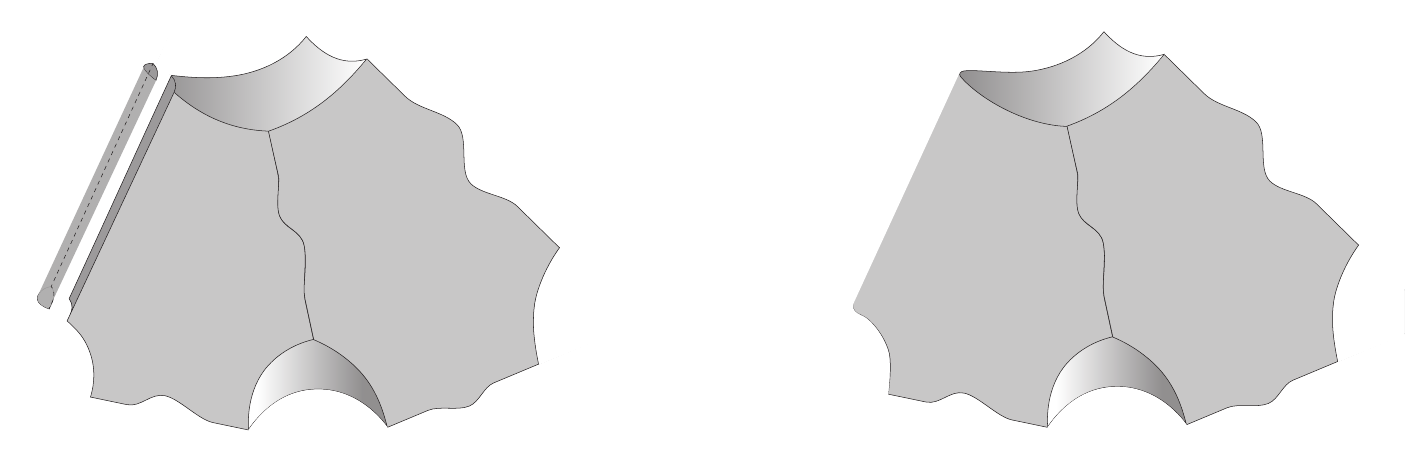}};
	  		\begin{scope}[x={(image.south east)},y={(image.north west)}]
	  			\draw[
	  			thick,
	  			decorate,
	  			decoration={snake, amplitude=0.6mm, segment length=2mm}
	  			]
	  			(0.47,0.5) -- (0.53,0.5);
	  			
	  			\draw[
	  			thick,
	  			-{Stealth[length=3mm,width=2mm]}
	  			]
	  			(0.53,0.5) -- (0.55,0.5);
	  			
	  			%
	  		\end{scope}
	  	\end{tikzpicture}		
	   	\caption{Gluing $U \times I$ to $\protect\domain'$ along $\Sigma_\Gamma$.}
	  
	  	\label{fig:smoothedge}
	  \end{figure}
	  
	  Repeating this operation for every edge in $\mathcal E_{\domain}$ produces a polyhedral manifold $\domain''$ with no vertices (see Figure~\ref{fig:smoothedge2}), although its still has codimension-two faces consists of closed smooth curves. Applying the same cutting-and-pasting construction to small neighborhoods of these curves removes the remaining codimension-two singularities and yields a manifold with smooth boundary (see Figure~\ref{fig:smooth}), while preserving the Fredholm index. Thus, the index problem on a polyhedral manifold is reduced to the classical smooth boundary case~\cite{AtiyahBott}; see also~\cite{Lottboundary}.
	
	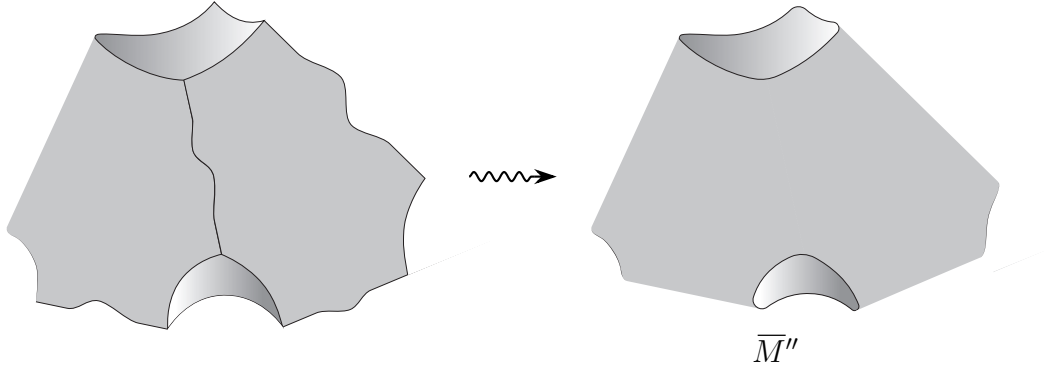
\begin{figure}[h]		
		\centering	
		\begin{tikzpicture}
			\node[anchor=south west,inner sep=0] (image) at (0,0) {\includegraphics[width=1\textwidth]{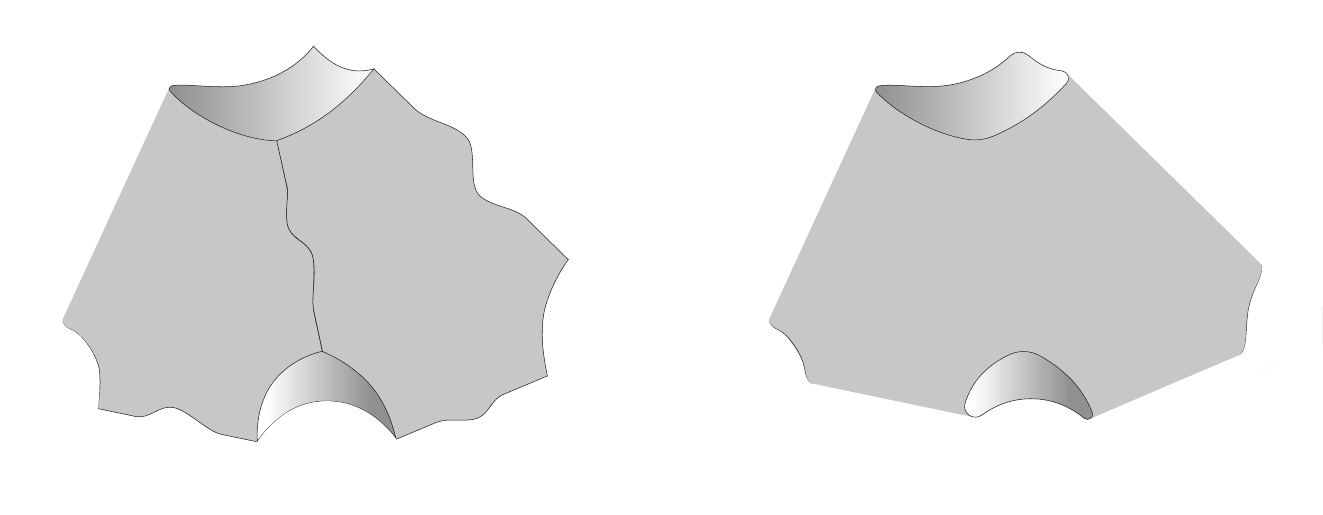}};
			\begin{scope}[x={(image.south east)},y={(image.north west)}]
				\draw[
				thick,
				decorate,
				decoration={snake, amplitude=0.6mm, segment length=2mm}
				]
				(0.47,0.5) -- (0.53,0.5);
				
				\draw[
				thick,
				-{Stealth[length=3mm,width=2mm]}
				]
				(0.53,0.5) -- (0.55,0.5);
				
				%
				
				\node at (0.75,0.1) {$\domain''$};
			\end{scope}
		\end{tikzpicture}				\caption{The manifold $\protect\domain''$ obtained after smoothing every edge in $\mathcal E_{\protect \domain}$.}
		
		\label{fig:smoothedge2}
	\end{figure}
	
	\begin{figure}
		\centering
		\includegraphics[scale=0.3]{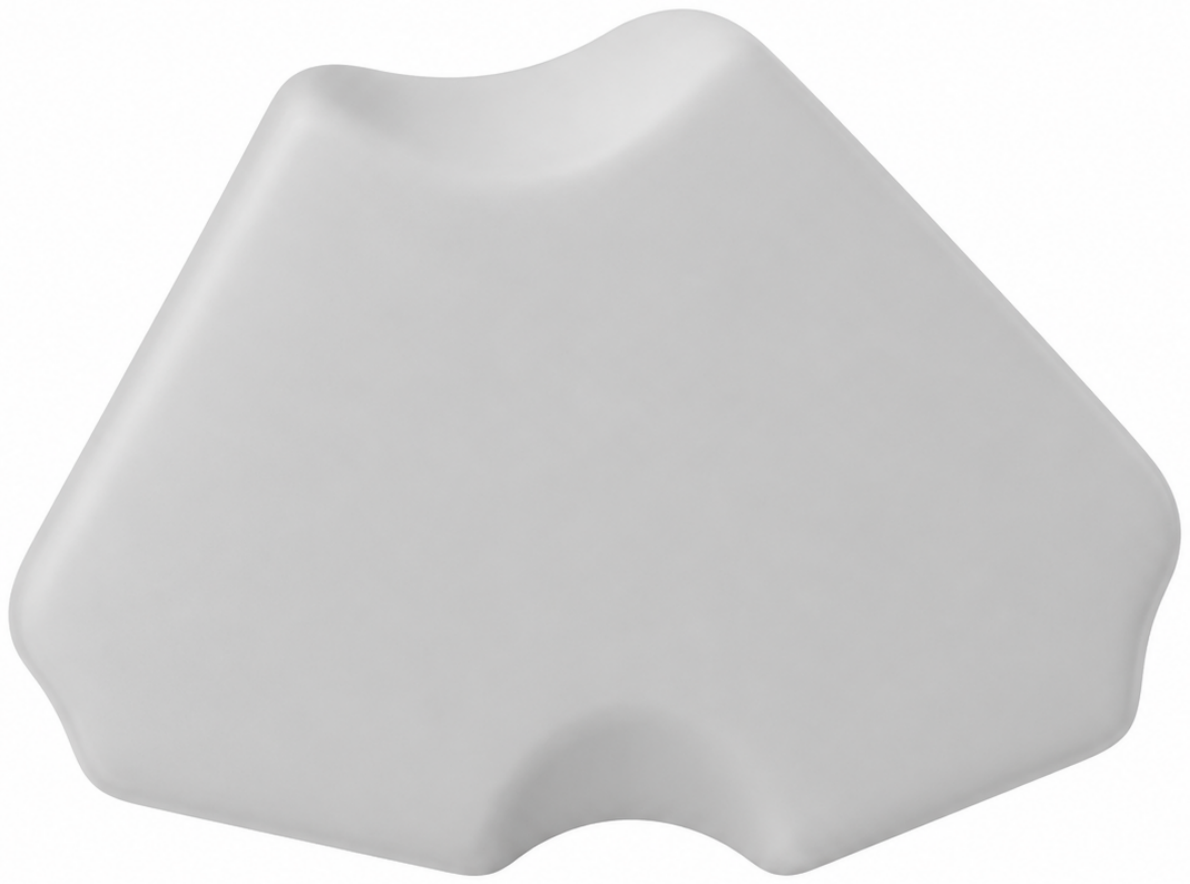}
\caption{The resulting manifold with smooth boundary.}
		\label{fig:smooth}
	\end{figure}
	
\textbf{(3).} Finally, we remove the strict comparison hypothesis. Theorem~\ref{thm:indexTheoremVector} computes the Fredholm index only under the stronger hypothesis that the comparison inequalities for the dihedral angles---or, equivalently, for the inner products of adjacent vectors---are strict (see condition~\ref{eq:strictInnerProductIndex}). It therefore does not directly establish Gromov's dihedral rigidity conjecture. To pass to the nonstrict case, we approximate the geometric data in Theorem~\ref{thm:main2} by a sequence of geometric data and associated boundary conditions satisfying the strict comparison condition~\ref{eq:strictInnerProductIndex}. For each approximation, Theorem~\ref{thm:indexTheoremVector} gives a nonzero index; hence, the corresponding Dirac operator has a nontrivial solution $\varphi_n$. The approximation lemma (Lemma~\ref{lemma:approximation}) then yields a nontrivial solution $\varphi \in H^1(\domain, E; B)$ for the Dirac operator associated with the original geometric data and boundary condition. Applying the Bochner--Lichnerowicz--Weitzenb\"ock formula to $\varphi$ gives the required equalities for the scalar and mean curvatures (see Lemma~\ref{lemma:angleRigidity}). Equality of the dihedral angles follows from a separate elementary computation using the boundary condition; see also~\cite{Wang:2022vf}.

After our papers \cite{Wang:2021tq, Wang:2022vf} appeared on arXiv, Brendle proved a case of Gromov's dihedral rigidity conjecture via a different approach under the additional assumption that all corresponding angles (not only dihedral angles) are equal \cite{MR4689374}.
Subsequently, Brendle and Wang proved another  case of Gromov’s dihedral rigidity conjecture under the  assumption that all angles are acute, that is, all  angles are $\leq \pi/2$  \cite{MR4978418}.

The authors thank  Gromov for many stimulating discussions and for suggesting that we write up a self-contained proof of the dihedral rigidity conjecture in the three-dimensional case.

\section{Geometric setup of an index problem}\label{sec:prelim}

In this section, we establish the geometric setup for our index problem. This includes defining the relevant vector bundles, elliptic operators, and boundary conditions, as well as clarifying their relationship to Theorem \ref{thm:main2}. Our focus will primarily be on the three-dimensional case.

\subsection{Polyhedral manifolds and twisted Dirac operators}\label{sec:polymanifold}
In this subsection, we describe the basic setup of the index problem and provide some key geometric estimates.

The geometric objects investigated in this paper are manifolds with polyhedral boundaries, which we will henceforth refer to as \emph{polyhedral manifolds}. As the current paper  is concerned with the three-dimensional case of Gromov's dihedral rigidity conjecture, we shall focus on the  dimension three case for simplicity.

First, we introduce the following notion of polyhedral cones. A one-dimensional polyhedral cone in $\mathbb R^1$ is simply the ray $[0, \infty)\subset \R^1$. 

\begin{definition}
A two dimensional polyhedral cone $C$ in $\mathbb R^2$ is a cone (i.e., for any $v \in C$ and $\lambda \geq 0$, $\lambda v \in C$) such that its link, which is defined to be $C\cap \mathbb S^1$, is a connected arc. 
\end{definition} 

To define three-dimensional polyhedral cones in $\mathbb R^3$, we first recall the definition of spherical polygons. 

\begin{definition}
Let $\mathbb S^2$ be the unit sphere in $\mathbb R^3$. A spherical polygon $P \subset \mathbb{S}^2$ is a two-dimensional submanifold with boundary such that its boundary $\partial P$ is piecewise smooth and consists of a finite union of geodesic arcs.
\end{definition}

\begin{definition}\label{def:polyhedralcone}
 A three-dimensional polyhedral cone $C$ in $\mathbb{R}^3$ is a cone (i.e., for any $v\in C$ and $\lambda\ge 0$, we have $\lambda v\in C$) whose link, defined by $C\cap \mathbb{S}^2$, is a connected two-dimensional spherical polygon. In this case, the codimension-one faces of $C$ are precisely the intersections of $C$ with its bounding planes.
\end{definition}

Let $(X, h)$ be an open, oriented, $3$-dimensional Riemannian manifold and let $M \subset X$ be an $3$-dimensional topological manifold with boundary.  

\begin{definition}\label{def:tangentcone}
	The tangent cone $\tancone_x M \subset T_x X$ at a point $x \in M$ is defined as the set of vectors $v \in T_x X$ such that there exists a smooth curve $\gamma: [0, \epsilon) \to M$ with $\gamma(0)=x$ and $\dot{\gamma}(0)=v$.
\end{definition} 

We adopt the following definition of polyhedral manifolds.

\begin{definition}\label{def:polymanifold}
	Let $(X, h)$ be an open, oriented, $3$-dimensional Riemannian manifold and let $M \subset X$ be a compact region. We say that $(M, g)$, with $g \coloneqq h|_M$, is a \emph{polyhedral manifold} if $M$ is a $3$-dimensional topological manifold with boundary and there exists a finite collection of pairwise transversal smooth open hypersurfaces $\mathcal{H} = \{H_1, \dots, H_p\}$ in $X$ such that:
	\begin{enumerate}
		\item The boundary $\partial M$ is contained in the union of these hypersurfaces:
		\[ \partial M \subset \bigcup_{i=1}^p H_i \]
		\item For any point $x \in \partial M$, the tangent cone $\tancone_x M \subset T_x X$ is a three dimensional polyhedral cone (identifying $T_x X$ with $\mathbb{R}^3$) in the sense of Definition \ref{def:polyhedralcone}. 
		\item Furthermore, the set of bounding planes of the polyhedral cone $\tancone_x M$ is exactly the tangent planes $\{T_x H_j \mid H_j \in \mathcal{H} \text{ and } x \in H_j\}$.
	\end{enumerate}
\end{definition}

We say a function on $M$ is \emph{smooth} on $M$ if it is the restriction to $M$ of some smooth function on $X$. Similarly, smooth vector bundles on $M$ are restrictions to $M$ of smooth vector bundles on $X$, and smooth sections of a smooth vector bundle over $M$ are restrictions of smooth sections of smooth vector bundles on $X$.

Let $M$ be a $3$-dimensional polyhedral manifold. Let $\mathcal{S}_k$ be the set of points $x \in M$ for which $k$ is the largest integer such that the tangent cone $\tancone_x M$ splits as $T_x M \cong \mathbb{R}^k \times C'_x$, where $C'_x$ is a polyhedral cone in $\mathbb{R}^{3-k}$.  The definition of a polyhedral manifold ensures that $\mathcal{S}_k$ is an open manifold of dimension $k$.


\begin{definition}\label{def:codimface}
	Let $\mathring{F}$ be a connected component of $\mathcal{S}_k$, and let $F$ denote its closure in $M$. We call $F$ a codimension $(3-k)$ face (or equivalently, a dimension $k$ face) of $M$. 
\end{definition}

This collection of faces, ordered by inclusion, endows $M$ with a natural stratification.

\begin{definition}
	Let $F$ be a codimension $(3-k)$ face of $M$, as introduced in Definition \ref{def:codimface}. For each $x\in \mathring{F}$, we define the normal cone $\normalcone_xF$ of $F$ at $x$ to be 
	\[  \normalcone_xF = \tancone_xM \cap (T_xF)^\perp,  \]
	where $(T_x F)^\perp$ denotes the orthogonal complement of the tangent space $T_x F$ within the ambient tangent space $T_x X$.
\end{definition}

We now discuss the types of maps between three dimensional polyhedral manifolds that are relevant for this paper. 

\begin{definition}\label{def:polyhedralmap}
	Let $(\domain,\overbar g)$ and $(\target,g)$ be three dimensional polyhedral manifolds. A  map $f\colon \domain \to \target$ is called a polyhedral map, if 
	\begin{enumerate}
		\item $f$ is Lipschitz; 
		\item $f$ is smooth away from the vertices;
		\item $f$ maps  each codimension $k$ face $\overbar F^k_i$ in $\domain$  to a codimension $k$ face  $F^k_{\varphi(i)}$ of $\target$, where $\varphi$ is a map  from the collection of faces of $\domain$ to the collection of faces of $\target$ such that each pair of adjacent\footnote{A pair of codimension one faces of $\domain $ are called adjacent if their intersection is a nonempty codimension two face of $\domain$. Since $\domain$ has dimension three, this simply means a pair of faces of $\domain$ are   adjacent if their intersection is an edge of $\domain$.} codimension one faces of $\domain $  are mapped to a pair of adjacent codimension one faces of $\target$.
    \end{enumerate} 
\end{definition}
\begin{remark}
	Condition (3) requires the existence of a specific assignment $i \mapsto \varphi(i)$ that maps each face of $\domain$ to a face of $\target$ of the same codimension. This data is crucial for the subsequent analysis of dihedral angle comparisons.
	
	It is important to note that $f$ is allowed to be ``degenerate'' on some faces, meaning the image $f(\overline{F}^k_i)$ may lie in a subset of $F^k_{\varphi(i)}$ with strictly lower dimension. For instance, if $f$ maps a 2-dimensional face $\overline{F}$ of $\domain$ to a 2-dimensional face $F$ of $\target$, the image $f(\overline{F})$ might be contained entirely within an edge $e \subset F$, or even a vertex. In such a case, the  image $f(\overline{F})$ is contained in more than one 2-dimensional faces of $\target$ (i.e. those $2$-dimensional  faces containing  $e$). Since the dihedral angles are defined between two specific adjacent faces, the  map $\varphi$ resolves this ambiguity by explicitly specifying which target faces are being compared.
\end{remark}

We emphasize that a polyhedral map $f\colon \domain \to \target$ between $3$-dimensional polyhedral manifolds is not required to be smooth at the vertices. For example, for two polyhedra $P_1$ and $P_2$ in $\mathbb{R}^3$ of the same combinatorial type, one can  construct a polyhedral map from $P_1$ to $P_2$ by defining a smooth $f$ on the interior, the $2$-dimensional faces, and the edges, but excluding small neighborhoods of the vertices, and then extending $f$ radially toward the vertices within these neighborhoods. If there are four or more 2-dimensional faces meeting at a vertex, such a map generally is not smooth at that vertex. 

\begin{definition}
	Let $f\colon (\domain,\overbar g)\to (\target,g)$ be a polyhedral map between  Riemannian polyhedral manifolds of dimension three. Denote by $T\domain$ and $T\target$ the  tangent bundles of $\domain$ and $\target$, respectively. The map $f$  is called spin if the bundle $T\domain\oplus f^\ast T\target$ admits a spin structure.  
\end{definition}

Now suppose $f\colon (\domain, \overbar g)\to (\target,g)$ is a spin polyhedral map between  Riemannian polyhedral manifolds of dimension three. We consider the spinor bundle associated with the vector bundle $T\domain\oplus T\target$: 
\[ E\coloneqq S(T\domain\oplus f^*T\target) \]  equipped with the spinorial connection $\nabla$  induced by the Levi-Civita connections $\nabla^{\domain}$ and $\nabla^{\target}$ on $T\domain$ and $T\target$, respectively. Denote by $\overbar c$ the Clifford action of $T\domain$ on $E$ and $\hat c$ the Clifford action of $f^\ast T\target$ of $E$. Since $T\domain\oplus f^\ast T\target$ has even rank,  there exists a natural $\mathbb Z/2$ grading operator $\mathscr E = i^{3} \overbar c(\overbar e_1) \overbar c(\overbar e_2)  \overbar c(\overbar e_3) \hat c(e_1) \hat c(e_2)  \hat c(e_3)$ on $E$, where $\{\overbar e_i\}$ and $\{e_i\}$ are oriented local orthonormal basis of $T\domain$ and $f^\ast T\target$. Let us define 
\[ c\coloneqq i\mathscr E\hat c. \]
Note that $\overbar c(\overbar u)$ and $c(u)$ commute with each other for all $\overbar u\in T\domain$ and  $u\in f^\ast T\target$. We extend these to Clifford actions of $2$-forms by setting 
$$\overbar c(\overbar u\wedge \overbar v)=\overbar c(\overbar u)\overbar c(\overbar v) ~ \mathrm{ and } ~ c(u\wedge v)=c(u)c(v),$$
whenever $\overbar u\perp \overbar v$ and $u\perp v$,  for $\overbar u,\overbar v\in T\domain$ and $u, v\in T\target$.

\begin{definition}
We define 
$C^\infty_{00}(\domain,E)$ to be the space  of smooth sections of $E = S(T\domain\oplus f^*T\target)$ over $\domain$ that vanish near all codimension two faces.
\end{definition} 

Consider the Dirac operator
$$D=\sum_{i=1}^n \overbar c(\overbar e_i)\nabla_{\overbar e_i}$$
acting on $C^\infty_{00}(\domain,E)$. For any $\sigma\in C^\infty_{00}(\domain,E)$, the Stokes formula reads:
\begin{equation}\label{eq:Stokes1}
		\int_{\domain}|D\sigma|^2=\int_{\domain}|\nabla\sigma|^2+\int_{\domain}\langle\mathscr R\sigma,\sigma\rangle\\
		+\sum_{\overbar F_k}\int_{\overbar F_k}\langle D^\partial\sigma,\sigma\rangle,
\end{equation}
where $\{\overbar F_k\}$ is the collection of codimension one faces of $\domain$, and the curvature term $\mathscr R$ is given by the Lichnerowicz formula (see for example \cite[Theorem II.8.17]{spingeometry})
\begin{equation}\label{eq:lichnerowicz}
	\mathscr R=\frac{\Sc_{\overbar g}}{4}-\frac 1 4\sum_{i\ne j} \overbar c(\overbar e_i)\overbar c(\overbar e_j)c(R^{g}(f_*\overbar e_i\wedge f_*\overbar e_j)),
\end{equation}
with $\Sc_{\overbar g}$ the scalar curvature of $\overbar g$ and  $R^{g}$ the curvature operator\footnote{Our convention is that the curvature operator of the unit round sphere is the identity operator.} acting on $\Bigwedge^2T\target$. The boundary operator is given as
\begin{equation}
	D^\partial=\sum_{\mu=1}^{n-1}\overbar c(\overbar \nu_k)\overbar c(\overbar e_\mu)\nabla_{\overbar e_\mu},
\end{equation}
where $\overbar \nu_k$ is the unit inner normal vector of $\overbar F_k$, and $\{\overbar e_\mu\}$ is a local orthonormal frame of $T\overbar F_k$.

We have the following lemma estimating the curvature term $\mathscr R$; see \cite[Section 1.1]{GoetteSemmelmann}. 
\begin{lemma}\label{lemma:curvature>=}
	If the curvature operator of $(\target,g)$ is non-negative, then
	\begin{equation}\label{eq:curvature>=}
		\mathscr R\geq\frac{\Sc_{\overbar g}}{4}-\|\wedge^2f_*\|\frac{f^*\Sc_{g}}{4}.
	\end{equation}
\end{lemma}
\begin{proof}
	At any point $x\in \domain$, consider the singular value decomposition of the map
	$$R^{g}\circ \wedge^2f_*\colon \Bigwedge^2T_x\domain\to \Bigwedge^2 T_{f(x)}\target,$$
	that is,   
	$$R^{g}\circ \wedge^2f_*(\alpha_i)=\lambda_i\beta_i,~\forall 1\leq i\leq n(n-1)/2,$$
	for some orthonormal bases $\{\alpha_i\}$ of $\Bigwedge^2T_x\domain$ and $\{\beta_i\}$ of $\Bigwedge^2T_{f(x)}\target$, and $\lambda_i\geq 0$.
	Then
	\begin{equation}
		\begin{split}
			&-\frac 1 4\sum_{i\ne j} \overbar c(\overbar e_i)\overbar c(\overbar e_j)c(R^{g}(f_*\overbar e_i\wedge f_*\overbar e_j))=-\frac 1 2\sum_i \overbar c(\alpha_i) c(\beta_i)\lambda_i\\
			\geq& -\frac 1 2\sum_i\lambda_i=-\frac 1 2\|R^{g}\circ \wedge^2f_*\|_1\\
			\geq& -\frac 1 2\|R^{g}\|_1\cdot \|\wedge^2f_*\|=-\frac 1 2 \|\wedge^2f_*\|\cdot \tr(R^{g})=-\|\wedge^2f_*\|\frac{f^*\Sc_{g}}{4}.
		\end{split}
	\end{equation}
\end{proof}

\subsection{Local boundary conditions}

In this subsection, we describe a local boundary condition for sections of $E = S(T\domain\oplus f^\ast T\target)$ over $\domain$ and discuss the relationship between the associated boundary value problem and Theorem \ref{thm:main2}.

\begin{definition}\label{def:boundaryCondition} Let $B$ denote the following boundary condition for sections of $E$ over $\domain$: a smooth section $\sigma\in C^\infty_{00}(\domain,E)$ is said to satisfy the boundary condition if, on every codimension-one face $\overbar F_k$ of $\domain$,
	\[ \mathscr E\overbar c(\overbar \nu_k)c(f^*\nu_k)\sigma = -\sigma, \]
where $\overbar \nu_k$ is the unit inner normal vector of $\overbar F_k$ and $\nu_k$ is the unit inner normal vector of its corresponding\footnote{See Definition \ref{def:polyhedralmap} for the precise definition of the correspondence between codimension one faces of $\domain$ and $\target$ under the map $f$.} codimension one face $F_k$ of $\target$. The space of all such sections is denoted by $C^\infty_{00}(\domain,E;B)$. For convenience, we define the operator \begin{equation}\label{eq:involution}
\gamma_k \coloneqq \mathscr E\overbar c(\overbar \nu_k)c(f^*\nu_k).
\end{equation} When no ambiguity arises, we will omit $f^*$ from the notation and simply write 
\[ \gamma_k = \mathscr E\overbar c(\overbar \nu_k)c(\nu_k). \]
\end{definition}

It is straightforward to verify that the Dirac operator $D$ is symmetric with respect to the boundary condition $B$, that is, for all smooth sections $\sigma, \tau \in C^\infty_{00}(\domain, E; B)$, the operator $D$ satisfies:
$$\int_{\domain} \langle D\sigma, \tau \rangle = \int_{\domain}\langle \sigma, D\tau \rangle. $$

\begin{example}\label{example:deRham}
	Consider the special case where  $\domain = \target$ and $f$ is the identity map. In this case, the bundle  $E = S(T\domain \oplus T\domain)$ is naturally identified with the exterior algebra bundle $\Lambda^* T\domain$. Under this identification, the Dirac operator $D$ on $E$ becomes the de Rham operator of $\domain$. The restriction of a section $\sigma \in C^\infty_{00}(\domain, \Lambda^* T\domain)$ to a codimension one face $\overbar F_k$ admits a decomposition $\sigma = \alpha + \nu_k^\ast \wedge \beta$, where $\nu_k^\ast$ is the 1-form dual to the unit normal vector $\nu_k$, and $\alpha, \beta$ are tangential forms (i.e., their contraction with $\nu_k$ vanishes). In this special case, the section $\sigma$ satisfies the boundary condition $B$ on $\overbar F_k$ if and only if $\beta = 0$ on $\overbar F_k$. This is referred to as the \emph{absolute boundary condition} for the de Rham operator in the literature.
\end{example}

Let $\gamma_k$ be the  self-adjoint involution from line \eqref{eq:involution}.  The  super-commutator
$$\mathscr A\coloneqq -\frac 1 2(D^\partial\gamma_k+\gamma_kD^\partial)$$
defines a bundle endomorphism of $E$ over $\overbar F_k$. Therefore, if $\sigma\in C^\infty_{00}(\domain,E;B)$, then it follows from  $\gamma_k\sigma = -\sigma$ that line \eqref{eq:Stokes1} becomes
\begin{equation}\label{eq:Stokes2}
	\int_{\domain}|D\sigma|^2=\int_{\domain}|\nabla\sigma|^2+ \int_{\domain}\langle\mathscr R\sigma,\sigma\rangle\\
	+\sum_{\overbar F_k}\int_{\overbar F_k}\langle \mathscr A\sigma,\sigma\rangle,
\end{equation}
A direct computation yields
\begin{equation}
	\begin{split}
		\mathscr A=&\frac{H_{\overbar g}}{2}+\frac 1 2 \sum_{\mu=1}^{n-1}\overbar c(\overbar \nu_k)\overbar c(\overbar e_\mu) c(\nu_k)c(\nabla^{\target}_{f_*\overbar e_\mu}\nu_k)\\
		=&\frac{H_{\overbar g}}{2}-\frac 1 2 \sum_{\mu=1}^{n-1}\overbar c(\overbar \nu_k)\overbar c(\overbar e_\mu) c(\nu_k)c(A^{g}f_*\overbar e_\mu),
	\end{split}
\end{equation}
where $H_{\overbar g}$ is the mean curvature of $\overbar F_k$ and $A^{g}$ denotes the second fundamental form\footnote{Our convention is that the second fundamental form of the unit sphere as the boundary of the unit ball is the identity operator.} of $F_k$.
We now estimate the endomorphism $\mathscr A$ (cf. \cite[Lemma 2.1]{Lottboundary}).
\begin{lemma}\label{lemma:secondff>=}
	If the second fundamental form $A^{g}$ of each codimension one face $F_k\subset \target$ is non-negative, then
	\begin{equation}\label{eq:secondff>=}
		\mathscr A\geq \frac{H_{\overbar g}}{2}-\|f_*^\partial\|\cdot\frac{f^*H_{g}}{2},
	\end{equation}
	where $f^\partial$ denotes the restriction of $f$ on $\partial \domain$.
\end{lemma}
\begin{proof}
	The argument is similar to the proof of Lemma \ref{lemma:curvature>=}. 
	For any point $x\in \overbar F_k$, consider the singular value decomposition of the map
	$$A^{g}\circ f^\partial_*\colon T_p\overbar F_k\to T_{f(p)}F_k,$$
that is,  
	$$A^{g}\circ f^\partial_*(\alpha_\mu)=\lambda_\mu\beta_\mu,~\forall 1\leq \mu\leq n-1, $$
for some  orthonormal bases $\{\alpha_\mu\}$ of $T_p\overbar F_k$ and $\{\beta_\mu\}$ of $T_{f(p)}F_k$, and $\lambda_\mu\geq 0$. 
	Then we have 
	\begin{equation}
		\begin{split}
			-\frac 1 2 \sum_{\mu=1}^{n-1}\overbar c(\overbar \nu_k)\overbar c(\overbar e_\mu) c(\nu_k)c(A^{g}f_*\overbar e_\mu)= &-\frac 1 2\sum_{\mu=1}^{n-1}\overbar c(\overbar \nu_k)\overbar c(\alpha_\mu) c(\nu_k)c(\beta_\mu)\lambda_\mu\\
			\geq& -\frac 1 2\sum_{\mu=1}^{n-1}\lambda_\mu=-\frac 1 2\|A^{g}\circ f^\partial_*\|_1\\
			\geq& -\frac 1 2\|A^{g}\|_1\cdot \|f^\partial_*\|\\
			=&-\frac 1 2 \|f^\partial_*\|\cdot \tr(A^{g})
			= -\|f^\partial_*\|\frac{f^*H_{g}}{2}.
		\end{split}
	\end{equation}
\end{proof}

Since $f\colon \domain\to \target$ is a Lipschitz map, the pullback bundle $f^*T\target$ is a Lipschitz vector bundle. Consequently, the Sobolev space of sections $H^1(\domain, E)$ is well-defined. By  \eqref{eq:Stokes2} and the Sobolev trace theorem, there exist constants $C_1, C_2 > 0$ such that$$ C_1 \|\sigma\|_1 \le \|\sigma\|_D \le C_2 \|\sigma\|_1 $$for all $\sigma \in C^\infty_{00}(\domain, E; B)$, where $\|\cdot\|_1$ denotes the standard $H^1$-norm \[ \|\sigma\|_1^2=\|\sigma\|^2+\|\nabla\sigma\|^2\]and $\|\cdot\|_D$ denotes the graph norm $\|\sigma\|_D^2 = \|\sigma\|^2 + \|D\sigma\|^2$. Thus, the two norms are equivalent on this subspace. 

We define $H^1(\domain, E; B)$ to be the closure of $C^\infty_{00}(\domain, E; B)$ with respect to the $H^1$-norm. It is a standard fact that subsets of codimension two have vanishing $H^1$-capacity; therefore,  $H^1(\domain, E; B)$ coincides with the closed subspace of $H^1(\domain, E)$ consisting of sections whose traces on the codimension-one faces satisfy the boundary condition $B$. In particular, the identity \eqref{eq:Stokes2} extends  to all $\sigma \in H^1(\domain, E; B)$.

To summarize the preceding discussion, we obtain the following proposition.

\begin{proposition}\label{prop:smooth>=}
	With the notation as above, assume that $\target$ has a non-negative curvature operator and that each codimension-one face of $\target$ has a non-negative second fundamental form. Then, for any $\sigma\in C^\infty_{00}(\domain,E;B)$, the following inequality holds:
	\begin{equation}\label{eq:Stokes3}
	\begin{split}
		\int_{\domain}|D\sigma|^2\geq &\int_{\domain}|\nabla\sigma|^2+ \int_{\domain}\frac{\Sc_{\overbar g}-\|\wedge^2f_*\|\cdot f^*\Sc_{g}}{4}|\sigma|^2\\
		&+\sum_{\overbar F_k}\int_{\overbar F_k}\frac{H_{\overbar g}-\|f^\partial_*\|\cdot f^*H_{g}}{2}|\sigma|^2,
	\end{split}
	\end{equation}
where the sum runs over all codimension-one faces $\overbar F_k$ of $\domain$. Moreover, the inequality \eqref{eq:Stokes3} extends to any $\sigma\in H^1(\domain,E;B)$.
\end{proposition}

The following lemma shows how the boundary value problem for the Dirac operator $D$ is applied to resolve Gromov's dihedral rigidity conjecture in dimension three.
\begin{lemma}\label{lemma:angleRigidity}
	Suppose  $(\target,g)$ is a three dimensional convex Euclidean polyhedron, where $g$ is the Euclidean metric. Let $(\domain, \overbar g)$ be  a  three dimensional connected polyhedral manifold such that 
	$\Sc_{\overbar g}\geq0$ and $H_{\overbar g}\geq 0.$ Let $f\colon (\domain,\overbar g)\to (\target,g)$ be a spin polyhedral map and $E = S(T\domain \oplus f^\ast T\target)$. 
	If
	\begin{itemize}
		\item there exists a non-zero $\sigma\in H^1(\domain,E;B)$ satisfying $D\sigma=0$,
	\end{itemize}
	then $\Sc_{\overbar g}=0$, $H_{\overbar g}=0$, and $\theta_{\overbar g}=f^*\theta_g$. Moreover, $(\domain ,\overbar g)$ is flat.
\end{lemma}
\begin{proof}
Let $\sigma \in H^1(\domain, E; B)$ be a non-zero section satisfying $D\sigma = 0$. By hypothesis, the scalar curvature $\Sc_{\overbar g}$ and the mean curvature $H_{\overbar g}$ are non-negative. Applying the inequality \eqref{eq:Stokes3} from Proposition \ref{prop:smooth>=}, the vanishing of $D\sigma$ combined with the non-negativity of the geometric terms forces the inequality to become an equality. In particular, we deduce that $\nabla\sigma = 0$; that is, $\sigma$ is a parallel section. On a connected manifold, a non-zero parallel section is nowhere vanishing.\footnote{In the current setting, since $(\target,g)$ is Euclidean, the bundle $T\target$ is a trivial bundle equipped with a flat connection. Consequently, $E = S(T\domain \oplus f^*T\target) \cong S(T\domain\oplus \underline{\mathbb{R}}^3)$ is a smooth vector bundle over $\domain$, where $\underline{\mathbb{R}}^3=f^*TM$ is a trivial flat bundle. Standard regularity theory implies that a parallel section is smooth.} Consequently, the inequality  \eqref{eq:Stokes3} implies that $\Sc_{\overbar g}=0$ and $H_{\overbar g}=0$.

Since $\sigma$ is parallel, the Ricci identity for spinors implies that$$ 0 = \sum_{j=1}^3 \overbar c(\overbar e_j)\left(\nabla_{\overbar e_i}\nabla_{\overbar e_j} - \nabla_{\overbar e_j}\nabla_{\overbar e_i} - \nabla_{[\overbar e_i,\overbar e_j]}\right)\sigma = -\frac{1}{2} \overbar c(\textup{Ric}_{\overbar g}(\overbar e_i))\sigma $$for each vector $\overbar e_i$, where $\{\overbar e_j\}$ is a local orthonormal frame of $(\domain,\overbar g)$. This implies that $(\domain,\overbar g)$ is Ricci-flat. Since $\dim \domain= 3$, the Riemann curvature tensor is determined by the Ricci tensor; thus, $(\domain, \overbar g)$ is flat.
	
	Finally, consider two adjacent codimension one faces $\overbar F_i$ and $\overbar F_j$, and suppose $x\in \overbar F_i\cap \overbar F_j$. Since $\sigma$ satisfies the boundary condition $B$ (see Definition \ref{def:boundaryCondition}) on each face, we have:
	\[ \mathscr E c(\nu_{i, f(x)})\sigma_x = -\overbar c(\overbar \nu_{i,x})\sigma_x \quad \text{and} \quad \mathscr E c(\nu_{j, f(x)})\sigma_x = -\overbar c(\overbar \nu_{j, x})\sigma_x. \]
	where $\sigma_x$ is the value of $\sigma$ at $x$ and $\overbar \nu_{i,x}$ is the value of $\overbar \nu_i$ at $x$. 
	By linearity, for any $a,b \in \mathbb{R}$,
	\[ \mathscr E c\big(a\nu_{i, f(x)} + b\nu_{j, f(x)}\big)\sigma_x = -\overbar c\big(a\overbar \nu_{i, x} + b\overbar \nu_{j, x}\big)\sigma_x. \]
     Taking the squared norm of both sides and dividing by $|\sigma_x|^2$ (which is non-zero), we obtain
	\[ |a\nu_{i, f(x)}+ b\nu_{j, f(x)}|^2 = |a\overbar \nu_{i, x} + b\overbar \nu_{j, x}|^2. \]
	Expanding the squares and using the fact that the normal vectors are chosen to have unit length, we conclude that
	\[ \langle \overbar \nu_{i, x}, \overbar \nu_{j, x} \rangle = \langle \nu_{i, f(x)}, \nu_{j, f(x)} \rangle. \]
	It follows that the dihedral angles satisfy $\theta_{\overbar g}(x) = \theta_g(f(x))$.
\end{proof}

\section{Essential self-adjointness of twisted Dirac  operators on polyhedral manifolds}\label{sec:essentiallySefladjoint}
In this section, we show that the twisted Dirac operator $D$ associated with the bundle $E = S(T\domain\oplus f^*T\target)$ (introduced in Section \ref{sec:prelim}) is essentially self-adjoint subject to the boundary condition $B$ described in Definition \ref{def:boundaryCondition}, under suitable assumptions on the comparison of dihedral angles. As a consequence, we conclude that such an operator $D_B$ is Fredholm. 

We begin by recalling the notion of essential self-adjointness for unbounded symmetric operators.

\begin{definition}\label{def:ess-sa}Let $D$ be the Dirac operator with initial domain $C^\infty_{00}(\domain,E;B)$ as in Section \ref{sec:prelim}. The minimal domain of $D$ is the closure of $C^\infty_{00}(\domain,E;B)$ with respect to the Sobolev  norm $\|\cdot\|_1$, which is $H^1(\domain,E;B)$.  
	We define the maximal domain $\mathcal{D}_{\max}(D)$ of $D$ as the space of sections $\xi\in L^2(\domain,E)$ such that the linear functional
	\[ \sigma \in C^\infty_{00}(\domain,E;B) \mapsto \int_{\domain}\langle \xi, D\sigma \rangle \in \mathbb{C} \]
	extends to a bounded linear functional on $L^2(\domain,E)$.
	
	We say that $D$ is \emph{essentially self-adjoint} if the minimal and maximal domains coincide, i.e., $\mathcal{D}_{\max}(D) = H^1(\domain,E;B)$.

\end{definition}

If a section lies in $\mathcal{D}_{\max}(D)$ but not in $H^1(\domain,E;B)$, it must fail to be an $H^1$-section near some point of $\domain$.  Therefore, by a standard partition-of-unity argument, the question of essential self-adjointness can be studied locally in arbitrarily small neighborhoods of each point.

\begin{definition}\label{def:ess-saLocal}
	Let $D$ be the twisted Dirac operator associated with the bundle $E = S(T\domain\oplus f^*T\target)$, subject to the boundary condition $B$, as above. We say that $D$ is \emph{locally essentially self-adjoint} at a point $x \in \domain$ if there exists an open neighborhood $U$ of $x$  such that every section $\xi \in \mathcal{D}_{\max}(D)$ with support contained in $U$ belongs to $H^1(\domain, E; B)$.
\end{definition}
	
Consequently, $D$ is (globally) essentially self-adjoint if and only if it is locally essentially self-adjoint at every point $x \in \domain$.

\subsection{Reduction to standard model cases}
In this subsection, we discuss how to reduce the verification of the essential self-adjointness of $D_B$ to certain standard model cases.

We begin by describing two typical geometric situations that will play a central role in the  reduction process.

\begin{example}\label{ex:codimtwo}
	Let $\Omega \subset \mathbb{R}^3$ be a region whose boundary, near the origin, consists of two smooth surfaces $\Sigma_1$ and $\Sigma_2$ meeting transversely along a curve $\Gamma$. Assume that $\Gamma$ passes through the origin and is tangent to the $x$-axis at the origin. Let $\theta \in (0,\pi)$ denote the  dihedral angle between $\Sigma_1$ and $\Sigma_2$ at the origin. 
	
	Suppose the tangent planes  $T_0 \Sigma_1$ and $T_0 \Sigma_2$ of $\Sigma_1$ and $\Sigma_2$ at the origin. We shall construct a local diffeomorphism near the origin of $\mathbb R^3$ such that it maps $\Sigma_i$ to  $T_0\Sigma_i$. 
	
	By performing a rotation if necessary, we assume without loss of generality that $T_0 \Sigma_1$ is $\Pi_1 = \{ (x,y,z) \mid z = 0 \}$ and $T_0 \Sigma_2$ is $\Pi_2 = \{ (x,y,z) \mid y \sin \theta - z \cos \theta = 0 \}$. 
	
	Consider smooth defining functions $g_1, g_2: U \to \mathbb{R}$ such that $\Sigma_i \cap U = \{ (x, y, z) \in U \mid g_i(x, y, z) = 0 \}$, where $U$ is a sufficiently small neighborhood of the origin. We normalize these functions such that their gradients at the origin align with the unit inner normal vectors of $\Pi_1$ and $\Pi_2$:$$\nabla g_1(0) = (0, 0, 1), \quad \nabla g_2(0) = (0, \sin \theta, -\cos \theta).$$
	We define the diffeomorphism $\Phi\colon U \to \mathbb{R}^3$ explicitly as:
	$$\Phi(x, y, z) = \left( x, \frac{g_2(x, y, z) + g_1(x, y, z) \cos \theta}{\sin \theta}, g_1(x, y, z) \right)$$
	Now a straightforward computation shows that $\Phi$ is a local diffeomorphism and its derivative at the origin is the identity map, that is, $\Phi'(0) = I$. By construction, $\Phi$ maps $\Sigma_i$ to $T_0\Sigma_i$. 
\end{example}

\begin{example}\label{ex:codimthree}
Suppose $N$ is a three-dimensional polyhedral manifold in $\mathbb R^3$ and the origin $0$ is a vertex of $N$. The boundary $\partial N$ near the origin consists of the collection of smooth surfaces  $\mathcal{S} = \{\Sigma_1, \dots, \Sigma_n\}$  passing through the origin. The tangent planes $\Pi_i = T_0 \Sigma_i$ enclose a polyhedral cone  $\fiber$ (see Definition \ref{def:polyhedralcone}) in the tangent space $T_0 \mathbb{R}^3 \cong \mathbb{R}^3$. We shall construct  a local $C^1$-diffeomorphism near the origin of $\R^3$ such that it simultaneously maps each $\Sigma_i$ to its tangent plane $\Pi_i$.
	
	For each pair of adjacent $\Sigma_i$ and $\Sigma_j$, let $\Phi_{ij}\colon  U \to \mathbb{R}^3$ be the diffeomorphism constructed as in Example \ref{ex:codimtwo} that maps $\Sigma_i$ and $\Sigma_j$ to its corresponding tangent plane at the origin and $\Phi'_{ij}(0) = I$. To obtain a single map $\Phi\colon U \to \mathbb{R}^3$ that simultaneously maps each $\Sigma_i$ to its tangent plane $\Pi_i$, we pass to polar coordinates. 
	The link $\link = \fiber \cap \sph^2$ of  the polyhedral cone $\fiber$ is a spherical polygon in $\mathbb{S}^2$. Let $\{W_\alpha\}$ be an open cover of $\mathbb{S}^2$ such that each $W_\alpha$ contains at most one vertex of the polygon $\link$, and, if $W_\alpha$ does not contain a vertex of $\link$, then it intersects at most one edge of $\link$. Let $\{\rho_\alpha\}$ be a partition of unity subordinate to this cover. For each $W_\alpha$, choose $\Phi_\alpha \colon U\to \mathbb R^3$ as follows: 
	\begin{enumerate}
		\item If $W_\alpha$ contains a vertex of $\link$ formed by the intersection $\Pi_i \cap \Pi_j\cap \sph^2$, we set $\Phi_\alpha = \Phi_{ij}$.
		
		\item If $W_\alpha$ intersects an edge $\Pi_i \cap \mathbb{S}^2$ but contains no vertices of $\link$, consider the endpoints of this edge, which are the vertices of $\link$ given by  $ \Pi_i\cap \Pi_j \cap \mathbb{S}^2$ and $\Pi_i\cap \Pi_k \cap \mathbb{S}^2$. Both associated maps, $\Phi_{ij}$ and $\Phi_{ik}$, map $\Sigma_i$ to $\Pi_i$; thus, we set $\Phi_\alpha$ to be either one.
		\item If $W_\alpha$ is disjoint from the edges and vertices of $\link$, then we  set $\Phi_\alpha$ to be the identity map.
	\end{enumerate}
	We define the map $\Phi\colon  U\to \mathbb R^3$ by:
	\begin{equation}     	
		\Phi(x) = \sum_\alpha \rho_\alpha(\sigma) \Phi_\alpha(x),
	\end{equation}
	where $x \in U \setminus \{0\}$ is expressed in polar coordinates as $x = r\sigma$ with $r = \|x\|$ and $\sigma = x/\|x\| \in \mathbb{S}^2$. We extend this map to the origin by setting $\Phi(0)=0$. 
	
	We claim that  $\Phi\colon U \to \mathbb{R}^3$ is a local $C^1$-diffeomorphism near the origin.  It is clear that  $\Phi$ is smooth away from the origin. Now let us consider the differentiability of $\Phi$ at the origin. Since each $\Phi_\alpha$ satisfies $\Phi_\alpha(0)=0$ and $\Phi'_\alpha(0)=I$, we have the Taylor expansion: 
	\begin{equation*} 
		\Phi_\alpha(x) = x + R_\alpha(x),
	\end{equation*}
	where the remainder term $R_\alpha(x)$ satisfies
	\[ 
	\lim_{x \to 0} \frac{\|R_\alpha(x)\|}{\|x\|} = 0.
	\] 
	It follows that 	
	\begin{equation}
		\begin{split}
			\Phi(x) &= \sum_\alpha \rho_\alpha(\sigma) \left( x + R_\alpha(x) \right)  \\
			&= \Big(\sum_\alpha \rho_\alpha(\sigma) \Big) x + \sum_\alpha \rho_\alpha(\sigma) R_\alpha(x)\\
			& = x +  \sum_\alpha \rho_\alpha(\sigma) R_\alpha(x)
		\end{split}
		\label{eq:taylor}
	\end{equation}
	where we have used $\sum_\alpha \rho_\alpha(\sigma) = 1$ for all $\sigma \in \mathbb{S}^2$. This implies that $\Phi$ is differentiable at the origin with derivative $\Phi'(0) = I$. A similar computation shows that in fact the derivative $\Phi'$ is continuous at the origin. Therefore $\Phi$ is a local  $C^1$-diffeomorphism. By construction, $\Phi$  maps each surface $\Sigma_i$ to its corresponding tangent plane $\Pi_i$ at the origin. 
\end{example}

The following theorem (Theorem~\ref{thm:reduction}) is a key result that allows us to reduce the verification of the essential self-adjointness of $D_B$ to certain standard model cases. We emphasize that Theorem~\ref{thm:reduction} does not itself establish the essential self-adjointness of $D_B$; rather, it serves purely as a reduction principle. In particular, the hypotheses involving the comparison of dihedral angles are slightly more general than those that will be imposed in the actual essential self-adjointness result (Theorem \ref{thm:ess-saVector}).

\begin{theorem}\label{thm:reduction}
	Let $f \colon \domain \to \target$ be a spin polyhedral map between two polyhedral manifolds of dimension $3$.  Let $D$ be the Dirac operator associated with the bundle  $E= S(T\domain\oplus T\target)$. Let $\overbar \nu_k$ be the unit inner normal vector of each codimension one  face $\overbar F_k$ of $\domain$, and let $\nu_k$ be a smooth unit-length section of the bundle $f^\ast T\target$ on $\overbar F_k$. Let $B$ be the local boundary condition\footnote{Note that $\overbar \nu_k$ denotes the unit inward normal vector to the face $\overbar F_k$ of $\domain$, whereas $\nu_k$ is not necessarily the unit inward normal vector to the corresponding face $F_k$ of $\target$. Thus the boundary condition $B$ considered here is slightly more general than those in Definition~\ref{def:boundaryCondition}. This additional flexibility will be needed for our index theorem (Theorem \ref{thm:indexTheoremVector}). Nevertheless, it is still straightforward to verify that the Dirac operator $D$ is symmetric with respect to this more general boundary condition $B$. } on $E $ given by
	$$\mathscr E \overbar c(\overbar \nu_k)c(\nu_k)\sigma=-\sigma\textup{ on }\overbar F_k.$$  
	Assume that all dihedral angles of $\domain$ are less than $\pi$. If for each pair of adjacent codimension one faces $\overbar F_i$ and $ \overbar F_j$ of $\domain$, we have  either
	\begin{equation}
		\langle\overbar \nu_i,\overbar \nu_j\rangle \neq \langle\nu_i,\nu_j\rangle \textup{ along each connected component of } \overbar F_i\cap \overbar F_j, 
	\end{equation}
	or
	\begin{equation}
		\langle\overbar \nu_i,\overbar \nu_j\rangle = \langle\nu_i,\nu_j\rangle  \textup{ along each connected component of } \overbar F_i\cap \overbar F_j,
	\end{equation} 
	then for any point $x\in \partial \domain$,  there exist a neighborhood $ W$ of $x$ in $\domain$, a neighborhood $U$ of the origin in the tangent cone $\tancone_x\domain$, and a Lipschitz bundle isomorphism $\Psi$ from the trivial spinor bundle $S(\mathbb R^3\oplus \mathbb R^3)$ over $U$ to the spinor bundle $E = S(T\domain \oplus f^\ast T\target)$ over $W$, satisfying the following properties:
	\begin{enumerate}[label*=$(\arabic*)$]
		\item $\Psi$ maps the constant boundary condition determined by the
		constant vector fields
		\[
		\overbar\omega_i \equiv \overbar\nu_i(x) \textup{ and } 
		\omega_i \equiv \nu_i(x)
		\]
		on the codimension-one faces \(U\cap \Pi_i\) of \(U\)  (see line \eqref{eq:modelbd}),  where $\Pi_i= T_x\overbar F_i$ is the tangent plane of $\overbar F_i$ at $x$,  to  the boundary condition along codimension one faces $W\cap \overbar F_i$ of $W$.  
		\item  The  operator $ \Psi^* D \Psi$ on $U$ takes the form
		\[ \Psi^* D \Psi = D^\dR + \mathcal{A} + \mathcal{B}, \]
		where $D^\dR$ is the standard de Rham on the Euclidean tangent cone $\tancone_x\domain$, $\mathcal{A}$ is a first-order differential operator whose coefficients are continuous and vanish at the origin, and   $\mathcal{B}$ is a zeroth-order operator given by a uniformly bounded matrix-valued function over $U$. Consequently, for any $\delta > 0$, the neighborhood $U$ can be chosen sufficiently small such that the coefficients of $\mathcal{A}$ are bounded by $\delta$ in the supremum norm.
	\end{enumerate}
\end{theorem}
\begin{proof} 
	The case where $x$ lies in the interior of a codimension one face is classical and  elementary. We shall focus the proof on the case where $x$ lies in the interior of a codimension two face (i.e. an edge) or $x$ lies in a codimension three face (i.e. a vertex).

	\textbf{Codimension two case.} In this case, $x$ lies in the interior of an edge, say, the intersection of codimension one faces $\overbar F_1$ and $\overbar F_2$.  Recall that $\domain$ is contained in an open manifold $X$.  Let $\exp_x \colon T_x X \to X$ be the exponential map at $x$, which  a diffeomorphism from a small neighborhood $\widetilde{\mathcal W} \subset T_x X$ of the origin onto a neighborhood $\mathcal W \subset X$ of $x$. The  hypersurfaces $\overbar F_1\cap \mathcal W$ and $\overbar F_2\cap \mathcal W$ pull back to two smooth hypersurfaces $\Sigma_1, \Sigma_2 \subset \widetilde{\mathcal W}$. Let $\widetilde E = \exp_x^*(E)$ be the pullback bundle over $\widetilde{\mathcal W}$. We identify $\widetilde E$ with the trivial bundle $\widetilde{\mathcal W}\times E_x$ via parallel transportation along radial geodesics of $\mathcal W$. Note that, under this identification, the unit inner normal vector $\overbar \nu_i(x)$ of $\overbar F_i$ at $x$ coincides with the unit inner normal vector of $\Sigma_i$. Let $\Pi_i$ be the tangent plane of  $\overbar F_i$ at $x$. The tangent cone $\tancone_x \domain$ at $x$ is the polyhedral cone in $T_xX$ enclosed by $\Pi_1$ and $\Pi_2$. Again, under the above identification, the unit inner normal vector of $\Pi_i$ coincides with $\overbar \nu_i(x)$. 
	
	Let $\Phi \colon \widetilde{\mathcal W} \subset \R^3 \to \mathbb{R}^3$ be the local diffeomorphism constructed in Example \ref{ex:codimtwo}. This map satisfies:
	\begin{itemize}
		\item $\Phi(0) = 0$ and $\Phi'(0) = I$.
		\item $\Phi$ maps $\Sigma_i$ to $\Pi_i$.
	\end{itemize}
	By choosing a smaller neighborhood of the origin if necessary, we may assume without loss of generality that $\Phi$ is a diffeomorphism $\Phi \colon \widetilde{\mathcal W} \to \mathcal U$, where $\mathcal U = \Phi(\widetilde {\mathcal W})$.	Define  $\Lambda\coloneqq \exp_x \circ \Phi^{-1} \colon \mathcal U  \to \mathcal W$. By construction,  $\Lambda$ is a diffeomorphism such that $\Lambda(0)=x$ and $\Lambda'(0) = I$, and maps $\Pi_i\cap \mathcal U$  to $\overbar F_i\cap \mathcal W$. Moreover, we have the following identification of bundles. 
	
	\begin{center}
		\begin{tikzcd}[row sep=large, column sep= 3.5cm]
			\mathcal U \times E_x 
			\arrow[r, "\Phi^{-1} \times \operatorname{id}_{E_x}"] 
			\arrow[d, "\pi_{\mathcal U}"'] 
			& 
			\widetilde{\mathcal W} \times E_x 
			\arrow[r, "\text{Parallel Transport}"] 
			\arrow[d, "\pi_{\widetilde{\mathcal W}}"] 
			& 
			E|_{\mathcal W} 
			\arrow[d, "\pi_{\mathcal W}"] 
			\\
			\mathcal U 
			\arrow[r, "\Phi^{-1}"] 
			& 
			\widetilde{\mathcal W} 
			\arrow[r, "\exp_x"] 
			& 
			\mathcal W
		\end{tikzcd}
	\end{center}
	We denote this bundle identification
	 by $\mathcal J$. 
	Now set $U = \mathcal U \cap \tancone_x\domain$ and $W = \mathcal W\cap \domain$. 	Recall that, at the codimension one face $\overbar F_j$, the boundary condition for a section $\sigma$ of $E$  is defined by the operator $B_j = \mathscr{E} \bar{c}(\bar{\nu}_j) c(\nu_j)$:
	$$\mathscr{E} \bar{c}(\bar{\nu}_j) c(\nu_j) \sigma = -\sigma.$$ 
	we define a \textit{constant boundary condition} on the codimension one faces $\Pi_1$ and $\Pi_2$ of $\tancone_x\domain$ by
	\begin{equation}\label{eq:modelbd}
		\mathscr{E} \bar{c}(\bar{\omega}_i) c(\omega_i) \sigma = -\sigma.
	\end{equation} 
	where  $\bar{\omega}_i \equiv \bar{\nu}_i(x)$ and $\omega_i \equiv \nu_i(x)$  are constant vector fields along $\Pi_i$.

	With the above bundle identification,  $\overbar \nu_i$ and $\nu_i$ become unit-length vector fields\footnote{Although $\overbar \nu_i$ coincides with the unit inner normal vector of $\Pi_i$ at the origin of $U$, but $\overbar \nu_i$ is not necessarily the unit inner normal vector  of $\Pi_i$ at other points of $\Pi_i$.} of the trivial vector bundle $\R^3\oplus \R^3$ along $\Pi_i$.   	For each  $y\in \Pi_i\cap U$, let $V_i(y)\subset  S(\mathbb R^3\oplus \mathbb R^3)$  be the $(-1)$-eigenspace of the operator 
	\[ B_i = \mathscr{E} \bar{c}(\bar{\nu}_i) c(\nu_i) \colon S(\mathbb R^3\oplus \mathbb R^3) \to S(\mathbb R^3\oplus \mathbb R^3).  \] Note that $\bar{c}(\bar{\nu}_i) B_i = - B_i\bar{c}(\bar{\nu}_i)$, which forces  $V_i(y)$ to be of  dimension $8/2 = 4$.  In other words, the above bundle identification pulls back the boundary conditions for codimension one faces $\overbar F_i\cap W$ of $W$ to boundary conditions for codimension one faces $\Pi_i\cap U$ of $U$. 
	
	This pull-back boundary condition $B_i$ coincides with the constant boundary condition for $\Pi_i\cap  U$ at $x$,  but it generally does not coincide with the constant boundary condition at nearby points $y\in \Pi_i\cap U$. To remedy this, we shall construct a bundle isomorphism that maps the boundary conditions for  $ U$ to the boundary conditions for $ W$. 
	
	The bundle isomorphism $\Theta\colon  U\times S(\mathbb R^3\oplus \mathbb R^3)  \to  U\times  S(\mathbb R^3\oplus \mathbb R^3)$ is constructed as follows. Importantly, because the subbundles $V_i$ are preserved by the grading operator $\mathscr{E}$ (Lemma \ref{lemma:boundaryface} and \ref{lemma:boundaryedge}), the following construction can and should be carried out individually on each subspace $S^\pm(\mathbb{R}^3 \oplus \mathbb{R}^3)$, where $S^\pm(\mathbb{R}^3 \oplus \mathbb{R}^3)$ is the $(\pm 1)$-eigenspace of the grading operator $\mathscr E$. This ensures that the resulting isomorphism $\Theta$ commutes with $\mathscr{E}$, hence preserves the grading on $S(\mathbb{R}^3 \oplus \mathbb{R}^3)$. For notational simplicity, we shall omit the superscript $\pm$ in the following construction.
	\begin{enumerate}
		\item If $\langle\overbar \nu_1,\overbar \nu_2\rangle \neq \langle\nu_1,\nu_2\rangle$ along $\overbar F_1\cap \overbar F_2$, then it follows from  Lemma \ref{lemma:boundaryedge} that  the intersection $V_1(y) \cap V_2(y) = \{0\}$ for all $y \in \Pi_1\cap \Pi_2\cap U$.  Let $p_i(y)$ be the orthogonal projection  $S(\mathbb R^3\oplus \mathbb R^3) \to V_i(y)$. Since $V_1(y) \cap V_2(y) = \{0\}$ and $\dim V_1(y) = \dim V_2(y) = \frac{1}{2} \dim  S(\mathbb R^3\oplus \mathbb R^3) = 4$, it follows that $p_1(y) + p_2(y)$ is invertible.   As of now,  $p_i$ is a projection-valued function only defined on $\Pi_i\cap U$. We extend it smoothly to a projection-valued function on $U$, for example, consider an extension\footnote{Such an extension of, say $p_1$,  is only used to help us define the bundle isomorphism on $U$, and is \emph{not} meant as introducing an auxiliary boundary condition on the other codimension one face $\Pi_2$.} that is constant along the normal direction of $\Pi_i$.  Since $\langle\overbar \nu_1,\overbar \nu_2\rangle \neq \langle\nu_1,\nu_2\rangle$ is an open condition, by choosing a smaller neighborhood of the origin of $\R^3$ if necessary,  we define  $\Theta$ by setting 
		\[ \Theta(z) = \left( p_1(0) p_1(z) + p_2(0) p_2(z) \right) \left( p_1(z) + p_2(z) \right)^{-1} \]
		for all $z\in U$. By construction, $\Theta$ maps $V_i(y)$ to $V_i(0)$ for each $y\in \Pi_i\cap U$ in a small neighborhood of the origin.  
		\item The case where $\langle\overbar \nu_1,\overbar \nu_2\rangle =  \langle\nu_1,\nu_2\rangle$ is similar,  only slightly more involved. In this case, it follows from  Lemma \ref{lemma:boundaryedge} that  the intersection $\dim(V_1(y) \cap V_2(y)) = 2$ for all $y \in \Pi_1\cap \Pi_2\cap U$. Without loss of generality, assume that 
		\begin{align*}
			\Pi_1 &= \{ (z_1, z_2, z_3) \in \mathbb{R}^3 \mid z_3 = 0 \}, \\
			\Pi_2 &= \{ (z_1, z_2, z_3) \in \mathbb{R}^3 \mid z_2 \sin \alpha -	 z_3 \cos \alpha = 0 \},
		\end{align*}
		where  $\alpha \in (0, \pi)$ is the dihedral angle between $\Pi_1$ and $\Pi_2$.

		Define the linear coordinates $(u, v)$ on the $(z_2, z_3)$-plane by:
		\begin{equation}
			u(z_2, z_3) = z_3, \qquad v(z_2, z_3) = z_2 \sin \alpha - z_3 \cos \alpha.
		\end{equation}
		Under this transformation, $\Pi_1$ corresponds exactly to $u = 0$, and $\Pi_2$ corresponds exactly to $v = 0$.
		
		Let $\{e_i(z_1, 0, 0)\}_{i=1, 2}$ be a smooth orthonormal basis for the intersection $V_1 \cap V_2$ along the edge $\Pi_1 \cap \Pi_2 \cap U$. Because $V_1$ is a smooth bundle over $\Pi_1$, we can smoothly extend $e_1$ and $e_2$ to orthonormal sections $\varphi_1$ and $\varphi_2$ of $V_1 \subset S(\mathbb{R}^3 \oplus \mathbb{R}^3)$ over $\Pi_1$, and then extend $\varphi_1$ and $\varphi_2$ to smooth sections of $S(\mathbb{R}^3 \oplus \mathbb{R}^3)$ over $U$ by extending them constantly along the lines where $v$ and $z_1$ are constant and $u$ varies. We still denote these  extended sections over $U$ by $\varphi_1$ and $ \varphi_2$. 
		
		Similarly, we can smoothly extend $e_1$ and $e_2$ to orthonormal sections $\psi_1$ and $\psi_2$ of $V_2 \subset S(\mathbb{R}^3 \oplus \mathbb{R}^3)$ over $\Pi_2$, and then extend $\psi_1$ and $\psi_2$ to smooth sections of $S(\mathbb{R}^3 \oplus \mathbb{R}^3)$ over $U$ by extending them constantly along the lines where $u$ and $z_1$ are constant and $v$ varies. We still denote these extended sections over $U$ by $\psi_1$ and $\psi_2$.
		
		We define a section $\tilde{e}_i$ of $S(\mathbb{R}^3 \oplus \mathbb{R}^3)$ over $U$ by 
		\begin{equation} \label{eq:extension}
			\tilde{e}_i(z_1, z_2, z_3) = \varphi_i(z_1, z_2, z_3) + \psi_i(z_1, z_2, z_3) - e_i(z_1, 0, 0).
		\end{equation}
		By construction, we have 
		$  \tilde{e}_i(z_1, 0,0 ) = e_i(z_1, 0, 0)$
		and $\tilde{e}_i|_{\Pi_1} \in V_1$ and $\tilde{e}_i|_{\Pi_2} \in V_2$. 
		
		Now let $\tilde w_1$ and $\tilde w_2$ be an orthonormal basis of the orthogonal complement of $\{\tilde e_1, \tilde e_2\}$ in $V_1$ along $\Pi_1$.  We extend them to smooth sections of $S(\mathbb{R}^3 \oplus \mathbb{R}^3)$ over $U$ by extending them constantly along the lines where $v$ and $z_1$ are constant and $u$ varies. We still denote these  extended sections over $U$ by $\tilde w_1$ and $\tilde w_2$.
		
		Now let $\tilde w'_1$ and $\tilde w'_2$ be an orthonormal basis of the orthogonal complement of $\{\tilde e_1, \tilde e_2\}$ in $V_2$ along $\Pi_2$.  We extend them to smooth sections of $S(\mathbb{R}^3 \oplus \mathbb{R}^3)$ over $U$ by extending them constantly along the lines where $u$ and $z_1$ are constant and $v$ varies. We still denote these  extended sections over $U$ by $\tilde w'_1$ and $\tilde w'_2$.
		
		Let $\widetilde{V}_1(z)$ be the linear span of $\{\tilde{e}_1(z), \tilde{e}_2(z), \tilde{w}_1(z), \tilde{w}_2(z)\}$ and $\widetilde{V}_2(z)$ be the linear span of $\{\tilde{e}_1(z), \tilde{e}_2(z), \tilde{w}'_1(z), \tilde{w}'_2(z)\}$.
		Since $\tilde e_1(z), \tilde e_2(z), \tilde w_1(z)$ and $\tilde w_2(z)$  are linearly independent at $z = (z_1, 0, 0)$, by choosing a smaller neighborhood of the origin if necessary, we may assume without loss of generality that $\tilde e_1(z), \tilde e_2(z), \tilde w_1(z)$ and $\tilde w_2(z)$  are linearly independent at every $z = (z_1, z_2, z_3)\in U$.  By the same reasoning, we may assume without loss of generality that $\tilde e_1(z), \tilde e_2(z), \tilde w'_1(z)$ and $\tilde w'_2(z)$  are linearly independent at every $z = (z_1, z_2, z_3)\in U$, and $\dim \widetilde{V}_1(z)\cap \widetilde V_2(z)  = 2$ at every $z\in U$. 
		
		Let $p_1(z)$ be the projection  $S(\mathbb{R}^3 \oplus \mathbb{R}^3)\to \widetilde V_1(z)$ and   $p_2(z)$ be the projection  $S(\mathbb{R}^3 \oplus \mathbb{R}^3)\to \widetilde V_2(z)$.   Let $q_3(z)$ be the projection  $S(\mathbb{R}^3 \oplus \mathbb{R}^3)\to \widetilde V_1(z) \cap \widetilde V_2(z)$.  Define $q_1(z) = p_1(z) - q_3(z)$, $q_2(z) = p_2(z) - q_3(z)$, and $q_4(z)$ as the projection onto $(V_1(z) + V_2(z))^\perp$. We define the bundle isomorphism $\Theta$ by 
		\[ \Theta(z) = \Big( \sum_{j=1}^4 q_j(0) q_j(z) \Big) \Big( \sum_{j=1}^4 q_j(z) \Big)^{-1}. \]
		By construction, $\Theta$ maps $V_i(y)$ to $V_i(0)$ for each $y\in \Pi_i\cap U$. 
	\end{enumerate}
	
	We emphasize again that the construction of $\Theta$ should be interpreted as being carried out on each subspace $S^\pm(\mathbb{R}^3 \oplus \mathbb{R}^3)$ independently (cf. Lemmas \ref{lemma:boundaryface} and \ref{lemma:boundaryedge}). In particular, $\Theta$ commutes with the grading operator $\mathscr{E}$ by construction. We define $\Psi \coloneqq \mathcal J\circ \Theta^{-1}$. Recall that $\mathcal J$ is the bundle identification obtained from \(\Phi^{-1}\), \(\exp_x\), and parallel
	transport. By construction, $\Psi$ maps sections satisfying the constant boundary conditions on $U$ to sections satisfying the boundary conditions on $W$.

	Let us consider the  difference $\Psi^* D \Psi- D^\dR$. Let $\{f_j = \frac{\partial}{\partial y_j}\}_{1\leq j \leq 3}$ be the standard orthonormal  basis of $TU = \mathbb{R}^3$. We choose a local orthonormal frame $\{e_j\}$ for $TW$ such that $e_j(x) = d\Lambda_0 (f_j)$, where $d\Lambda_0$ is the differential map of $\Lambda$ at the origin $0\in U$. 
	
	Let $\overbar g$ denote the Riemannian metric on $\domain$. The Riemannian volume form on $W$ pulls back to a volume form on $U$ given by $$\sqrt{\det(\Lambda^* \overbar g)_y} \, dy_1 \wedge dy_2 \wedge dy_3,$$ where the entries of the matrix  $(\Lambda^\ast \overbar g)_y$ are defined as 
	\[ ((\Lambda^\ast \overbar g)_y)_{ij} = \overbar g_{\Lambda(y)}(d\Lambda_y(f_i), d\Lambda_y(f_j)) \]
	for each $y\in U$. 
	
	The difference $\Psi^* D \Psi- D^\dR$ can be written as
	$$\Psi^* D \Psi - D^\dR = \sum_{k=1}^3 A_k(y) \frac{\partial }{\partial y_k}	 + \mathcal B(y),$$   where $A_k$ and $\mathcal B$ are matrix-valued smooth functions on $U$.
	The coefficients $A_k(y)$ are given by  $$A_k(y) = - \overbar c(f_k) + \sqrt{\det(\Lambda^* \overbar g)_y} \sum_{j=1}^3 a_{jk}(y) \, (\Theta^{-1})^*_y \circ \overbar c(e_j(\Lambda(y)))\circ \Theta^{-1}_y,$$
	where $\overbar c(e_j(\Lambda(y)))$ is the Clifford multiplication by the vector $e_j(\Lambda(y))$ and the continuous functions $a_{jk}(y)$ are the coefficients of $d\Lambda^{-1}(e_j) $ in the standard basis  $\{f_j = \frac{\partial}{\partial y_j} : 1\leq j \leq 3\}$: 
	\[d\Lambda^{-1}(e_j) = \sum_{k=1}^3 a_{jk}(y) \frac{\partial}{\partial y_k}.\] 
	At the origin $y=0$, we have $\Lambda(0)=x$, $d\Lambda_0=I$, and $\Theta_0=I$. It follows that  $\sqrt{\det(\Lambda^* \overbar g)_0} = 1$. Furthermore, since $e_j(x) = d\Lambda_0 (f_j)$, we have $a_{jk}(0) = \delta_{jk}$. Therefore, we have $$A_k(0) = - \overbar c(f_k)  + \overbar c(f_k) = 0$$
	It follows that for any $\delta > 0$, there exists a sufficiently small neighborhood where $\|A_k(y)\| < \delta$.	
	
	\textbf{Codimension three case.} In this case, $x$ lies at a vertex, the intersection of codimension one faces $\overbar F_1, \dots, \overbar F_n$. Recall that $\domain$ is a subspace inside some open manifold $X$. Let $\exp_x \colon T_x X \to X$ be the exponential map at $x$, which is a diffeomorphism from a small neighborhood $\widetilde{\mathcal W} \subset T_x X$ of the origin onto a neighborhood $\mathcal W \subset X$ of $x$. The hypersurfaces $\overbar F_1\cap \mathcal W, \dots, \overbar F_n\cap \mathcal W$ pull back to smooth surfaces $\Sigma_1, \dots, \Sigma_n \subset \widetilde{\mathcal W}$. 
	
	Let $\widetilde E = \exp_x^*(E)$ be the pullback bundle over $\widetilde{\mathcal W}$. We identify $\widetilde E$ with the trivial bundle $\widetilde{\mathcal W}\times E_x$ via parallel transportation along radial geodesics of $\mathcal W$. Note that, under this identification, the unit inner normal vector $\overbar \nu_i(x)$ of $\overbar F_i$ at $x$ coincides with the unit inner normal vector of $\Sigma_i$. Let $\Pi_i$ be the tangent plane of $\overbar F_i$ at $x$. The tangent cone $\tancone_x \domain$ at $x$ is the polyhedral cone in $T_xX$ enclosed by $\Pi_1, \dots, \Pi_n$. Again, under the above identification, the unit inner normal vector of $\Pi_i$ coincides with $\overbar \nu_i$.

	Let $\Phi \colon \widetilde{\mathcal W} \subset \R^3 \to \mathbb{R}^3$ be the local $C^1$-diffeomorphism constructed in Example \ref{ex:codimthree}. This map satisfies:
	\begin{itemize}
		\item $\Phi(0) = 0$ and $\Phi'(0) = I$.
		\item $\Phi$ maps $\Sigma_i$ to $\Pi_i$.
	\end{itemize}
	By choosing a smaller neighborhood of the origin if necessary, we may assume without loss of generality that $\Phi$ is a $C^1$-diffeomorphism $\Phi \colon \widetilde{\mathcal W} \to \mathcal U$, where $\mathcal U = \Phi(\widetilde {\mathcal W})$.	Define  $\Gamma \coloneqq \exp_x \circ \Phi^{-1} \colon \mathcal U  \to \mathcal W$. By construction, we have $\Gamma$ is a $C^1$-diffeomorphism such that $\Gamma(0)=x$ and $\Gamma'(0) = I$, and maps $\Pi_i\cap \mathcal U$  to $\overbar F_i\cap \mathcal W$. Similar to the codimension two case, let $\mathcal J\colon \mathcal U\times E_x\to E|_{\mathcal W}$ be the the bundle identification obtained from \(\Phi^{-1}\), \(\exp_x\), and parallel
	transport.

	Now set $U = \mathcal U \cap \tancone_x\domain$ and $W = \mathcal W\cap \domain$. 
	We consider a \textit{constant boundary condition} on the codimension one faces $\Pi_i$ of $\tancone_x\domain$ by
	\begin{equation}\label{eq:modelbd_codim3}
		\mathscr{E} \bar{c}(\bar{\omega}_i) c(\omega_i) \sigma = -\sigma.
	\end{equation}
	where $\bar{\omega}_i \equiv \bar{\nu}_i(x)$ and $\omega_i \equiv \nu_i(x)$ are constant vector fields along $\Pi_i$.  
	
	With the above bundle identification, $\overbar \nu_i$ and $\nu_i$ become unit-length vector fields of the trivial vector bundle $\mathbb{R}^3\oplus \mathbb{R}^3$ along $\Pi_j$. For each $y \in \Pi_i$, let $V_i(y)\subset S(\mathbb R^3\oplus \mathbb R^3)$  be the $(-1)$-eigenspace of the operator 
	\[
	B_i = \mathscr{E} \bar{c}(\bar{\nu}_i) c(\nu_i) \colon S(\mathbb R^3\oplus \mathbb R^3) \to S(\mathbb R^3\oplus \mathbb R^3).
	\]
	In other words, the bundle identification pulls back the boundary conditions for codimension one faces $\overbar F_i\cap W$ of $W$ to boundary conditions for codimension one faces $\Pi_i\cap U$ of $U$.
	
	This pull-back boundary condition $B_i$ coincides with the constant boundary condition for $\Pi_i\cap U$ at $x$, but it generally does not coincide with the constant boundary condition at nearby points $y\in \Pi_i\cap U$. To remedy this, we construct a bundle isomorphism $\Theta\colon U\times S(\mathbb R^3\oplus \mathbb R^3) \to U\times S(\mathbb R^3\oplus \mathbb R^3)$ using a partition of unity approach analogous to the construction of $\Phi$.
	
	Let $\Theta_{ij}$ be the bundle isomorphism near the origin constructed in the codimension two case for the intersection of faces $\Pi_i$ and $\Pi_j$. Using the same open cover $\{W_\alpha\}$ of $\mathbb{S}^2$ and partition of unity $\{\rho_\alpha\}$ as in the construction of the map $\Phi$ of Example \ref{ex:codimthree}, we select local bundle isomorphisms $\Theta_\alpha$ as follows:
	\begin{enumerate}
		\item If $W_\alpha$ contains a vertex of $\link$ formed by $\Pi_i \cap \Pi_j \cap \sph^2$, we set $\Theta_\alpha = \Theta_{ij}$.
		\item If $W_\alpha$ intersects an edge $\Pi_i \cap \mathbb{S}^2$ but no vertices, its endpoints correspond to intersections $\Pi_i \cap \Pi_j$ and $\Pi_i \cap \Pi_k$. Both $\Theta_{ij}$ and $\Theta_{ik}$ map the constant boundary condition on $\Pi_i$ to the boundary condition  pulled back from $\Sigma_i$; thus, we set $\Theta_\alpha$ to be either one.
		\item If $W_\alpha$ is disjoint from the edges and vertices of $\link$, we set $\Theta_\alpha$ to be the identity map.
	\end{enumerate}
	
	We define the global bundle isomorphism $\Theta$ on $U$ by:
	\begin{equation}\label{eq:bundlegluemap}
	\Theta(z) = \sum_\alpha \rho_\alpha(\sigma) \Theta_\alpha(z),
	\end{equation}
	where $z \in U \setminus \{0\}$ is expressed in polar coordinates as $z = r\sigma$. Because $\Theta_\alpha(0) = I$ for all $\alpha$, it is clear that $\Theta(0) = I$. By choosing a smaller neighborhood of the origin if necessary, $\Theta(z)$ remains  invertible on $U$, thus defining a bundle isomorphism. Clearly, by construction,  $\Theta$ is smooth on $U\setminus \{0\}$.  A computation similar to the line \eqref{eq:taylor} shows that $\Theta$ is Lipschitz on $U$.  
	
	We define $\Psi \coloneqq \mathcal J \circ \Theta^{-1}$, where $\mathcal J\colon U\times E_x\to E|_W$ is the bundle identification  obtained from \(\Phi^{-1}\), \(\exp_x\), and parallel
	transport. By construction, $\Psi$ maps sections satisfying the constant boundary conditions on $U$ to sections satisfying the boundary conditions on $W$.

	The difference $\Psi^* D \Psi- D^\dR$ can be written as
	$$\Psi^* D \Psi - D^\dR = \sum_{k=1}^3 A_k(y) \frac{\partial }{\partial y_k} + \mathcal B(y),$$   where $A_k$ are matrix-valued continuous functions and $\mathcal B$ is a matrix-valued \emph{uniformly bounded}\footnote{$\mathcal B$ is continuous on $U\setminus \{0\}$, but not necessarily continuous at the origin. This is because $\mathcal B$ involves the derivative of $\Theta$, where $\Theta$ is Lipschitz but its derivative is generally not continuous.} function on $U$. By the same argument for the coefficients $A_k(y)$ used in the \textbf{codimension two case}, we have $A_k(0) = 0.$
	In particular, it follows that for any $\delta > 0$, there exists a sufficiently small neighborhood where $\|A_k(y)\| < \delta$.	
\end{proof}

\begin{lemma}\label{lemma:boundaryface}
With the same notation as Theorem \ref{thm:reduction}, at a point $y$ in a codimension one face $\overbar F_k$ of $\domain$, let $V = V(y)$ be $(-1)$-eigenspace of the operator 
\[ B = \mathscr{E} \overbar{c}(\overbar{\nu}_k) c(\nu_k) \colon E_y \to E_y.  \]
We denote by $E_y^\pm$ the $(\pm 1)$-eigenspace of the grading operator $\mathscr E$, then
\[  \dim (V\cap E^+_y) = \dim( V \cap E_y^-) = \frac{1}{2}\dim V. \]
\end{lemma}
\begin{proof}
	Note that $B$ commutes with $\mathscr E$, that is $B \mathscr E = \mathscr E B$. It follows that $V$ has an orthogonal decomposition $V = V^+ \oplus V^-$, where $V^\pm = V\cap E^\pm_y$.
	Let $\overbar e \in T_y\domain $ be a unit vector that is orthogonal to $\overbar \nu_k$. Since $\overbar c(\overbar e)$ anticommutes with $\mathscr E$ and commutes with $B$, it follows that  $\overbar c(\overbar e)$ maps $V^\pm$  isomorphically to $V^\mp$. Consequently, $\dim V^+ = \dim V^- = \frac{1}{2}\dim V$. 
\end{proof}

\begin{lemma}\label{lemma:boundaryedge}
	With the same notation as Theorem \ref{thm:reduction}, at a point $y\in \overbar F_1\cap \overbar F_2$, let $V_1  = V_1(y)$ and $V_2 = V_2(y)$ be the $(-1)$-eigenspace of the operator 
	\[ B_j = \mathscr{E} \bar{c}(\bar{\nu}_j) c(\nu_j) \colon E_y \to E_y.  \]
	Then the dimension of  $V_1 \cap V_2$ satisfies the following: 
	\begin{enumerate}[label=$(\alph*)$]
		\item if $\langle \overbar{\nu}_1, \overbar{\nu}_2 \rangle = \langle \nu_1, \nu_2 \rangle$, then $\dim(V_1 \cap V_2) = 2$;
		\item if $\langle \overbar{\nu}_1, \overbar{\nu}_2 \rangle \neq \langle \nu_1, \nu_2 \rangle$, then $\dim(V_1 \cap V_2) = 0$.
	\end{enumerate} 
	Moreover, we have 
	\[ \dim(V_1 \cap V_2 \cap E^+_y) =  \dim(V_1 \cap V_2 \cap E^-_y) = \frac{1}{2} \dim(V_1 \cap V_2).\] 
\end{lemma}

\begin{proof}
	The operators $B_j$ are involutions (i.e., $B_j^2 = 1$) on $E_y$. A spinor $\sigma \in V_1 \cap V_2$ must satisfy $B_1 \sigma = -\sigma$ and $B_2 \sigma = -\sigma$. Consequently, $\sigma$ must be an eigenvector of the product operator $A = B_1 B_2$ with eigenvalue $1$.
	
	Consider the orthogonal decomposition $T_y \domain = \mathfrak{V} \oplus \mathfrak{V}^\perp$, where $\mathfrak{V} = \operatorname{span}(\overbar{\nu}_1, \overbar{\nu}_2)$ is the linear subspace spanned by $\overbar{\nu}_1$ and $ \overbar{\nu}_2$. Note that $\mathfrak V$ is two dimensional, since $\overbar{\nu}_1$ and $ \overbar{\nu}_2$ are linearly independent by our geometric assumption on $\domain$. We also choose a two-dimensional subspace
    $ \mathfrak W\subset (f^*T\target)_y $	containing \(\nu_1\) and \(\nu_2\). Of course, if \(\nu_1\) and \(\nu_2\) are linearly
	independent, then $
	\mathfrak W=\operatorname{span}(\nu_1,\nu_2).$  Similarly, decompose $(f^* T\target)_y = \mathfrak{W} \oplus \mathfrak{W}^\perp$. So  we have $E_y = S(\mathfrak V)\otimes  S(\mathfrak W) \otimes S(\mathfrak V^\perp \oplus \mathfrak W^\perp)$. The operators $B_1$ and $B_2$ act as the identity on the factor $S(\mathfrak V^\perp \oplus \mathfrak W^\perp)$; therefore, it suffices to analyze their action on the factor $S(\mathfrak V)\otimes  S(\mathfrak W)$.
	
	A direct computation shows that
	\[
	A = B_1 B_2 = \overline{c}(\overline{\nu}_1) \overline{c}(\overline{\nu}_2) \, c(\nu_1) c(\nu_2).
	\]
	Let $\alpha \in (0, \pi)$ be the angle such that $\langle \overline{\nu}_1, \overline{\nu}_2 \rangle = \cos \alpha$. The operator $L \coloneqq  \overline{c}(\overline{\nu}_1)\overline{c}(\overline{\nu}_2)$ acts on $S(\mathfrak V)$ with eigenvalues $\{ -e^{i\alpha}, -e^{-i\alpha} \}$ and acts by identity on $S(\mathfrak W)$.   
	Similarly, let $\beta \in [0, \pi]$ be the angle such that $\langle \nu_1, \nu_2 \rangle = \cos \beta$. The operator $R \coloneqq c(\nu_1)c(\nu_2)$ acts on $S(\mathfrak W)$ with eigenvalues $\{ -e^{i\beta}, -e^{-i\beta} \}$ and  acts by identity on $S(\mathfrak V)$.   
	
	It follows that the eigenvalues of  $A\colon S(\mathfrak V)\otimes  S(\mathfrak W) \to S(\mathfrak V)\otimes  S(\mathfrak W) $ are
	\[
	\{ e^{i(\alpha+\beta)}, e^{-i(\alpha+\beta)}, e^{i(\alpha-\beta)}, e^{-i(\alpha-\beta)} \}.
	\]
	Now, consider the two cases:
	\begin{enumerate}[label=$(\alph*)$]
		\item If $\langle \overbar{\nu}_1, \overbar{\nu}_2 \rangle = \langle \nu_1, \nu_2 \rangle$, then $\alpha = \beta$. Then $1$ occurs as an eigenvalue of $A$  with multiplicity $2$. Let $Q$ denote
		this two-dimensional eigenspace. Let $u_+, u_-$ be  eigenvectors of
		$ L = \bar c(\bar\nu_1)\bar c(\bar\nu_2)$ with eigenvalues
		$-e^{i\alpha}, -e^{-i\alpha}$, respectively;  and 
		$w_+, w_-$  eigenvectors of $ R = c(\nu_1)c(\nu_2)$ with eigenvalues
		$-e^{i\alpha}, -e^{-i\alpha}$, respectively. Then
		\[
		Q = \operatorname{span}\{u_+\otimes w_-,\,u_-\otimes w_+\}.
		\]
		 It remains to impose the condition
		\(B_1\sigma=-\sigma\).
		A direct computation shows that $ A B_1 = B_1 A^{-1}$, which implies that $B_1$ maps $Q$ to $Q$. 
	    Moreover, we have  
	    \[
	    B_1 L B_1= L^{-1} \textup{ and }
	    B_1RB_1=R^{-1}. 
	    \] 
	    This implies that $B_1(Q_+) = Q_-$ and $B_1(Q_-) = Q_+$, where  $Q_+ = \mathbb C(u_+\otimes w_-)$ and $Q_- = \mathbb C(u_-\otimes w_+)$. Therefore, with respect to this basis,
	    \[
	    B_1|_Q
	    =
	    \begin{pmatrix}
	    	0 & \lambda \\
	    	\lambda^{-1} & 0
	    \end{pmatrix}
	    \]
	    for some  $\lambda\neq 0$. It follows that the subspace
	    \[
	    J \coloneqq \{ \sigma \in Q \mid B_1 \sigma = -\sigma \}
	    \]
	    has dimension $1$.
		The intersection $V_1\cap V_2$ is then given by
		\[
		V_1 \cap V_2 = J \otimes S(\mathfrak V^\perp \oplus \mathfrak W^\perp).
		\]
		Since $\dim(J)=1$ and $\dim(S(\mathfrak V^\perp \oplus \mathfrak W^\perp)) = 2$, we have $\dim(V_1 \cap V_2) = 2$.
		
		\item If $\langle \overbar{\nu}_1, \overbar{\nu}_2 \rangle \neq \langle \nu_1, \nu_2 \rangle$, then $\alpha \neq \beta$. Consequently, $1$ is not an eigenvalue of $A$. Thus $V_1 \cap V_2 = 0$.
	\end{enumerate}
	To prove the final statement, let $\overbar e\in T_y\domain$ be a unit vector  that is orthogonal to both $\overbar \nu_1$ and $\overbar \nu_2$. Since $\overbar c(\overbar e)$ anticommutes with $\mathscr E$ and commutes with $B_j$, it follows that  $\overbar c(\overbar e)$ preserves the subspace $V_1\cap V_2$. Moreover, $\overbar c(\overbar e)$ maps $V_1 \cap V_2 \cap E^\pm_y$  isomorphically to $V_1 \cap V_2 \cap E^\mp_y$. Consequently, 	\[ \dim(V_1 \cap V_2 \cap E^+_y) =  \dim(V_1 \cap V_2 \cap E^-_y) = \frac{1}{2} \dim(V_1 \cap V_2).\] 
	This finishes the proof.
\end{proof}

\subsection{Coordinate change and conical operators}
Theorem \ref{thm:reduction} allows us to reduce the verification of the essential self-adjointness of $D_B$ to the following standard  model cases. Let $f\colon \mathbb{F}\to\mathbb{F}'$ be a polyhedral map between two polyhedral cones in $\mathbb{R}^k$ that are bounded by hyperplanes passing through the origin. Since these cones and their faces are flat, their tangent bundles are trivial, the normal vectors to the codimension-one faces are constant vector fields, canonically identified with vectors in $\mathbb{R}^k$. In this setting, the precise details of the map $f$ become irrelevant; its only role is to prescribe the correspondence between the faces of $\mathbb{F}$ and the faces of $\mathbb{F}'$.

Under the trivializations of $T\fiber$ and $f^*T\fiber'$, the spinor bundle  $S(T\fiber\oplus f^\ast T\fiber')$ is  naturally identified with $\Bigwedge^*\fiber \coloneqq \Bigwedge^*\R^k$ equipped with the trivial connection. Under this identification,  the Clifford actions $\overbar c$ and $c$ are identified with the usual Clifford actions on forms given by
$$\overbar c(v)=v\wedge -\iota_{v,L},\textup{ and }c(v)=\wedge v -\iota_{v,R},$$
where $\iota_{v, L}$ and $\iota_{v,R}$ denote the corresponding left and right contraction by $v$. See \cite[I.3.9]{spingeometry}. With these identifications, the twisted Dirac operator $D$ on $S(T \fiber \oplus f^\ast T\fiber')$ coincides with the standard de Rham operator $D^{\dR} = d + d^*$ acting on $\Bigwedge^*\R^k$. The boundary condition $B$ is determined simply by the inner normal vectors of $\mathbb{F}$ and $\mathbb{F}'$. Since the faces are flat, these normal vectors are constant along each codimension-one face of $\mathbb{F}$. Consequently, the boundary condition $B$  is constant along each codimension one face. 

The Euclidean metric on $\fiber$ can be written in polar coordinates as
$$ g^\fiber = dr^2 + r^2 g^\link, $$
where $r$ is the radial distance from the origin and $\link = \fiber \cap \mathbb{S}^{k-1}$ is the link of $\fiber$.

	There is a natural unitary that identifies the de Rham operator $D^\dR$ on $\fiber$ with an elliptic operator on the cylinder $(0, \infty) \times \link$ equipped with the cylindrical metric (i.e., product metric), see \cite[Section 5]{LeschTopology}. Denote by $\Omega^\ast (\fiber)$ the space of differential forms on $\fiber$. The unitary $\Psi$ is given by
	\begin{equation}\label{eq:coordinateChange}
		\Psi=(\Psi_{\mathrm{even}},\Psi_{\mathrm{odd}})\colon 
		C^\infty((0,\infty),\Omega^*\link)\oplus C^\infty((0,\infty),\Omega^*\link)
		\longrightarrow \Omega^\ast(\fiber),
	\end{equation}
	where  
	\begin{equation}\label{eq:transformeven}
		\Psi_{\mathrm{even}}\colon C^\infty((0,\infty),\Omega^* \link)\longrightarrow \Omega^{\mathrm{even}} \fiber,~
		\omega_p\longmapsto\begin{cases}
			r^{p-\frac{k-1}{2}}\omega_p,&\text{ if $p$ is even}\\
			r^{p-\frac{k-1}{2}}\omega_p\wedge dr,&\text{ if $p$ is odd}
		\end{cases}
	\end{equation}
	and
	\begin{equation}\label{eq:transformodd}
		\Psi_{\mathrm{odd}}\colon C^\infty((0,\infty),\Omega^* \link)\longrightarrow \Omega^{\mathrm{odd}} \fiber,~
		\omega_p\longmapsto\begin{cases}
			r^{p-\frac{k-1}{2}}\omega_p,&\text{ if $p$ is odd}\\
			r^{p-\frac{k-1}{2}}\omega_p\wedge dr,&\text{ if $p$ is even}
		\end{cases}
	\end{equation}
	The Clifford actions in the cylindrical coordinates are given by
	\begin{equation}\label{eq:cylindricalClifford1}
		\Psi^{-1}\mathscr E\Psi=\begin{pmatrix}
			1&0\\0&-1
		\end{pmatrix},~\Psi^{-1}\overbar c(\partial_r)\Psi=\begin{pmatrix}
		0&-1\\1&0
		\end{pmatrix},~
		\Psi^{-1}c(\partial_r)\Psi=\begin{pmatrix}
			0&-\mathscr E_\link\\\mathscr E_\link&0
		\end{pmatrix},
	\end{equation}
	where $\mathscr E_\link$ is the grading operator on $\Omega^*\link$, and
	\begin{equation}\label{eq:cylindricalClifford2}
		\Psi^{-1}\overbar c(v)\Psi=\begin{pmatrix}
			0&\overbar c(v)\\\overbar c(v)&0
		\end{pmatrix},~
		\Psi^{-1}c(v)\Psi=\begin{pmatrix}
			0&\mathscr E_\link c(v)\\-\mathscr E_\link c(v)&0
		\end{pmatrix}
	\end{equation}
	if  $v$ is orthogonal to $\partial_r$.
	
	With respect to the even/odd grading of differential forms, we have 
	$$D^\dR=\begin{pmatrix}
		&D^{\dR,-}\\D^{\dR,+}&
	\end{pmatrix},$$
	where $D^{\dR,-}\colon \Omega^{\mathrm{odd}} \mathbb F \to \Omega^{\mathrm{even}} \mathbb F$ and 
	$D^{\dR,+}\colon \Omega^{\mathrm{even}} \mathbb F \to \Omega^{\mathrm{odd}} \mathbb F$. Let us define
	\begin{equation}\label{eq:linkOperator}
		P\coloneqq \begin{pmatrix}
			c_0&d^*&&&\\
			d&c_1&&&\\
			&&\ddots&&\\
			&&& c_{k-2}&d^*\\
			&&& d&c_{k-1}
		\end{pmatrix}
	\end{equation}
	where $d$ is the de Rham differential on $\Omega^\ast \link$, $d^\ast$ is the adjoint of $d$, and  \[ c_p=(-1)^p\big(p-\frac{k-1}{2}\big). \] 

	The  operator $P$ above  differs from the de Rham operator of $\link$ by a bounded endomorphism. A direct  computation shows that 
	\begin{equation}\label{eq:fiberdeRham1}
		\Psi_{\mathrm{odd}}^{-1}D^{\dR,+}\Psi_{\mathrm{even}}=\frac{\partial}{\partial r}+\frac{1}{r}P\colon C^\infty((0,\infty),\Omega^* \link)\to C^\infty((0,\infty),\Omega^* \link),
	\end{equation}
	and
	\begin{equation}\label{eq:fiberdeRham2}
		\Psi_{\mathrm{even}}^{-1}D^{\dR,-}\Psi_{\mathrm{odd}}=-\frac{\partial}{\partial r}+\frac{1}{r}P\colon C^\infty((0,\infty),\Omega^* \link)\to C^\infty((0,\infty),\Omega^* \link).
	\end{equation}
Equivalently,
\begin{equation}\label{eq:deRhamAfterConj}
	\Psi^*D^{\dR}\Psi=\begin{pmatrix}
		0&-\frac{\partial}{\partial r}\\
		\frac{\partial}{\partial r}&0
	\end{pmatrix}+\frac 1 r\begin{pmatrix}
	0&P\\P&0
\end{pmatrix}.
\end{equation}
See \cite[Section 5]{BruningSeeley} and \cite[Proposition 5.3]{LeschTopology}.

	We now turn to the boundary condition $B$. On a codimension one face $F_j$ of $\mathbb{F}$, the boundary condition $B$ is defined by
	$$ \gamma_j \sigma = -\sigma, \quad \text{where} \quad \gamma_j \coloneqq \mathscr E\overbar c(\nu_j)c(\nu_j'). $$
	Here, $\nu_j$ is the unit inner normal to $F_j$, $\nu_j'$ is the unit inner normal to the corresponding face $F_j'$ of $\mathbb{F}'$, and $\mathscr E$ is the  even-odd grading operator on $\Bigwedge^*T\mathbb{F}$.
	
	Note that the operator $\gamma_j$ preserves the even-odd grading of differential forms, that is, $\gamma_j$ commutes with $\mathscr E$. Under the unitary transformation $\Psi$, the boundary condition $B$ induces a boundary condition $B_{\mathbb{L}}$ on the bundle $(\Bigwedge^*T\mathbb{L})^{\oplus 2}$ over the link $\mathbb{L} = \mathbb{F} \cap \mathbb{S}^{k-1}$. The link $\mathbb{L}$ is itself a polyhedral manifold of dimension $k-1$. Observe that, for a codimension one face $F_j$ of $\fiber$,  it unit inner normal vector $\nu_j$ is orthogonal to the radial vector $\partial_r$. Consequently, the  operator $\gamma_j$ commutes with $\overbar c(\partial_r)$. Together with the fact that $\gamma_j$ commutes with $\mathscr E$,  it follows from the matrix representation in \eqref{eq:cylindricalClifford1} that the boundary condition $B$ splits into two identical boundary conditions on the two copies of $\Bigwedge^*T\mathbb{L}$. We denote this induced boundary condition on the faces of the link by $B_{\mathbb{L}}$.
	
	 Although  $\nu_j$ is orthogonal to $\partial_r$, $\nu_j’$ need not be orthogonal to $\partial_r$. Consequently, the induced boundary condition $B_{\mathbb{L}}$ may not respect the  even-odd grading $\mathscr E_{\mathbb{L}}$ on $\Bigwedge^*T\mathbb{L}$.

	Let us denote by $D^\dR_{B}$ the de Rham operator of $\fiber$  with initial domain consisting of smooth differential forms that satisfy the boundary condition $B$ and are supported away from the vertex of $\fiber$ (i.e., forms in $C^\infty_c(\fiber \setminus \{0\}, \Bigwedge^*T\fiber; B)$). Similarly, let $P_B$ be the operator $P$ on the link $\link$ (see \eqref{eq:deRhamAfterConj}) subject to the boundary condition $B_\link$.

	The following lemma characterizes when $D^\dR_{B}$ is essentially self-adjoint in terms of the spectrum of the operator $P_B$.  
	\begin{lemma}[{cf. \cite[Theorem 3.1]{BruningSeeley}}]\label{lemm:>=1/2}
		Assume that  $P_B$ is essentially self-adjoint.  Then  $D^\dR_{B}$ is essentially self-adjoint if and only if $|P_B|\geq 1/2$.
	\end{lemma}
	\begin{proof}
		Set $E_\pm(D^\dR_{B} )\coloneqq \ker((D^\dR_{B})^*\mp i)$, where $(D^\dR_{B})^*$ is the adjoint of $D^\dR_{B}$ (as an unbounded operator). The von Neumann deficiency indices theorem states that $D^\dR_{B}$ is essentially self-adjoint if and only if 
		\[ E_+(D^\dR_{B} )=E_-(D^\dR_{B} )=0. \]
		
		Since $\link$ is compact and $P_B$ is an essentially self-adjoint elliptic operator with local boundary condition, its closure has compact resolvent. Hence $L^2(\link,\bigwedge^*T^*\link)$ admits an orthonormal basis $\{\phi_\lambda\}$, where each $\phi_\lambda$ is the eigenvector of $P_B$ with eigenvalue $\lambda$.
		
		Suppose that $\varphi=\varphi_{\text{even}}\oplus \varphi_{odd}$ lies in the kernel of $(D^\dR_{B})^*- i$. Let us denote  $\psi_0=\Psi_{\mathrm{even}}^{-1}(\varphi_{\text{even}})$ and 
		$\psi_1=\Psi_{\mathrm{odd}}^{-1}(\varphi_{\text{odd}})$. We have
		\begin{equation}\label{eq:deficiency}
			\Big(\begin{psmallmatrix}
				0&-\frac{\partial}{\partial r}\\
				\frac{\partial}{\partial r}&0
			\end{psmallmatrix}+\frac{1}{r}\begin{pmatrix}
				0&P\\P&0
			\end{pmatrix}- i\Big)\begin{pmatrix}\psi_0\\\psi_1
			\end{pmatrix}=0.
		\end{equation}
	If we write
		$$\psi_0=\sum_\lambda\psi_{0,\lambda}(r)\phi_\lambda \textup{ and } \psi_1=\sum_\lambda\psi_{1,\lambda}(r)\phi_\lambda, $$
then Equation \eqref{eq:deficiency} then splits into the following system of
ordinary differential equations for each eigenvector  $\phi_\lambda$:  
		\begin{equation}\label{eq:psi0,1def}
			\begin{cases}
				\displaystyle	\frac{\partial}{\partial r} \psi_{0,\lambda}+\frac \lambda r\psi_{0,\lambda}=i\psi_{1,\lambda},\\
				\displaystyle	-\frac{\partial}{\partial r}\psi_{1,\lambda}+\frac \lambda r\psi_{1,\lambda}=i\psi_{0,\lambda}.
			\end{cases}
		\end{equation}
		It follows that 
		$$\begin{cases}
			\displaystyle -\frac{\partial^2}{\partial r^2}\psi_{0,\lambda}+\frac{\lambda}{r^2}\psi_{0,\lambda}+\frac{\lambda^2}{r^2}\psi_{0,\lambda}+\psi_{0,\lambda}=0, 	\vspace{.3cm}\\
			\displaystyle i\psi_{1,\lambda}=\frac{\partial}{\partial r}\psi_{0,\lambda}+\frac{\lambda} {r}\psi_{0,\lambda}, 
		\end{cases}$$
		where the solution to the first differential equation consists of modified Bessel functions. More precisely, we have
		\begin{equation}\label{eq:solutionKI}
			\begin{cases}
				\psi_{0,\lambda}=c_1\sqrt r\cdot K_{\lambda+ 1/2}(r)+c_2\sqrt r\cdot I_{\lambda+ 1/2}(r)\\
				\psi_{1,\lambda} = i c_1\sqrt{r}\cdot K_{\lambda-1/2}(r) - i c_2\sqrt{r}\cdot I_{\lambda-1/2}(r),
			\end{cases}	
		\end{equation}
		where $I_\nu$ and $K_\nu$ are modified Bessel functions of the first and the second kind, respectively. Since $I_\nu$ grows exponentially as $r\to \infty$ and $K_\nu$ decays exponentially as $r\to \infty$,  in order for $\psi_{0,\lambda}\oplus \psi_{1,\lambda}$ to be in  $L^2$, we have 
		\[
			\begin{cases}
				\psi_{0,\lambda}=c_1\sqrt r\cdot K_{\lambda+ 1/2}(r)\phi_\lambda\\
				\psi_{1,\lambda} = i  c_1\sqrt{r}\cdot K_{\lambda-1/2}(r)\phi_\lambda. 
			\end{cases}	
		\]
		 Moreover, since 
		 \[ K_{\nu}(r) \sim \begin{cases}
		 	\ln r & \textup{ if } \nu = 0, \\
		 	r^{-|\nu|}  & \textup{ if } \nu \neq  0, 
		 \end{cases}\] 
		 it follows that $\psi_{0,\lambda}\oplus \psi_{1,\lambda}$ is  an $L^2$ solution if and only if we have  $-1/2<\lambda<1/2$ (cf. \cite[Lemma 4.2]{LeschTopology}). This shows that $E_+(D^\dR_{B} ) = 0$ if and only if $|P_B|\geq 1/2$. The same argument shows that $E_-(D^\dR_{B} ) = 0$ if and only if $|P_B|\geq 1/2$.  This finishes the proof. 
	\end{proof}

\subsection{Essential self-adjointness of $D^\dR_{B}$  for two dimensional model spaces}
	
In this subsection, we investigate the essential self-adjointness of $D^\dR_{B}$  on $\fiber$ when $\dim \fiber=2$.

	\begin{lemma}\label{lemma:essensa-jumpanglewithf}
		Let $\fiber $ and $\fiber'$ be two sectors in $\R^2$ with angles $\alpha$ and $\beta$. Let $B$ be the boundary condition on $\Bigwedge^*T\fiber = \Bigwedge^\ast \mathbb R^2$ along each edge of $\fiber$  given by
		$$\mathscr E\overbar c(\nu_k)c(\nu_k')\omega=-\omega$$
		for $k=1,2$ as in Definition \ref{def:boundaryCondition}, where $\nu_k$ are unit inner normal vectors of the edges of $\fiber$ and  $\nu'_k$ are unit inner normal vectors of the corresponding edges of $\fiber'$ (see Figure \ref{fig:SNfSM}).	
		Let $D^\dR_{B}$ be the de Rham operator acting on $\Bigwedge^*T\fiber$ with the boundary condition $B$. Suppose that the angle $\alpha$ of $\fiber$ is  less than or equal to $\pi$. Then $D^\dR_{B}$ is essentially self-adjoint if and only if $\alpha\leq\beta$ and $\alpha+ \beta \leq 2\pi$.
	\end{lemma}
	\begin{proof}
		
		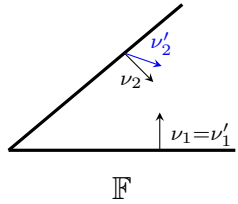
\begin{figure}[h]
			\begin{tikzpicture}[scale=1]
				\draw[very thick,black] (0,0) -- (3,0);
				\draw[very thick,black] (0,0) -- ({3*sin(50)},{3*cos(50)});
				\draw [-stealth] (2,0) --  (2,0.5) node[anchor=north west] {$\scriptstyle \nu_1=\nu_1'$};
				\draw [-stealth] ({2*sin(50)},{2*cos(50)}) --  ({2*sin(50)+0.5*cos(40)},{2*cos(50)-0.5*sin(50)}) 
				node[anchor=east] {$\scriptstyle \nu_2$};
				\draw [blue,-stealth] ({2*sin(50)},{2*cos(50)}) --  ({2*sin(50)+0.5*cos(10)},{2*cos(50)-0.5*sin(20)})
				node[anchor=south] {$\scriptstyle \nu_2'$};
				\filldraw (1.5,-0.5) node {$\fiber$};
			\end{tikzpicture}
			\caption{The boundary conditions at the two edges of $\fiber$.}
			\label{fig:SNfSM}
		\end{figure}
		By applying a rotation on $\fiber'$ if necessary, we may assume $\nu_1=\nu_1'$.  Then the vector $\nu_2'$ differs from $\nu_2$ by a counterclockwise rotation through the angle $(\beta - \alpha)$.
		
		By line \eqref{eq:deRhamAfterConj}, the de Rham operator $D^\dR$ is conjugate to
				\begin{equation}
		 	\Psi^*D^{\dR}\Psi = 	\begin{psmallmatrix}
				0&-\frac{\partial}{\partial r}\\
				\frac{\partial}{\partial r}&0
			\end{psmallmatrix}+\frac{1}{r}\begin{pmatrix}
				0&P\\P&0
			\end{pmatrix}.
		\end{equation}
		Here, the operator $P$ is given by 
		$$P=\begin{pmatrix}
			-1/2&-\frac{\partial}{\partial \theta}\\\frac{\partial}{\partial \theta}&-1/2
		\end{pmatrix}=D^\dR_{\link}-\frac 1 2$$
		  with respect to the splitting $\Omega^*\link \cong \Omega^0\link \oplus \Omega^1\link$,  where  $D^\dR_{\link}$ is the de Rham operator on $\link = [0, \alpha]$. See line \eqref{eq:linkOperator}.

		By a direct computation, the boundary condition $B_\link$ induced by $B$ is as follows. If 
		$$\Psi^{-1} w=\begin{pmatrix}
			\varphi_0\\\varphi_1
		\end{pmatrix}\textup{, where }\varphi_i=\varphi_{i,0}+\varphi_{i,1}d\theta,$$ then the boundary condition $B_{\link}$ requires
		\begin{equation}\label{eq:bdryConditionBL}
			\varphi_{i,1}(0)=0,\textup{ and }
						-\varphi_{i,0}(\alpha)\sin\frac{\beta - \alpha}{2}+\varphi_{i,1}(\alpha)\cos\frac{\beta - \alpha}{2}=0,~\forall i=0,1.
		\end{equation}
	
	Suppose that $\phi=\phi_0+\phi_1d\theta$ satisfying $B_\link$ above is an eigenvector of $P$ with eigenvalue $\lambda-\frac 1 2$. Therefore, we have $D_{\link}^\dR\phi=\lambda\phi$, namely
				\begin{equation}
						-\phi'_1=\lambda \phi_0,\text{ and }\phi'_0=\lambda \phi_1.
					\end{equation}
				Hence $\phi''_1=-\lambda^2\phi_1$. By  the boundary condition $\phi_1(0)=0$, we see that  $\phi_1(\theta)=C\cdot \sin(\lambda\theta)$ for some constant $C$. It follows that $\phi_0(\theta)=- C\cdot \cos(\lambda\theta)$. The boundary condition at $\theta=\alpha$ implies that
				\begin{equation}
						\sin(\lambda\alpha)\cos\frac{\beta - \alpha}{2}+\cos(\lambda\alpha)\sin\frac{\beta - \alpha}{2}=0,
					\end{equation}
				that is, $\sin(\lambda\alpha+\frac{\beta - \alpha}{2})=0$. Therefore, the spectrum of the operator $D_{\link}^\dR$ with respect to $B_\link$ is 
				$$\Big\{-\frac{\beta-\alpha}{2\alpha}+\frac{m\pi}{\alpha}\Big\}_{m\in\Z}. $$
				Hence the spectrum of $P_B=-1/2+D_{\link}^\dR$ with the boundary condition $B_\link$ is given by
				$$ \Big\{-\frac{\beta}{2\alpha}+\frac{k\pi}{\alpha}\Big\}_{k\in\Z} =\left\{\cdots,-\frac 1 2-\frac{\beta-\alpha}{2\alpha},\frac 1 2+\frac{2\pi-\beta-\alpha}{2\alpha},\cdots\right\}.$$
		In particular, $P_B$ has a spectral gap $\geq 1/2$ if and only if $\beta\geq\alpha$ and $\alpha+ \beta\leq 2\pi$.  By Lemma~\ref{lemm:>=1/2}, this is equivalent to the essential
		self-adjointness of $D_B^{\dR}$. This finishes the proof. 
	\end{proof}

A direct computation shows that the angle comparison condition:
\[ \alpha \leq \pi, \alpha\leq \beta \textup{ and } \alpha + \beta\leq 2\pi \]
in Lemma \ref{lemma:essensa-jumpanglewithf} is equivalent to the following inner product comparison condition:
\[ \alpha \leq \pi \textup{ and } 	\langle \nu_1,\nu_2\rangle\leq \langle \nu'_1,\nu'_2\rangle.\]
This shows that the key input from the sector  $\fiber'$ is the unit inner normal vectors along its edges, rather than the geometric angle itself.  
In fact,  boundary condition $B$ along the edges of $\fiber$:
$$\mathscr E\overbar c(\nu_k)c(\nu_k')\omega=-\omega$$
is well-defined as long as we specify some unit vectors
$\nu'_k$. Here the vectors $\nu_k$ remain the unit inner normal vector of the edges of $\fiber$, but $\nu'_k$ need not be the unit inner normal vectors of an actual sector $\fiber'$. In fact, we may omit $\fiber'$ from the discussion entirely and focus solely on $\fiber$, its unit inner normal vectors, and a chosen auxiliary unit vector along each edge.

The above observation will play a crucial role in the computation of the Fredholm index of the twisted Dirac operator $D_B$ later. Let us summarize the above discussion by the following lemma. 

\begin{lemma}\label{lemma:ess-sa-innerproductcomparison}
			Let $\fiber $ be a sectors in $\R^2$ with angle $0<\alpha\leq\pi$.
			Suppose that $\nu_1$ and $\nu_2$ are the inner normal vectors of $\fiber$, and $\nu_1'$ and $\nu_2'$ are two unit vectors in $\mathbb R^2$.
			 Let $B$ be the boundary condition on $\Bigwedge^*T\fiber = \Bigwedge^\ast \mathbb R^2$ over each edge given by
	$$\mathscr E\overbar c(\nu_k)c(\nu_k')\omega=-\omega$$
	for $k=1,2$.	
	Let $D^\dR_{B}$ be the de Rham operator acting on $\Bigwedge^*T\fiber$ with the boundary condition $B$. Then $D^\dR_{B}$ is essentially self-adjoint if and only if 
 \[			\langle \nu_1,\nu_2\rangle\leq \langle \nu_1',\nu_2'\rangle. \]
 Moreover, $|P_B| > 1/2$ if and only if $\langle \nu_1,\nu_2\rangle < \langle \nu_1',\nu_2'\rangle$, where $P_B$ is the operator $P$ along the link (subject to the boundary condition $B_\link$) as in the proof of Lemma \ref{lemma:essensa-jumpanglewithf}.  
\end{lemma}

\subsection{Essential self-adjointness for three dimensional model spaces}
 In this subsection, we investigate the essential self-adjointness of
$D_B^{\dR}$ on $\fiber$ in the case $\dim\fiber=3$.

\begin{figure}[h]
	\begin{tikzpicture}[line cap=round, line join=round, scale=0.7]
		
		\coordinate (V) at (0,0);
		
		\coordinate (TL) at (2.0, 4.2);   
		\coordinate (BL) at (1.8, 2.4);   
		\coordinate (TR) at (4.5, 3.2);   
		\coordinate (BR) at (4.0, 1.6);   

		\coordinate (Center1) at (3.1, 2.85);
		
		\def\ext{1.15}
		\coordinate (ExtTL) at ($ (V)!\ext!(TL) $);
		\coordinate (ExtTR) at ($ (V)!\ext!(TR) $);
		\coordinate (ExtBL) at ($ (V)!\ext!(BL) $);
		\coordinate (ExtBR) at ($ (V)!\ext!(BR) $);
		
		\draw[thick, densely dashed, gray!80] (V) -- (TR);
		\draw[thick, densely dashed, gray!80] (TR) -- (ExtTR); 
		
		\fill[gray!30]
		(TL) to[bend left=8] (TR)
		-- (BR)
		to[bend right=8] (BL) 
		-- cycle;
		
		\draw[thick, black] (TL) to[bend left=8] (TR);
		\draw[thick, black] (BL) to[bend left=8] (BR);
		\draw[thick, black] (TL) -- (BL);
		\draw[thick, black] (TR) -- (BR);
		
			\node at (Center1) {$\link$}; 
		
		\draw[thick, black] (V) -- (TL);
		\draw[thick, black] (V) -- (BL);
		\draw[thick, black] (V) -- (BR);
		
		\draw[thick, black] (TL) -- (ExtTL);
		\draw[thick, black] (BL) -- (ExtBL);
		\draw[thick, black] (BR) -- (ExtBR);
		
	\end{tikzpicture}
	\caption{A two dimensional link $\link$ of a three dimensional cone.}
	\label{fig:link2}
\end{figure}
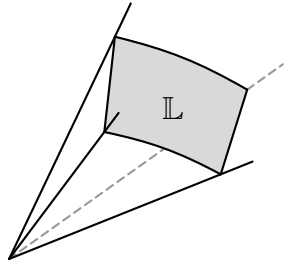

Let $\fiber$ be a convex polyhedral cone in $\R^3$, bounded by planes through the
origin. The boundary condition $B$ on $\Bigwedge^*T\fiber$ over $\fiber$ is given as
$$\mathscr E \overbar c(\nu_k)c(\nu_k')w=-w \textup{ on each face }F_k,$$
where $\nu_k$ is the inner normal vector on $F_k$, and $\nu_k'$ is some constant unit vector.
The metric on $\fiber$ is conical, and the link $\link = \fiber \cap \sph^2$ is a convex spherical polygon. See Figure \ref{fig:link2}. The boundary condition $B$ induces a boundary condition $B_\link$ on $\link$.
The dihedral angles of $\link$ are precisely the corresponding  dihedral angles of $\fiber$.

\
\begin{lemma}\label{lemma:ess-saSphericalLink}
With the above notation, let $D^\dR_\link$ be the de Rham operator  on $\Bigwedge^*T\link$, and $B_\link$ the boundary condition on $\link$ induced by $B$. Assume that dihedral angles of $\link$ are less than $\pi$. If 
	$$\langle \nu_i,\nu_j\rangle\leq \langle\nu_i',\nu_j'\rangle$$
	for any adjacent pair of faces $F_i,F_j$, then $D^\dR_\link$ subject to the boundary condition $B_\link$ is essentially self-adjoint.
\end{lemma}
\begin{proof}
	By the discussion after Definition 
	\ref{def:ess-sa},   it suffices to show that the de Rham operator $D^\dR_\link$ subject to the boundary condition $B_\link$ is locally essentially self-adjoint at every point $x \in \link$.
	
	\textbf{Interior of $\link$.} If $x$ lies in the interior of $\link$, local essential self-adjointness follows from classical elliptic regularity theory.
	
		\textbf{Edge of $\link$.} Suppose $x$ lies in the interior of a codimension-one face of $\link$, say $x\in F_i \cap \link$ for some fixed $i$.  Recall that $\Bigwedge^*T\fiber = S(\R^3 \oplus \R^3)$. A rotation on the second copy of  $\mathbb{R}^3$ induces a unitary operator $U$ on $S(\R^3 \oplus \R^3)$ such that $U^* D^\dR U = D^\dR$. Moreover, by fixing an orthogonal identification between the two copies of  $\mathbb{R}^3$ in $S(\mathbb{R}^3 \oplus \mathbb{R}^3)$, we can naturally view both $\nu'_i$ and $\nu_i$ as elements of the same ambient space $\mathbb{R}^3$. Now after applying a unitary transformation on the second copy of $\mathbb R^3$, we may assume without loss of generality that $\nu'_i=\nu_i$. In this case, the induced boundary condition $B_\link$ preserves the even-odd grading of the differential forms $\Bigwedge^*T\link$ of $\link$. In fact, $B_\link$ becomes the standard absolute boundary condition (see Example \ref{example:deRham}) for the de Rham operator $D^\dR_\link$. It is a  classical result that $D^\dR_\link$ is locally essentially self-adjoint at $x$ in this case.

	\textbf{Vertex of $\link$.}	Suppose $x$ is a vertex of $\link$, say $x = F_i \cap F_j \cap \link$ for some fixed $i$ and $j$. By applying a unitary transformation on the second copy of $\mathbb R^3$	 if necessary, we may assume without loss of generality that $\nu'_i = \nu_i$ and that $\nu_j, \nu'_j$ are tangential to $\link$ at the vertex $x$.
	
	Let $\link_r = \fiber \cap \mathbb S^2_r$ denote the link at radius $r$. Let $\mathbb G_r \coloneqq  \tancone_{rx}\link_r$ be the tangent cone of $\link_r$ at $rx$, which is a sector in $T_{rx}\link_r$. Let $U\subset \mathbb G = \mathbb G_1$ be a sufficiently small neighborhood of the origin in $\tancone_x\link$. The exponential map $\exp_x \colon U \to \link $ takes  $U$ to the corresponding spherical sector in $\link$. Define 
	\[  \Phi\colon (1-\varepsilon, 1+\varepsilon) \times U \to \fiber, 
	\qquad \Phi(r, z) = r\exp_{x}(z).\]
	Under the canonical identification of the tangent space $T_{rx}\fiber$ with $\R^3$, the tangent space $T_{rx} \mathbb S^2_r$ can be naturally viewed as the plane passing through $rx$ and orthogonal to the edge $F_i \cap F_j$ . In particular, $(1-\varepsilon, 1+\varepsilon) \times U$ can be naturally viewed as subspace of $\fiber$. Therefore, $\Phi$ is a smooth map from a subspace of  $\fiber$ to $\fiber$.

	Because the bundle $S(\R^3 \oplus \R^3)$ is trivial on $\fiber$, its pullback bundle under the map $\Phi$ is canonically identified with the trivial bundle $S(\R^3 \oplus \R^3)$ itself. Consequently, the boundary condition $B$ along codimension one faces $F_i$ and $F_j$ of  $\fiber$ pulls back exactly to the same constant boundary condition. In particular, the induced boundary condition $B_{\mathbb G}$ on the boundary rays of $\mathbb{G}$ is tangential to $\mathbb{G}$, since the vectors $\nu_i$, $\nu'_i$, $\nu_j$, and $\nu'_j$ all lie within the plane of $\mathbb{G}$.
	
	The operator $D^\dR_\link$ on the spherical link $\link$ pulls back to a differential operator on the flat sector $\mathbb G$. For sufficiently small $\rho$, an explicit computation in polar coordinates $(\rho, \theta)$ for the sector  yields:
	$$\begin{aligned} \Phi^\ast D^\dR_\link \Phi &= \overbar c(\partial_\rho) \Big(\frac{\partial}{\partial \rho} + \frac{1}{2}\left(\cot \rho - \rho^{-1}\right) \Big) + \frac{1}{\sin \rho} \overbar c(\partial_\theta) \frac{\partial}{\partial \theta} \\ &= D^\dR_{\mathbb G} + \frac{1}{2} \left(\cot \rho - \rho^{-1}\right) \overbar c(\partial_\rho) + \Big(\frac{1}{\sin \rho} - \frac{1}{\rho}\Big) \overbar c(\partial_\theta) \frac{\partial}{\partial \theta} \end{aligned}$$
	where $D^\dR_{\mathbb G}$
	is the standard de Rham operator on the flat sector $\mathbb{G}$. 
	
	By assumption,  dihedral angles of $\link$ are less than $\pi$ and  $\langle \nu_i, \nu_j \rangle \leq \langle\nu_i', \nu_j'\rangle$ holds. The canonical identification above preserves these properties for $\mathbb{G}$.  Therefore, by Lemma \ref{lemma:ess-sa-innerproductcomparison}, the operator $D^\dR_{\mathbb G}$ is essentially self-adjoint on the flat sector $\mathbb{G}$.

	 To conclude that the operator $\Phi^\ast D^\dR_\link \Phi$ is locally  essentially self-adjoint at the vertex $x\in \link$, we apply the Kato--Rellich perturbation theorem (see for example \cite[Chapter 33, Section 4]{MR1892228}).\footnote{Recall the Kato--Rellich theorem: Suppose $\mathbf{D}$ is essentially self-adjoint acting on a Hilbert space $H$, and let $\mathbf{T}$ be a symmetric operator whose domain contains $\dom(\mathbf{D})$. If there exist constants $a \geq 0$ and $b < 1$ such that for all $u \in \dom(\mathbf{D})$, $\|\mathbf{T}u\|^2 \leq a\|u\|^2 + b\|\mathbf{D}u\|^2$, then $\mathbf{D} + \mathbf{T}$ is also essentially self-adjoint. Furthermore, the domain of the closure $\overbar{\mathbf{T}}$ contains $\dom(\overbar{\mathbf{D}})$, and $\overbar{\mathbf{D}} + \overbar{\mathbf{T}}$ is self-adjoint on $\dom(\overbar{\mathbf{D}})$.} Note that the  coefficients $\left(\cot \rho - \rho^{-1}\right)$ and $\left(\frac{1}{\sin \rho} - \frac{1}{\rho}\right)$ vanish  as $\rho \to 0$. Therefore, by restricting to a sufficiently small neighborhood of the origin of $\mathbb G$, we have  
	 \[ \big|\cot \rho - \rho^{-1}\big| \leq \delta \textup{ and } \big|\frac{1}{\sin \rho} - \frac{1}{\rho}\big| \leq \delta, \] for an arbitrarily small $\delta > 0$. We can smoothly extend these coefficients to compactly supported functions $\varphi_1$ and $\varphi_2$ defined on all of $\mathbb{G}$,  satisfying the uniform bounds $|\varphi_1| \leq 2\delta$ and $|\varphi_2| \leq 2\delta$. Let us define the  operator:
	 $$\mathcal{D} = D^\dR_{\mathbb G} + \frac{1}{2} \varphi_1 \overbar c(\partial_\rho) + \varphi_2 \overbar c(\partial_\theta) \frac{\partial}{\partial \theta}.$$ 
	 As long as $\delta$ is chosen to be sufficiently small, the assumptions of the Kato-Rellich theorem are satisfied for $\mathcal{D}$ relative to $D^\dR_{\mathbb G}$. It immediately follows that $\mathcal{D}$ is essentially self-adjoint on $\mathbb{G}$. In particular, $\mathcal{D}$ is locally essentially self-adjoint at the origin of $\mathbb{G}$.
	 
	  Because $\Phi^\ast D^\dR_\link \Phi = \mathcal{D}$ in a neighborhood of the origin, $\Phi^\ast D^\dR_\link \Phi$ is also locally essentially self-adjoint at the origin. Finally, since the exponential map $\Phi$ is smooth and naturally maps $H^1(\link, S(\R^3 \oplus \R^3); B)$ to $H^1(\mathbb{G}, S(\R^3 \oplus \R^3); B)$, we conclude that the original operator $D^\dR_\link$ is locally essentially self-adjoint at the vertex $x$. This completes the proof. 
\end{proof}

Let $P$ be the link operator as in \eqref{eq:linkOperator}. Since $P$ differs from the $D^\dR_\link$ by a bounded order zero operator. The essentially self-adjointness of $P$ (subject to the boundary condition $B_\link$) is equivalent to that of $D^\dR_\link $ (subject to the boundary condition $B_\link$).
We now estimate the spectral gap of $P_{B_\link}$. 

\begin{lemma}\label{lemma:spectrumSphericalLink}
	Let $\fiber\subset\mathbb R^3$ be a convex polyhedral cone. The boundary condition $B$ on $\Bigwedge^*T\fiber$ over $\fiber$ is given as
	$$\mathscr E \overbar c(\nu_k)c(\nu_k')w=-w \textup{ on each face }  F_k,$$
	where $\nu_k$ is the inner normal vector on $F_k$, and $\nu_k'$ is some constant unit vector. Let $P$ be the link operator in
	\eqref{eq:linkOperator}, and let $B_\link$ be the boundary condition on
	the link $\link$ induced by the boundary condition $B$ on $\fiber$. Assume that
	$P_{B_\link}$ is essentially self-adjoint. Then $P_{B_\link}$ is invertible
	and satisfies $|P_{B_\link}|\geq \sqrt 2/2$.	
\end{lemma}
\begin{proof}
	By Lemma \ref{lemma:ess-saSphericalLink}, the operator $P$ on the spherical link $\link$ with boundary condition $B_\link$ is essentially self-adjoint with domain $H^1(\link,\Bigwedge^*T\link;B_\link)$, provided that the inner product comparison holds. 
	Since $P_{B_\link}$ is essentially self-adjoint, it suffices to estimate its spectrum on  $C^\infty_{00}(\link,\Bigwedge^*T\link;B_\link)$ consisting of smooth forms satisfying
	the boundary condition $B_\link$, supported away from the vertices of
	$\link$.

	Let $D = D^{\dR}_\link$ the de Rham operator acting on $\Bigwedge^*\mathbb L$ over the link $\link$.  The operator $P$ differs from $D$ by the diagonal zeroth-order term
	$(-1)^p(p-1)$ on $p$-forms. See \eqref{eq:linkOperator}. Let $\overbar c$ and $c$ be the  left and right Clifford multiplications  on  $\Bigwedge^*\mathbb L$. A direct computation gives 
		\begin{equation}
			P=D+\frac 1 2 \sum_i\overbar c(e_i)c(e_i),
		\end{equation}
		where $\{e_i\}$ is a local orthonormal frame of $T\link$. Define a new connection $\widehat\nabla$ on $\Bigwedge^*\mathbb L$ by
		\begin{equation}\label{eq:modconnection}
			\widehat\nabla_X\coloneqq \nabla_X+\frac 1 2 c(X),
		\end{equation}
		for $X\in T\link$, where $\nabla$ is the canonical connection on $\Bigwedge^*T\mathbb L$ induced by the Levi-Civita connection of $\link$. Then 
		$$P=\sum_{i} \overbar c(e_i)\widehat\nabla_{e_i}.$$

		Now we derive a Lichnerowicz formula for  $P^2$. At a point $x\in \link$, choose a local orthonormal frame $\{e_i\}$ of $T\link$ such that $[e_i,e_j]=0$ at $x$. Since the Clifford multiplications $\overbar c$ and $c$ commute, we have
		$$P^2=-\sum_{i}\widehat\nabla_{e_i}\widehat\nabla_{e_i}+\sum_{i<j}\overbar c(e_i) \overbar c(e_j)\widehat R_{e_i,e_j},$$
		where $\widehat R$ is the curvature operator of $\widehat\nabla$:
		$$\widehat R_{e_i,e_j}=\widehat \nabla_{e_i}\widehat\nabla_{e_j}-\widehat \nabla_{e_j}\widehat\nabla_{e_i}.$$
		Since the Levi-Civita connection on $\link$ is torsion free, we obtain that
		\begin{align*}
			\widehat R_{e_i,e_j}=&~(\nabla_{e_i}+\frac 1 2 c(e_i))(\nabla_{e_j}+\frac 1 2 c(e_j))-(\nabla_{e_j}+\frac 1 2 c(e_j))(\nabla_{e_i}+\frac 1 2 c(e_i))\\
			=&~\nabla_{e_i}\nabla_{e_j}-\nabla_{e_j}\nabla_{e_i}+\frac 1 2 c(e_i)c(e_j).
		\end{align*}
		Since $\link$ has constant sectional curvature $1$ with scalar curvature $2$, we obtain that
		$$\sum_{i<j}\overbar c(e_i) \overbar c(e_j)(\nabla_{e_i}\nabla_{e_j}-\nabla_{e_j}\nabla_{e_i})=\frac{1}{2}-\frac 1 2\sum_{i<j} \overbar c(e_i) \overbar c(e_j) c(e_i)c(e_j), $$
		cf. the curvature identity \eqref{eq:lichnerowicz}, applied with $f=\id \colon \link \to \link$. 
		Hence the Lichnerowicz formula for $P^2$ reads 
		\begin{align*}
			P^2=&~\widehat\nabla^*\widehat\nabla+\sum_{i<j}\overbar c(e_i) \overbar c(e_j)\widehat R_{e_i,e_j}\\
			=&~\widehat\nabla^*\widehat\nabla+\frac{1}{2}-\frac 1 2 \sum_{i<j} \overbar c(e_i) \overbar c(e_j) c(e_i)c(e_j)   +\frac 1 2 \sum_{i<j} \overbar c(e_i) \overbar c(e_j) c(e_i)c(e_j)\\
			=&~\widehat\nabla^*\widehat\nabla+\frac{1}{2}
		\end{align*}
		
		
		Let $
		\link_k=F_k\cap\link$
		be an edge of $\link$. For a differential form $\varphi$ in $C_{00}^\infty(\mathbb L,\Bigwedge^*\mathbb L;B_\link)$, 
		 we have
		\begin{equation}\label{eq:deRhamestimate}
			\begin{split}
				&\int_{\mathbb L}\langle P \varphi,  P \varphi\rangle \\
				=&
				\int_{\mathbb L}\langle P^2\varphi,\varphi\rangle+\sum_k\int_{\link_k}\langle \overbar c(\nu_k)P\varphi,\varphi\rangle\\ 
				=&\int_{\mathbb L}\langle \big(\widehat\nabla^*\widehat\nabla+\frac{1}{2}\big)\varphi,\varphi\rangle+\sum_k\int_{\link_k}\langle \overbar c(\nu_k)P\varphi,\varphi\rangle\\
				=&\int_{\mathbb L}|\widehat\nabla\varphi|^2+ 
				\frac{1}{2}|\varphi|^2+\sum_{k}\int_{\link_k}\langle \overbar c(\nu_k)\overbar c(e^k)\widehat\nabla_{e^k}\varphi,\varphi\rangle,
			\end{split}
		\end{equation}
		where $e^k$ is a unit vector field of $T\link_k$. Note that $e^k$ is an orthonormal basis of $T\link_k$, since the edge $\link_k$ is one dimensional. 
		
	Define
		\begin{equation}\label{eq:hatDpartial}
			\widehat D^\partial_k= \overbar c(\nu_k) \overbar c(e^k)\widehat\nabla_{e^k}=\overbar c(\nu_k)\left(
			\overbar c(e^k)\nabla_{e^k}+\frac 1 2\overbar c(e^k)c(e^k).			
			\right) 
		\end{equation}
	
	It remains to show that the boundary integrals  in 
	\eqref{eq:deRhamestimate} vanish. The boundary condition $B_{\link}$ on the edge $\link_k  $ of $\link$ is induced by the boundary condition 
	$$\mathscr E \overbar c(\nu_k)c(\nu_k')w=-w,$$ where $\nu_k$ is the inner unit normal vector field of $F_k$ and $\nu_k'$ is a constant unit vector. Let $\gamma_k= \mathscr E \overbar c(\nu_k)c(\nu_k').$
		\begin{claim*} $\widehat D^\partial_k\gamma_k+\gamma_k\widehat D^\partial_k=0.$
	\end{claim*}
	Since $\gamma_k  \overbar c(\nu_k) \overbar c(e^k) = -  \overbar c(\nu_k) \overbar c(e^k) \gamma_k$, it suffices to show that $	[\widehat\nabla_{e^k},\gamma_k] = 0$. 
	In general, the unit vector $\nu_k'$ may not be tangent to the link $\link$, so we write its decomposition along $\link_k$ as
		\begin{equation}\label{eq:decompositionOfVi}
			\nu_k'=\alpha_k\partial_r+ \widehat\nu_k,
		\end{equation}
	where $\widehat\nu_k$ is orthogonal to $\partial_r$. Under the identification from  \eqref{eq:cylindricalClifford1} and \eqref{eq:cylindricalClifford2}, we have 
	\[   \gamma_k  = \begin{pmatrix}
		\mathscr E_\link \overbar c(\nu_k)\left(-\alpha_k+ c(\widehat\nu_k) \right)  & 0 \\
		0 & \mathscr E_\link \overbar c(\nu_k)\left(-\alpha_k+ c(\widehat\nu_k) \right)
	\end{pmatrix}\]
	on two copies of $\Omega^\ast\link$, 	where $\mathscr E_\link$  is the even-odd grading operator on $\Bigwedge^*T\link$. Therefore, the induced 
	the boundary condition $B_\link$ on each copy of  $\Omega^\ast\link$ along $\link_k$ is given by 
		\begin{equation}
			\mathscr E_\link \overbar c(\nu_k)\left(-\alpha_k+ c(\widehat\nu_k) \right)\varphi  = - \varphi.
		\end{equation}
		
		
	To simplify notation, we omit the subscript $k$ in $\gamma_k, \nu_k$, etc., since no ambiguity is likely to arise.   Let $\overbar \nabla$ be the Euclidean flat connection on $T\fiber$ over $\fiber$, and $\nabla^\link$ the Levi-Civita connection on $T\link$ over $\link$. Then, for any tangent vector fields $X$ and $Y$ along $\link$, 
		$$\begin{cases}\displaystyle
			\overbar\nabla_X Y=\nabla^\link_X Y-\langle X,Y\rangle\partial_r,\\
			\displaystyle\overbar\nabla_X \partial_r=X.
		\end{cases}$$  
		Since $\nu'$ is constant in $\mathbb R^3$, applying $\overbar\nabla$ to \eqref{eq:decompositionOfVi} gives 
		$$0=\overbar\nabla_X \nu' = X(\alpha )\partial_r+\alpha  X+ \nabla^\link_X\widehat\nu - \langle X,\widehat\nu \rangle\partial_r.$$
		Separating radial  and tangential components yields
		\begin{equation}\label{eq:alphaBeta}
			\begin{cases}
				X(\alpha )=\langle X,\widehat\nu \rangle\\
				-\alpha  X= \nabla^\link_X\widehat\nu
			\end{cases}
		\end{equation}
		 Since every edge of $\link$ is totally geodesic and $\nu$ is constant in $\mathbb R^3$, it follows that $c(\nu)$ is parallel along the edge $\link_k$. Equivalently, $c(\nu)$ commutes with the connection $\nabla$ on $\Bigwedge^*\mathbb L$.  Recall that $\gamma =\mathscr E_\link \overbar c(\nu )(-\alpha +  c(\widehat{\nu }))$. Applying \eqref{eq:alphaBeta}, we have 	
		\begin{align*}
			[\nabla_X,\gamma]=&~ \mathscr E_\link \overbar c(\nu) \left(-X(\alpha)+ c(\nabla^\link_X\widehat \nu)\right)\\
			=&-\mathscr E_\link \overbar c(\nu) \left( \langle X,\widehat \nu\rangle+\alpha c(X)\right).
		\end{align*}
		On the other hand, we have 
		\begin{align*}
			\frac 1 2[c(X),\gamma]=&-\frac 1 2\mathscr E_\link \overbar c(\nu)\left(
			c(X)(-\alpha +  c(\widehat{\nu }))+(-\alpha + c(\widehat{\nu}))c(X)\right)\\
			=&~\mathscr E_\link \overbar c(\nu)\left(\alpha c(X)+ \langle X,\widehat \nu\rangle\right).
		\end{align*}
		Therefore
		\[
		[\widehat\nabla_X,\gamma]
		=
		[\nabla_X,\gamma]+\frac12[c(X),\gamma]
		=
		0.
		\]
		This proves the claim. 
		
		 It follows from $\widehat D^\partial_k\gamma_k+\gamma_k\widehat D^\partial_k=0$ that $\langle \widehat D^\partial_k\varphi,\varphi\rangle=0$ if $\varphi$ satisfies the boundary condition $B_\link$ on every face $\link_k$. Consequently, by \eqref{eq:deRhamestimate}, we obtain that
		$$|P_{B_\link}|\geq \frac{\sqrt{2}}{2}.$$
	\end{proof}
	
	\begin{remark}
In the proof of Lemma \ref{lemma:spectrumSphericalLink}, the computation involving  the modified connection $\widehat\nabla$ is carried out explicitly. Here is a more conceptual viewpoint. The connection $\widehat\nabla$ can be interpreted as tensor product connection of the spinorial connection on $T\link$ over  $\link$, with the spinorial flat connection on $\R^3$ over $\link$. Under this interpretation,  $\widehat D_k^\partial\gamma_k+\gamma_k\widehat D_k^\partial$ is precisely determined by two contributions:  the mean curvature of the
face $\link_k\subset\link$ and the  derivative of the  vector
$\nu'_k$. The first contribution vanishes because $\link_k$ is totally geodesic
in $\link$, and the second vanishes because $\nu'_k$ is a constant vector in
the ambient Euclidean space. This gives a conceptual proof of the vanishing of the boundary integrals in \eqref{eq:deRhamestimate}.
	\end{remark}

	By Lemma \ref{lemm:>=1/2}, Lemma \ref{lemma:ess-sa-innerproductcomparison} and Lemma \ref{lemma:spectrumSphericalLink}, we conclude this subsection by the following result.
	\begin{corollary}\label{corollary:ess-saThreeDimension}
		Let $\fiber$ be a polyhedral corner in $\R^3$, and $B$ the boundary condition on $\Bigwedge^*T\fiber$ over $\fiber$ given by
		$$\mathscr E \overbar c(\nu_k)c(\nu_k')w=-w\textup{ on each face }\overbar F_k,$$
 where $\nu_k$ is the inner normal vector on $F_k$, and $\nu_k'$ is a constant unit vector.
		Assume the dihedral angles of $\fiber$ are less than $\pi$. If
		$$\langle\nu_i,\nu_j\rangle\leq \langle \nu_i',\nu_j'\rangle$$
		on every edge made by adjacent faces $F_i$ and $F_j$, then the de Rham operator subject to the boundary condition $B$ is essentially self-adjoint.
	\end{corollary}

	\subsection{Essential self-adjointness for the  general case}\label{subsec:ess-selfadj-general}
	In this subsection, we establish the essential self-adjointness of the twisted Dirac operator $D_B$ subject to suitable conditions on dihedral angles. 
	
	\begin{theorem}\label{thm:ess-saVector}
		Let $f \colon \domain \to \target$ be a spin polyhedral map between two compact polyhedral manifolds of dimension $3$. For each codimension-one face $\overbar F_k$ of $\domain$, let
		$\overbar\nu_k$ denote its inward unit normal, and let $\nu_k$ be a smooth
		unit-length section of $f^*T\target$ over $\overbar F_k$. Let $B$ be the local boundary condition on $E = S(T\domain \oplus f^\ast T\target)$ given by
		$$\mathscr E \overbar c(\overbar \nu_k)c(\nu_k)\sigma=-\sigma\textup{ on }\overbar F_k.$$
		Assume that all dihedral angles of $\domain$ are less than $\pi$. If for each pair of adjacent codimension one faces $\overbar F_i$ and $ \overbar F_j$ of $\domain$, on each connected component of $\overbar F_i\cap \overbar F_j$, we have  either
		\begin{equation}\label{eq:angleStrict}
			\langle\overbar \nu_i,\overbar \nu_j\rangle < \langle\nu_i,\nu_j\rangle  
		\end{equation}
		or
		\begin{equation}\label{eq:angleEqual}
			\langle\overbar \nu_i,\overbar \nu_j\rangle = \langle\nu_i,\nu_j\rangle,  
		\end{equation}
		then   $D$ is essentially  self-adjoint subject to the boundary condition $B$. 
		Furthermore, its self-adjoint closure ${D}_B$ is Fredholm  with domain   $H^1(\domain,E;B)$.
	\end{theorem}
\begin{remark}
	The assumption that the comparison of inner products is either strict or an equality  along each connected component of a codimension-two
	face in Theorem \ref{thm:ess-saVector} is a technical hypothesis used to  control the perturbation of boundary conditions along edges. In this form, Theorem \ref{thm:ess-saVector} does not
	directly cover the geometric setting of Gromov's dihedral rigidity conjecture,
	where the dihedral angle comparison is only non-strict and may vary between
	strict inequality and equality along the same edge.  This difficulty will be resolved via an approximation lemma from Section \ref{sec:approximation}.
\end{remark}
\begin{proof}[Proof of Theorem \ref{thm:ess-saVector}]
		By the discussion after Definition 
	\ref{def:ess-sa}, it suffices to show that $D$ (subject to the boundary condition $B$) is locally essentially self-adjoint at every point $x \in \domain$.
	
   \textbf{Interior of $\domain$.} If $x \in \domain$ lies in the interior of $\domain$, the classical elliptic regularity theory implies that $D$ is locally essentially self-adjoint at $x$.

	\textbf{Codimension one face of $\domain$.} If $x$ lies in the interior of a codimension-one face of $\domain$, then $D$ is locally modeled on a Dirac-type operator on a smooth manifold with
	boundary, subject to a local elliptic boundary condition.  Alternatively, by Theorem \ref{thm:reduction}, we can reduce this case to the standard de Rham operator on manifolds with smooth boundary subject to the absolute boundary condition (see Example \ref{example:deRham}). Therefore, $D$ is locally essentially self-adjoint at $x$.

		\textbf{Codimension two face of $\domain$.} Suppose that $x$ lies in the interior of an edge
		$\overbar F_i\cap\overbar F_j$. Let
		\[
		\overbar\omega_i=\overbar\nu_i(x),
		\qquad
		\overbar\omega_j=\overbar\nu_j(x),
		\qquad
		\omega_i=\nu_i(x),
		\qquad
		\omega_j=\nu_j(x).
		\]
		By Theorem~\ref{thm:reduction}, after passing to a neighborhood of $x$ and
		conjugating by the Lipschitz bundle isomorphism $\Psi$ constructed there,  we have 
		\[
		\Psi^*D\Psi
		=
		D^{\dR}+\mathcal A+\mathcal B
		\]
		on a neighborhood $U$ of the origin in the tangent cone
		$\tancone_x\domain$. Here $D^{\dR}$ is the standard de Rham operator on the
		flat tangent cone, $\mathcal A$ is a first-order operator whose coefficient
		matrices are continuous and vanish at the origin, and $\mathcal B$ is a
		bounded zeroth-order operator. Moreover, $\Psi$ maps the original
		boundary condition $B$ to the constant model boundary condition determined by
		$\overbar\omega_i,\overbar\omega_j,\omega_i,\omega_j$.
		
		The tangent cone $\tancone_x\domain$ is isometric to
		$\mathbb R\times \mathbb G$, where $\mathbb G$ is a sector in $\mathbb R^2$.  After fixing an orthogonal identification of the two copies of
		$\mathbb R^3$ in $S(\mathbb R^3\oplus \mathbb R^3)$, and  after applying an orthogonal transformation on the
		second copy of $\mathbb R^3$ if necessary, we may assume that
		\[
		\overbar\omega_i,\quad \overbar\omega_j,\quad
		\omega_i,\quad \omega_j
		\]
		all lie in the plane of $\mathbb G$. This  transformation commutes with
		the de Rham operator and preserves essential self-adjointness.
		
		Since $\overbar \omega_i, \omega_i, \overbar \omega_j$ and $\omega_j$ are all tangential to $\mathbb G$,  the operator $D^\dR$ subject to the constant  boundary condition $B_{cst}$ is 
		\[ D^\dR = D^\dR_{\mathbb R} \otimes 1 + 1 \otimes D^\dR_{\mathbb G} \]
		where $D^\dR_{\mathbb G}$ is the standard de Rham subject to the induced boundary condition and $D^\dR_{\mathbb R}$ is the standard de Rham operator on $\mathbb R$.  Since either \eqref{eq:angleStrict} or
		\eqref{eq:angleEqual} holds along the edge, we have
		\[
		\langle \overbar\omega_i,\overbar\omega_j\rangle
		\leq
		\langle \omega_i,\omega_j\rangle .
		\]
		By Lemma~\ref{lemma:ess-sa-innerproductcomparison}, the de Rham operator on
		the sector $\mathbb G$ with this boundary condition is essentially
		self-adjoint. Hence $D^\dR$  on $\mathbb R\times\mathbb G$ subject to the constant  boundary condition 
		is essentially self-adjoint.
		
	    For any 
		$\delta>0$, Theorem~\ref{thm:reduction} allows us to shrink $U$  so that
		the coefficient matrices of $\mathcal A$ have supremum norm at most
		$\delta$. After multiplying by a smooth cutoff function, we may extend $\mathcal A$ and $\mathcal B$ to operators on $\mathbb R\times \mathbb G$ with compactly supported coefficients such that $\mathcal A+ \mathcal B$ is formally symmetric, $\mathcal B$ is bounded and  $\mathcal A$ satisfies
		\[
		\|\mathcal A \varphi\|
		\leq
		C_1\delta\|D^{\dR}\varphi\|
		+
		C_2\|\varphi \|
		\]
		for $\varphi$ in the domain of $D^{\dR}$, where $C_1$ and $C_2$ are independent of $\delta$. Choosing $\delta$ sufficiently small gives bound $C_1\delta <1$. The Kato--Rellich perturbation theorem (see for example \cite[Chapter 33, Section 4]{MR1892228}) therefore implies that
		\[
		D^{\dR}+\mathcal A+\mathcal B
		\]
		is locally essentially self-adjoint. Since $\Psi$ is a Lipschitz bundle
		isomorphism with Lipschitz inverse and maps the constant model boundary
		condition to $B$, it follows that $D$ is locally essential self-adjoint at
		$x$.

   \textbf{Codimension three face of $\domain$.} Now suppose that $x$ is a vertex of $\domain$.  The argument is completely similar to the codimension two case above. By using the Lipschitz bundle isomorphism $\Psi$ constructed in Theorem \ref{thm:reduction} for the codimension-three case, the local geometry reduces to a flat polyhedral cone in $\mathbb{R}^3$.  The operator $D$ is then analyzed via a perturbation of the standard de Rham operator $D^\dR$ defined on this flat cone. The same argument as in the codimension-two case above allows us to conclude that the original Dirac operator $D$ on $\domain$ is locally essentially self-adjoint at the vertex $x$. This completes the proof that $D$ (subject to boundary condition $B$) is essentially self-adjoint. 
   
  It remains to prove  the closure $\overbar{D}_B$ of $D_B$ is Fredholm. Since $D_B$ is essentially self-adjoint, $\overbar{D}_B$ is self-adjoint and its domain  is the minimal domain $H^1(\domain, E; B)$. Consider the resolvent operator \[ (\overbar{D}_B+i)^{-1} \colon L^2(\domain, E) \to H^1(\domain, E; B). \] By Proposition \ref{prop:smooth>=} and the  discussion preceding it, this map is norm-continuous. Furthermore, by Rellich's lemma, the inclusion mapping $H^1(\domain, E; B) \hookrightarrow L^2(\domain, E)$ is compact.
  By composing these maps, it follows that $(\overbar{D}_B+i)^{-1} \colon L^2(\domain, E) \to L^2(\domain, E)$ is a compact operator. As a result, the spectrum of $(\overbar{D}_B+i)^{-1}$, and therefore that of $\overbar{D}_B$, is  discrete. The kernel of $\overbar{D}_B$ corresponds exactly to the eigenspace of $(\overbar{D}_B+i)^{-1}$ associated with the eigenvalue $-i$. Because the eigenspaces of a compact operator corresponding to non-zero eigenvalues are finite-dimensional, it follows that the kernel of $\overbar{D}_B$ is finite-dimensional.
  
  Moreover, the discreteness of the spectrum of $\overbar{D}_B$ ensures a spectral gap around zero, which implies that the image of $\overbar{D}_B$ is closed. Because $\overbar{D}_B$ is self-adjoint, its cokernel is isomorphic to its kernel and is thus also finite-dimensional. This completes the proof that $\overbar{D}_B$ is Fredholm.
\end{proof}


\subsection{Boundedness of multiplication of $1/r$ on $\mathbb R^3$}
In this subsection, we recall Hardy's inequality in the form needed later for
the proof of the index theorem for polyhedral manifolds
(Theorem~\ref{thm:indexTheoremVector}). For the reader's convenience, we
include a proof. 

		\begin{lemma}\label{lemma:1/rEmbedding}
		Let $n\geq3$, and let $r=|x|$ be the Euclidean radial function on
		$\mathbb R^n$.  Then multiplying by $r^{-1}$ defines a bounded linear operator
		$$H^1(\R^n)\to L^2(\R^n),~\varphi\mapsto r^{-1} \varphi,$$
		where $H^1(\R^n)$ is the Sobolev $H^1$-space on $\R^n$. 
	\end{lemma}
	\begin{proof}
		By the density of $C_c^\infty(\mathbb R^n)$ in $H^1(\mathbb R^n)$, it
		suffices to prove the estimate for $\varphi\in C_c^\infty(\mathbb R^n)$, where $C_c^\infty(\mathbb R^n)$ is the space of compactly supported smooth functions on $\mathbb R^n$.
		Using polar coordinates $x=r\theta$, with
		$\theta\in\mathbb S^{n-1}$ and $r\in(0,\infty)$, we have
		\[
		\left\|r^{-1}\varphi\right\|_{L^2}^2
		=
		\int_{\mathbb S^{n-1}}\int_0^\infty
		|\varphi(r,\theta)|^2 r^{n-3}\,dr\,d\theta .
		\]
		For fixed $\theta$, integration by parts gives
		\[
		\begin{aligned}
			\int_0^\infty |\varphi(r,\theta)|^2 r^{n-3}\,dr
			&=
			-\frac{2}{n-2}
			\operatorname{Re}\int_0^\infty
			\overline{\varphi(r,\theta)}
			\,\frac{\partial\varphi}{\partial r}(r,\theta)
			r^{n-2}\,dr                                      \\
			&\leq
			\frac{2}{n-2}
			\left(
			\int_0^\infty |\varphi(r,\theta)|^2 r^{n-3}\,dr
			\right)^{1/2}
			\left(
			\int_0^\infty
			\left|\frac{\partial\varphi}{\partial r}(r,\theta)\right|^2
			r^{n-1}\,dr
			\right)^{1/2}.
		\end{aligned}
		\]
		 Therefore,
		\[
		\int_0^\infty |\varphi(r,\theta)|^2 r^{n-3}\,dr
		\leq
		\frac{4}{(n-2)^2}
		\int_0^\infty
		\left|\frac{\partial\varphi}{\partial r}(r,\theta)\right|^2
		r^{n-1}\,dr .
		\]
		Integrating over $\mathbb S^{n-1}$ yields
		\[
		\left\|r^{-1}\varphi\right\|_{L^2}^2
		\leq
		\frac{4}{(n-2)^2}
		\int_{\mathbb S^{n-1}}\int_0^\infty
		\left|\frac{\partial\varphi}{\partial r}(r,\theta)\right|^2
		r^{n-1}\,dr\,d\theta .
		\]
		Since $
		\left|\partial\varphi/\partial r\right|
		\leq
		|\nabla\varphi|,$
		we obtain
		\[
		\left\|r^{-1}\varphi\right\|_{L^2}
		\leq
		\frac{4}{(n-2)^2}
		\|\nabla\varphi\|_{L^2} \leq 	\frac{4}{(n-2)^2}\|\varphi\|_{H^1}.
		\]
		The estimate extends by density to all $\varphi\in H^1(\mathbb R^n)$, proving
		the lemma.
	\end{proof}

			\section{A gluing formula for the Fredholm index}\label{sec:gluing}
	In this section, we establish a gluing formula for the Fredholm index of the
	Dirac operator. This formula will be used later in the computation of the
	Fredholm index of $D_B$.

		\begin{theorem}\label{thm:gluing}
		Let $f \colon \domain \to \target$ be a spin polyhedral map between two compact polyhedral manifolds of dimension $3$. For each codimension-one face $\overbar F_k$ of $\domain$, let
		$\overbar\nu_k$ denote its inward unit normal, and let $\nu_k$ be a smooth
		unit-length section of $f^*T\target$ over $\overbar F_k$. Let $B$ be the local boundary condition on $E = S(T\domain \oplus f^\ast T\target)$ given by
		$$\mathscr E \overbar c(\overbar \nu_k)c(\nu_k)\sigma=-\sigma\textup{ on }\overbar F_k.$$
		Assume that all dihedral angles of $\domain$ are less than $\pi$. If for each pair of adjacent codimension one faces $\overbar F_i$ and $ \overbar F_j$ of $\domain$, on each connected component of $\overbar F_i\cap \overbar F_j$, we have  either
		\begin{equation*}
			\langle\overbar \nu_i,\overbar \nu_j\rangle < \langle\nu_i,\nu_j\rangle  
		\end{equation*}
		or
		\begin{equation*}
			\langle\overbar \nu_i,\overbar \nu_j\rangle = \langle\nu_i,\nu_j\rangle.  
		\end{equation*} 
		Suppose that $\domain$ decomposes as $\domain = \domain_1 \cup_{\hsurface} \domain_2$, where $\domain_1$ and $\domain_2$ are polyhedral manifolds, and  $\hsurface= \domain_1 \cap \domain_2$ is a hypersurface in $\domain$ that is disjoint from all	faces of codimension greater than two, and is orthogonal to every codimension-one face $\overbar F_k$ that it intersects; see Figure \ref{fig:glue}.
			
			Let $\overbar \nu_\hsurface$ denote the unit inner normal vector field to $\hsurface$, when $\hsurface$ is viewed as a boundary hypersurface of $\domain_2$. Assume $\nu_\hsurface$ is a smooth unit-length section of $f^\ast T\target$ on $\hsurface$ such that $\nu_\hsurface$ is orthogonal to $\nu_k$ along  $\hsurface\cap \overbar F_k$. Let $B_1$ be the boundary condition for $E|_{\domain_1}$ which coincides with $B$ on $\partial \domain_1 \setminus \hsurface$ and is given on $\hsurface$ by
			\begin{equation}\label{eq:B1onSigma}
				\mathscr E \overbar c(\overbar \nu_\hsurface)c(\nu_\hsurface) \sigma = \sigma.
			\end{equation}
			Similarly, let $B_2$ be the boundary condition for $E|_{\domain_2}$ which coincides with $B$ on $\partial \domain_2 \setminus \hsurface$ and is given on $\hsurface$ by
			\begin{equation}\label{eq:B2onSigma}
				\mathscr E \overbar c(\overbar \nu_\hsurface)c(\nu_\hsurface) \sigma = -\sigma.
			\end{equation}
			Let $D^{\domain_i}$ denote the restriction of $D$ to $\domain_i$. Then
			$D^{\domain_1}_{B_1}$ and $D^{\domain_2}_{B_2}$ are essentially self-adjoint
			and Fredholm. Moreover,
			\begin{equation}\label{eq:gluingIndexFormula}
				\ind(D_B)=\ind(D^{\domain_1}_{B_1})+\ind(D^{\domain_2}_{B_2}).
			\end{equation}
		\end{theorem}
		
     As a preparation for the proof of Theorem \ref{thm:gluing}, we first prove a few lemmas. 
     
     \begin{figure}
     \begin{tikzpicture}[line cap=round, line join=round, scale=0.7]
     	
     	\coordinate (V) at (0,0);
     	
     	\coordinate (TL1) at (2.0, 4.2);   
     	\coordinate (BL1) at (1.8, 2.4);   
     	\coordinate (TR1) at (4.5, 3.2);   
     	\coordinate (BR1) at (4.0, 1.6);   
     	
     	\coordinate (Center1) at (3.1, 2.85);
     	
     	\def\dx{2.2}
     	\def\dy{2.0}
     	
     	\coordinate (TL2) at ($(TL1) + (\dx, \dy)$);
     	\coordinate (BL2) at ($(BL1) + (\dx, \dy)$);
     	\coordinate (TR2) at ($(TR1) + (\dx, \dy)$);
     	\coordinate (BR2) at ($(BR1) + (\dx, \dy)$);
     	
     	\coordinate (Center2) at ($(Center1) + (\dx, \dy)$);
     	
     	\coordinate (FTL) at ($(V)!1.5!(TL1) + (\dx, \dy)$);
     	\coordinate (FBL) at ($(V)!1.5!(BL1) + (\dx, \dy)$);
     	\coordinate (FTR) at ($(V)!1.5!(TR1) + (\dx, \dy)$);
     	\coordinate (FBR) at ($(V)!1.5!(BR1) + (\dx, \dy)$);
     	
     	\coordinate (ExtFTL) at ($(V)!1.65!(TL1) + (\dx, \dy)$);
     	\coordinate (ExtFBL) at ($(V)!1.65!(BL1) + (\dx, \dy)$);
     	\coordinate (ExtFTR) at ($(V)!1.65!(TR1) + (\dx, \dy)$);
     	\coordinate (ExtFBR) at ($(V)!1.65!(BR1) + (\dx, \dy)$);

     	
     	\draw[thick, densely dashed, gray!80] (V) -- (TR1);
     	
     	\draw[thick, densely dashed, gray!80] (TR2) -- (FTR);
     	\draw[thick, densely dashed, gray!80] (FTR) -- (ExtFTR);
     	\draw[thick, densely dashed, gray!80] (FTL) to[bend left=8] (FTR);
     	\draw[thick, densely dashed, gray!80] (FTR) -- (FBR);
     	
     	\fill[gray!30]
     	(TL2) to[bend left=8] (TR2)
     	-- (BR2)
     	to[bend right=8] (BL2)
     	-- cycle;
     	
     \draw[thick, densely dashed, black] (TL2) to[bend left=8] (TR2);
     \draw[thick, densely dashed, black] (TR2) -- (BR2);
     \draw[thick, black] (BL2) to[bend left=8] (BR2);
     \draw[thick, black] (TL2) -- (BL2);
     \node at (Center2) {$\Sigma$}; 
     	
     	\draw[thick, black] (TL2) -- (FTL) -- (ExtFTL);
     	\draw[thick, black] (BL2) -- (FBL) -- (ExtFBL);
     	\draw[thick, black] (BR2) -- (FBR) -- (ExtFBR);
     	
     	\draw[thick, black] (FBL) to[bend left=8] (FBR);
     	\draw[thick, black] (FTL) -- (FBL);
     	
     	\fill[gray!30]
     	(TL1) to[bend left=8] (TR1)
     	-- (BR1)
     	to[bend right=8] (BL1)
     	-- cycle;
     	
     	\draw[thick, black] (TL1) to[bend left=8] (TR1);
     	\draw[thick, black] (BL1) to[bend left=8] (BR1);
     	\draw[thick, black] (TL1) -- (BL1);
     	\draw[thick, black] (TR1) -- (BR1);
     	\node at (Center1) {$\Sigma$}; 
     
     	\draw[thick, black] (V) -- (TL1);
     	\draw[thick, black] (V) -- (BL1);
     	\draw[thick, black] (V) -- (BR1);
     	
     \end{tikzpicture}
     	\caption{$\protect\domain$ decomposes into $\protect\domain_1$ and $\protect\domain_2$ along $\protect\hsurface$.}
     	\label{fig:glue}
     \end{figure}

    \begin{lemma}\label{lemma:piecesEssSAFredholm}
    	The operators $D^{\domain_1}_{B_1}$ and $D^{\domain_2}_{B_2}$ are essentially
    	self-adjoint and Fredholm.
    \end{lemma}
    
    \begin{proof}
    	We prove the statement for $\domain_1$; the proof for $\domain_2$ is
    	identical. The codimension-one faces of $\domain_1$ consist of $\hsurface$ and
    	the faces
    	\[
    	\overbar F_{1,k}\coloneqq \overbar F_k\cap\domain_1
    	\]
    	coming from codimension-one faces $\overbar F_k$ of $\domain$.
    	
    	For adjacent faces $\overbar F_{1,i}$ and $\overbar F_{1,j}$ inherited from
    	$\domain$, the required inner-product comparison is exactly the corresponding
    	comparison on $\overbar F_i\cap\overbar F_j$. Now suppose that $\hsurface$ and
    	$\overbar F_{1,i}$ are adjacent. The  unit inner normal to $\hsurface$ as a
    	boundary face of $\domain_1$ is $-\overbar\nu_\hsurface$. Thus the boundary
    	condition \eqref{eq:B1onSigma} is precisely the standard condition
    	\[
    	\mathscr E\,\overbar c(-\overbar\nu_\hsurface)c(\nu_\hsurface)\sigma
    	=
    	-\sigma
    	\]
    	written using the inward normal of $\domain_1$. Moreover, by assumption,
    	\[
    	\langle -\overbar\nu_\hsurface,\overbar\nu_i\rangle
    	=
    	0
    	=
    	\langle\nu_\hsurface,\nu_i\rangle
    	\]
    	along $\hsurface\cap\overbar F_{1,i}$. Hence the equality case of the
    	inner-product comparison holds for the pair
    	$\hsurface,\overbar F_{1,i}$.
    	
    	Therefore all hypotheses of Theorem~\ref{thm:ess-saVector} hold for
    	$\domain_1$ with boundary condition $B_1$. It follows that
    	$D^{\domain_1}_{B_1}$ is essentially self-adjoint and Fredholm.
    \end{proof}
     
    With the same notation as in Theorem \ref{thm:gluing},  let $D^\hsurface$ be the restriction of  the Dirac operator $D$ on $\hsurface$. The boundary condition $B$ along the codimension one faces $\overbar F_i$ of $\domain$ restricts to a  boundary condition on the codimension one faces  $\hsurface_k\coloneqq \hsurface\cap \overbar F_k$ of $\hsurface$. Let us denote this induced  boundary condition by $B_\hsurface$. 
		\begin{lemma}\label{lemma:gluingEss-sa}
			Assume the hypotheses of Theorem~\ref{thm:gluing}. Then
			$D^\hsurface$, subject to $B_\hsurface$, is essentially self-adjoint and
			Fredholm.
		\end{lemma}
		\begin{proof}
			On $\hsurface$, consider the orthogonal decompositions
			\[
			T\domain=\mathbb R\overbar\nu_\hsurface\oplus T\hsurface
			\textup{ and }
			f^*T\target=\mathbb R\nu_\hsurface\oplus
			(\mathbb R\nu_\hsurface)^\perp .
			\]
			Set
			$\omega
			=
			\mathscr E\,\overbar c(\overbar\nu_\hsurface)c(\nu_\hsurface).$
			The operator $\omega$ is a self-adjoint involution. With respect to its
			$(\pm1)$-eigenspace decomposition, the bundle $E|_\hsurface$ is identified
			with two copies of $S\bigl(T\hsurface\oplus(\mathbb R\nu_\hsurface)^\perp\bigr)$: 
			\[
			E \cong S\bigl(T\hsurface\oplus(\mathbb R\nu_\hsurface)^\perp\bigr) \oplus S\bigl(T\hsurface\oplus(\mathbb R\nu_\hsurface)^\perp\bigr).
			\]
			Since $\hsurface$ is orthogonal to every face $\overbar F_k$ that it
			intersects and $\nu_\hsurface$ is orthogonal to the corresponding $\nu_k$, a
			direct  computation shows that $\omega$ commutes with the induced boundary condition
			$B_\hsurface$ along each $\hsurface_k = \hsurface\cap \overbar F_k$. Hence $B_\hsurface$ splits to boundary conditions on each
			copy of $
			S\bigl(T\hsurface\oplus(\mathbb R\nu_\hsurface)^\perp\bigr)$. Moreover, the
			inner-product comparison at the vertices of $\hsurface$ is exactly the
			corresponding comparison along the adjacent faces of $\domain$.
			
			The endomorphism $\omega$ need not be parallel for the original connection. We
			therefore introduce the modified connection
			\[
			\widetilde\nabla_X
			=
			\nabla_X+\frac12\omega(\nabla_X\omega),
			\qquad
			X\in T\hsurface .
			\]
			Since $\omega^2=1$, this connection satisfies $
			\widetilde\nabla_X\omega=0.$
		Define
			\[
			\widetilde D^\hsurface
			=
			\sum_a \overbar c(\overbar e_a)\widetilde\nabla_{\overbar e_a},
			\]
			where $\{\overbar e_a\}$ is a local orthonormal frame of $T\hsurface$. By construction,
			$\widetilde D^\hsurface$ commutes with $\omega$, and hence restricts to each
			of the two summands above.
			
			To summarize, we see that 
			the two-dimensional essential self-adjointness criterion
			(Lemma~\ref{lemma:ess-sa-innerproductcomparison} and Theorem \ref{thm:ess-saVector}) implies that $\widetilde D^\hsurface$ with boundary
			condition $B_\hsurface$ is essentially self-adjoint and  Fredholm. Finally, $D^\hsurface-\widetilde D^\hsurface$ is a bounded zeroth-order
			operator. It follows immediately that 
			$D^\hsurface$ with boundary condition $B_\hsurface$ is also essentially
			self-adjoint and Fredholm.
		\end{proof}

		Since the operator $D^\hsurface_{B_\hsurface}$ from Lemma \ref{lemma:gluingEss-sa} is essentially self-adjoint, the domain of its closure is $H^1(\hsurface, E; B_\hsurface)$. The operator $D^\hsurface_{B_\hsurface}$ has discrete spectrum. Choose a complete orthonormal basis $\{\varphi_\lambda\}$ of $L^2(\hsurface,E)$  consisting of eigenfunctions of $D^\hsurface_{B_\hsurface}$ with 
		$D^\hsurface_{B_\hsurface}\varphi_\lambda=\lambda\varphi_\lambda. $
		\begin{definition}\label{def:gluingH1/2}
			 Define
			$$H^{1/2}(\hsurface,E;B_\hsurface)=\Big\{\varphi=\sum_\lambda a_\lambda\varphi_\lambda\in L^2(\hsurface,E):\sum_\lambda|a_\lambda|^2(1+|\lambda|)<\infty\Big\},$$
			equipped with the norm
			$$\|\varphi\|_{1/2}^2=\sum_\lambda|a_\lambda|^2(1+|\lambda|).$$
		\end{definition}

		\begin{lemma}\label{lemma:gluingExtension}
		For $i=1,2$, let
		\[
		\tau_i\colon H^1(\domain_i,E)\longrightarrow H^{1/2}(\hsurface,E)
		\]
		be the trace map. Then there exists a bounded linear  operator
		\[
		\mathcal E_i\colon
		H^{1/2}(\hsurface,E;B_\hsurface)
		\longrightarrow
		H^1(\domain_i,E;B_{\mathrm{ext},i})
		\]
		such that $	\tau_i\circ\mathcal E_i=\id.$
		Here $B_{\mathrm{ext},i}$ denotes the boundary condition $B$ on
		$\partial\domain_i\setminus\hsurface$; no boundary condition is imposed on
		$\hsurface$.
		\end{lemma}
		\begin{proof}
		It suffices to construct the extension in a collar neighborhood of
		$\hsurface$. We first consider the product model
		\[
		[0,1]\times\hsurface
		\]
		with product metric $dt^2+g_\hsurface$, and with a boundary condition $\widetilde B$
		constant in the $t$-direction along the side faces
		$[0,1]\times\hsurface_k$. In this case, the corresponding Dirac operator takes the form
		$$\slashed{D}=\overbar c(\partial_t)\frac{\partial}{\partial t}+D^\hsurface.$$
		
		Let $\chi\in C^\infty([0,1])$ satisfy $\chi(t)=1$ for
		$0\leq t\leq 1/3$ and $\chi(t)=0$ for $t\geq 2/3$. Choose an even positive
		Schwartz function $\Phi$ on $\mathbb R$ such that $\Phi(0)=1$. For
	   \[
		\varphi=\sum_\lambda a_\lambda\varphi_\lambda
		\in H^{1/2}(\hsurface,E;B_\hsurface),
		\]
		define
		\[
		\mathcal E^\sharp\varphi(t)
		=
		\chi(t)\Phi(tD^\hsurface_{B_\hsurface})\varphi
		=
		\chi(t)\sum_\lambda \Phi(t\lambda)a_\lambda\varphi_\lambda .
		\]
		Then $
		\tau(\mathcal E^\sharp\varphi)=\varphi$, 
		because $\chi(0)=1$ and $\Phi(0)=1$.
		
		We now estimate the $H^1$ norm of $\mathcal E^\sharp\varphi$. The $L^2$ norm satisfies
		\[
		\|\mathcal E^\sharp\varphi\|_{L^2([0,1]\times\hsurface)}
		\leq C\|\varphi\|_{L^2(\hsurface)} .
		\]
		For the normal derivative, we have  
		\[
		\partial_t \mathcal E^\sharp\varphi(t)
		=
		\chi'(t)\Phi(tD^\hsurface)\varphi
		+
		\chi(t)D^\hsurface
		\Phi'(tD^\hsurface)\varphi .
		\] 
		The term involving $\chi'$ is bounded by
		$C\|\varphi\|_{L^2(\hsurface)}$. The remaining term satisfies:
		\[
		\begin{aligned}
			\left\|\chi(t)D^\hsurface
			\Phi'(tD^\hsurface)\varphi \right\|_{L^2([0,1]\times\hsurface)}^2
			&\leq
			C
			\sum_\lambda
			|\lambda|\,|a_\lambda|^2
			\int_0^{|\lambda|}|\Phi'(s)|^2\,ds                         \\
			&\leq
			C\|\Phi'\|_{L^2(\mathbb R)}^2
			\|\varphi\|_{1/2}^2 .
		\end{aligned}
		\]
		Similarly, 
		\[
		\begin{aligned}
			\left\|
			\chi(t)D^\hsurface_{B_\hsurface}
			\Phi(tD^\hsurface_{B_\hsurface})\varphi
			\right\|_{L^2([0, 1]\times \hsurface)}^2
			&\leq
			C
			\sum_\lambda
			|\lambda|\,|a_\lambda|^2
			\int_0^{|\lambda|}|\Phi(s)|^2\,ds                         \\
			&\leq
			C\|\Phi\|_{L^2(\mathbb R)}^2
			\|\varphi\|_{1/2}^2 .
		\end{aligned}
		\]
		For each $t>0$,  $\Phi(tD^\hsurface_{B_\hsurface})$ maps
		$L^2(\hsurface,E)$ to $\operatorname{Dom}(D^\hsurface_{B_\hsurface}) = H^{1}(\hsurface,E;B_\hsurface)$.
		In particular, 
		$ \mathcal E^\sharp\varphi(t) = \chi(t)\Phi(tD^\hsurface_{B_\hsurface})\varphi$
		belongs to $H^1(\hsurface,E;B_\hsurface)$ for every $t>0$.
		Since the graph norm of $D^\hsurface_{B_\hsurface}$ is equivalent to the
		$H^1$ norm on $H^1(\hsurface,E;B_\hsurface)$, we have 
		\[
		\|u\|_{H^1(\hsurface)}^2
		\leq
		C\left(
		\|u\|_{L^2(\hsurface)}^2
		+
		\|D^\hsurface_{B_\hsurface}u\|_{L^2(\hsurface)}^2
		\right),
		\qquad
		u\in H^{1}(\hsurface,E;B_\hsurface),
		\]
		Applying this to $u(t)= \mathcal E^\sharp\varphi(t)$, we obtain
		\[
		\int_0^1\|\nabla^{\hsurface}u(t)\|_{L^2	(\hsurface)}^2\,dt
		\leq
		C\int_0^1
		\left(
		\|u(t)\|_{L^2(\hsurface)}^2
		+
		\|D^\hsurface_{B_\hsurface}u(t)\|_{L^2(\hsurface)}^2
		\right)\,dt .
		\]
		These
		estimates together imply that
		\[
		\mathcal E^\sharp\colon
		H^{1/2}(\hsurface,E;B_\hsurface)
		\longrightarrow
		H^1([0,1]\times\hsurface,E;B_{\mathrm{side}})
		\]
		is bounded. Here $B_{\mathrm{side}}$ denotes the product boundary condition
		on the side faces $[0,1]\times\hsurface_k$.
		
		Now we consider the general case. Since $\hsurface$ is disjoint from codimension three faces of $\domain$, by an argument similar to the codimension two case of Theorem \ref{thm:reduction}, there is a smooth bundle isomorphism $\Psi$  from a tubular neighborhood of $\hsurface$ in $\domain$ to the direct product space $(-\varepsilon, \varepsilon)\times \hsurface$ (viewed as a subspace of the normal bundle of $\hsurface$) such that  $\Psi|_\hsurface$ equals the identity bundle map, $\Psi$ maps the boundary condition $B$ for the operator $D$ to the boundary condition $\widetilde B$ for $\slashed D$ that is constant along the $\partial_t$-direction.
		
		For $\varphi=\sum_\lambda a_\lambda\varphi_\lambda$ in $H^{1/2}(\hsurface,E; B_{\hsurface})$, we define  the extension map 
		$$ \mathcal  E(\varphi)\coloneqq \Psi^* \circ \mathcal E^\sharp \circ \Psi (\varphi)  $$
		where $\mathcal E^\sharp$ is the extension map defined  using $\slashed{D}$ (subject to the boundary condition $\widetilde B$)  in the product case above. 
		
		Since $\Psi$ maps $H^1$ Sobolev space to $H^1$ Sobolev space, it follows that $\mathcal E$ is a bounded linear map $H^{1/2}(\hsurface,E;B_\hsurface)\to H^1(\domain_i,E;B_{\mathrm{ext},i})$. Furthermore, since $\Psi|_\hsurface$ is the identity, we have 
		$$ \tau\circ \mathcal{E}(\varphi) =  \tau \circ \Psi^* \circ \mathcal{E}^\sharp \circ \Psi (\varphi) = \varphi. $$ This finishes the proof. 
		\end{proof}

\begin{proposition}\label{prop:gluingMixedEss-sa}
		Assume the notation of Theorem \ref{thm:gluing}. Consider the disjoint union $\domain_1 \sqcup \domain_2$, and let $V \subset E|_\hsurface$ denote the subbundle over $\hsurface$ given by 
	$$V = \left\{\sigma \in E|_\hsurface : \mathscr{E}\overbar{c}(\overbar{\nu}_\hsurface)c(\nu_\hsurface)\sigma = \sigma\right\}.$$
	Let $Q \colon E|_\hsurface \to V$ be the orthogonal projection onto $V$, and set $Q^{\perp} = 1 - Q$. For $s \in [0,1]$, define a family of boundary conditions $B^s$ on $E$ over $\domain_1 \sqcup \domain_2$ as follows. A pair of sections $(\sigma_1, \sigma_2)$ of $E|_{\domain_1} \sqcup E|_{\domain_2}$ satisfies $B^s$ if:
	\begin{enumerate}[label=$(\roman*)$]
	\item $s Q(\sigma_1) = Q(\sigma_2)$ and $Q^\perp(\sigma_1) = s Q^\perp(\sigma_2)$ on $\hsurface$;
					
	\item $\sigma_1$ and $\sigma_2$ satisfy the original boundary condition
	$B$ on all codimension-one faces of $\domain_1$ and $\domain_2$ other
	than $\hsurface$.
	\end{enumerate}
	Then the Dirac operator $D$ on $E$ over $\domain_1 \sqcup \domain_2$, subject to the boundary condition $B^s$, is essentially self-adjoint and Fredholm.
	\end{proposition}
	\begin{proof}
A direct computation shows that $D_{B^s}$ is formally symmetric. We will prove that $D_{B^s}$ is locally essentially self-adjoint at every point $x \in \domain$.

	By Theorem \ref{thm:ess-saVector} and its proof, $D$ is clearly  locally essential self-adjoint at any point that does not lie in $\hsurface$. 
	
	On $\hsurface$, we reduce the verification of  essential self-adjointness to a product case. Consider the orthogonal decompositions 
	\[ T\domain = \mathbb R(\overbar \nu_\hsurface) \oplus T\hsurface \textup{ and } f^\ast T\target =  \mathbb R(\nu_\hsurface) \oplus \mathbb R(\nu_\hsurface)^\perp.\] Consider the bundle $\widetilde E \coloneqq S(T\hsurface\oplus \mathbb R(\nu_\hsurface)^\perp ) \otimes S(\mathbb R(\overbar \nu_\hsurface)\oplus \mathbb R(\nu_\hsurface))$ over the  product space $(-\varepsilon, \varepsilon) \times \hsurface$ obtained by trivially extending the bundle $E$ over $\{0\}\times \hsurface$ along the interval direction $(-\varepsilon, \varepsilon)$. The boundary condition $B$ for codimension one face $\overbar F_k$ of $\domain$ restricts to a boundary condition $B_{\hsurface}$ along the codimension one face $\hsurface_k = \hsurface \cap \overbar F_k$ of $\hsurface$. We extend $B_{\hsurface}$ constantly along the interval direction to obtain a boundary condition $\widetilde B$ along the codimension one face $(-\varepsilon, \varepsilon)\times \hsurface_k$ of $(-\varepsilon, \varepsilon)\times \hsurface$. 
	
	On the disjoint union $(-\varepsilon, 0]\times \Sigma \sqcup [0, \varepsilon)\times \Sigma$, we say  a pair of sections $(\sigma_1, \sigma_2)$ of $\widetilde E|_{(-\varepsilon, 0]\times \Sigma} \sqcup \widetilde E|_{[0, \varepsilon)\times \Sigma}$ satisfies $\widetilde B^s$ if $s Q(\sigma_1) = Q(\sigma_2)$ and $Q^\perp(\sigma_1) = s Q^\perp(\sigma_2)$ on $\hsurface$.

	 By an argument similar to the codimension-two case of Theorem \ref{thm:reduction}, there exist a tubular neighborhood $U_1$ (resp. $U_2$) of $\hsurface$ in $\domain_1$ (resp. $\domain_2$) and a smooth bundle isomorphism from  the spinor bundle $\widetilde E = S(T\hsurface\oplus R(\nu_\hsurface)^\perp ) \otimes S(\mathbb R\oplus \mathbb R)$ on the product space $(-\varepsilon, 0] \times \hsurface$ (resp. $[0, \varepsilon)\times \hsurface $) to the spinor bundle $E$ over $U_1$ (resp. $U_2$)  such that 
	\begin{enumerate}
		\item Restricted to $\hsurface$, $\Psi$ equals the identity bundle map:
		$$\Psi|_\hsurface = \operatorname{id}: E|_\hsurface \to \widetilde{E}|_{\{0\} \times \hsurface};$$
		\item $\Psi$ maps the codimension one faces of $(-\varepsilon, 0] \times \hsurface$ (resp. $[0, \varepsilon)\times \hsurface $) to the corresponding codimension one faces of $U_1$ (resp. $U_2$), and the boundary condition $\widetilde B^s$ for sections of $\widetilde E$ to the boundary condition $B^s$ for sections of $E$;
		\item and $\Psi^* D\Psi = \slashed D + \mathcal A + \mathcal B$
		where 
		$$\slashed{D}=\overbar c(\partial_t)\frac{\partial}{\partial t}+D^\hsurface,$$ 
		and $\mathcal{A}$ is a first-order differential operator whose coefficients are smooth and vanish on $\{0\}\times \hsurface$, and   $\mathcal{B}$ is a smooth zeroth-order operator.
	\end{enumerate}   
	 Similar to the proof of Theorem \ref{thm:ess-saVector}, by applying the Kato-Rellich perturbation theorem, it suffices to show that $\slashed  D$ subject to the boundary condition $\widetilde B^s$ is essentially self-adjoint on $((-\infty, 0] \times \hsurface)\sqcup ([0, \infty)\times \hsurface )$ . 
		
	 
	  To verify the essential self-adjointness of $\slashed D_{\widetilde B^s}$, it suffices to show that  the deficiency spaces
	  $$ \mathcal N_\pm \coloneqq \left\{ \sigma \in \operatorname{Dom}_{\max}(\slashed  D_{B^s})  \mid \slashed{D} \sigma = \pm i \sigma \right\}$$ are zero. We give the
	  argument for $\mathcal N_+$;	  the proof for $\mathcal N_-$ is identical.
	
	   Let  $\sigma = (\sigma_1, \sigma_2) \in \mathcal N_+$, that is, 
	 \begin{equation}\label{eq:defi}
	 	\slashed D \sigma_k - i\sigma_k = 0. 
	 \end{equation} Applying $(\slashed{D} + i)$ to the above equation yields  
	 \[ (-\partial_t^2 + (D^\hsurface)^2 + 1)\sigma_k = 0.\] 
	 The $L^2$ condition on the half-cylinders implies
	 \[
	 \sigma_1(t)=e^{tA}\varphi_1\quad(t\leq0),
	 \textup{ and } 
	 \sigma_2(t)=e^{-tA}\varphi_2\quad(t\geq0),
	 \]where $A = \sqrt{(D^\hsurface)^2 + 1}$ and $\varphi_k\in L^2(\hsurface, \widetilde E)$. Substituting
	 these expressions into equation \eqref{eq:defi} gives
	 \begin{equation}\label{eq:gluingDefBoundary}
	 	(JA+D^\hsurface)\varphi_1=i\varphi_1
	 	\textup{ and }
	 	(-JA+D^\hsurface)\varphi_2=i\varphi_2 .
	 \end{equation}
	 where $J = \overbar c(\partial_t)$.
	 Decompose
	 \[
	 \varphi_k=a_k+b_k,
	 \qquad
	 a_k\in V,\quad b_k\in V^\perp .
	 \]
	 Since $D^\hsurface$ and $A$ commute with\footnote{Under the identification $\Psi\colon E\xrightarrow{\cong} \widetilde E$, the operator $\mathscr E\overbar c(\overbar \nu_\hsurface) c(\nu_\hsurface)$ becomes $\mathscr E\overbar c(\partial_t) c(\partial_t)$. } $\mathscr E\overbar c(\partial_t) c(\partial_t)$, they preserve $V$ and
	 $V^\perp$. Since $J$ anticommutes with $\mathscr E\overbar c(\partial_t) c(\partial_t)$, it interchanges $V$ and
	 $V^\perp$. The boundary condition $\widetilde B^s$ gives
	 \begin{equation}\label{eq:gluingBoundaryComponents}
	 	a_2=s a_1,
	 	\qquad
	 	b_1=s b_2 .
	 \end{equation}
	 Taking the $V$-component of the first equation in
	 \eqref{eq:gluingDefBoundary}, multiplying it by $s$, and using
	 \eqref{eq:gluingBoundaryComponents}, gives
	 \[
	 s^2 JA b_2+sD^\hsurface a_1=is a_1 .
	 \]
	 Taking the $V$-component of the second equation in
	 \eqref{eq:gluingDefBoundary} gives
	 \[
	 -JA b_2+sD^\hsurface a_1=is a_1 .
	 \]
	 Subtracting the two equations yields
	 \[
	 (1+s^2)JA b_2=0.
	 \]
	 Since $J$ is invertible and $A\geq1$, we get $b_2=0$, and hence $b_1=0$. Now
	 the $V^\perp$-component of the first equation in
	 \eqref{eq:gluingDefBoundary} gives
	 \[
	 JA a_1=0,
	 \]
	 so $a_1=0$, and therefore $a_2=0$. Hence $\varphi_1=\varphi_2=0$, and so
	 $\sigma_1=\sigma_2=0$. Thus $\mathcal N_+=0$. Similarly,
	 $\mathcal N_-=0$.
	 
	 Therefore, $\slashed{D}_{\widetilde{B}^s}$ is essentially self-adjoint. By the Kato-Rellich perturbation theorem,$\Psi^* D\Psi = \slashed D + \mathcal A + \mathcal B$, subject to boundary condition $\widetilde B^s$,  is essentially self-adjoint; consequently, so is $D_{B^s}$. The domain of the self-adjoint closure of $D_{B^s}$ is thus $H^1(\domain_1 \sqcup \domain_2, E; B^s)$, from which it follows that $D_{B^s}$ is Fredholm.
			\end{proof}
			
	Now let us  prove Theorem \ref{thm:gluing}.		
			\begin{proof}[Proof of Theorem \ref{thm:gluing}]
				By Proposition \ref{prop:gluingMixedEss-sa}, for each $s\in [0, 1]$, we have a Fredholm operator
				$$D_{B^s} \colon H^1(\domain_1\sqcup \domain_2,E;B^s)\to L^2(\domain_1\sqcup  \domain_2,E),$$
				where $B^s$ is the boundary condition  defined in Proposition \ref{prop:gluingMixedEss-sa}.
				
				For $i=1,2$, let $\tau_i$ be the trace map from $H^1(\domain_i,E)$ to $H^{1/2}(\hsurface,E)$, and $\mathcal E_i$ from $H^{1/2}(\hsurface,E;B_\hsurface)$ to $H^{1}(\domain_i,E)$ be the extension map as in Lemma \ref{lemma:gluingExtension}. For $s, t\in [0,1 ]$, define the linear map
				\begin{align*}
					T_{s,t}\colon H^1(\domain_1\sqcup  \domain_2,E;B^s)&\to H^1(\domain_1\sqcup \domain_2,E;B^t)\\
					( \sigma_1, \sigma_2)&\mapsto ( \sigma_1+(t-s)\mathcal E_1Q^\perp\tau_2 \sigma_2,  \sigma_2+(t-s)\mathcal E_2Q\tau_1  \sigma_1 ),
				\end{align*}
				where $Q$ is the projection $E|_\hsurface \to V$ defined in Proposition \ref{prop:gluingMixedEss-sa}. 
				Note that $T_{s,t}$ is invertible with its inverse given by $T_{t,s}$, and is  continuous in both $s$ and $t$ with respect to the operator norm. Therefore the family of Fredholm operators
				\[
				D_{B^s}\circ T_{1,s}\colon
				H^1(\domain_1\sqcup\domain_2,E;B^1)
				\longrightarrow
				L^2(\domain_1\sqcup\domain_2,E)
				\]
				depends continuously on $s$ in operator norm. By the homotopy invariance of the Fredholm index, $\ind(D_{B^s}\circ T_{1,s})$ is independent of $s$. Since $T_{1, s}$ is invertible,
				\[ \ind(D_{B^s}\circ T_{1, s})=\ind(D_{B^s}) . \] 
				Hence $\ind(D_{B^s})$ is independent of  $s\in [0, 1]$. 
				
				When $s=1$, the boundary condition $B^1$ is the matching condition across
				$\hsurface$, so $D_{B^1}$ is the original glued operator $D_B$. When $s=0$, the boundary condition $B^0$ becomes
				\[
				Q(\sigma_2)=0
				\textup{ and }
				Q^\perp(\sigma_1)=0.
				\]
				Since $V$ is the $+1$-eigenspace of
				$\mathscr E\,\overbar c(\overbar\nu_\hsurface)c(\nu_\hsurface)$, these
				are exactly the boundary conditions $B_1$ on $\domain_1$ and $B_2$ on
				$\domain_2$. Therefore
				\[
				\ind(D_{B^1})=\ind(D_B),
				\qquad
				\ind(D_{B^0})
				=
				\ind(D^{\domain_1}_{B_1})+\ind(D^{\domain_2}_{B_2}).
				\]
				Since the index is constant in $s$, it follows that 
				\[ 	\ind(D_B)=\ind(D^{\domain_1}_{B_1})+\ind(D^{\domain_2}_{B_2}).\] \end{proof}
			
\section{An approximation lemma and its consequences} \label{sec:approximation}
In this section, we prove an approximation lemma for boundary conditions. The
lemma allows us to pass to the limit from solutions satisfying a family of
approximating boundary conditions $B_t$ to a limiting solution satisfying the limiting boundary condition 
$B_0$, provided that  $B_0$ is extremal in the sense of Definition \ref{def:extremalBoundaryCondition} below.

\begin{definition} \label{def:extremalBoundaryCondition}
	Let $(\domain,\overbar g)$ and $(\target,g)$ be polyhedral manifolds and $f\colon \domain\to \target$ a spin polyhedral map. Let $D$ be the Dirac operator associated with  the spinor bundle $E=S(T\domain\oplus f^*T\target)$ over $\domain$. Let $B$ be the boundary condition defined on each codimension-one face $\overbar F_k$ of $\domain$ by
	$$\mathscr E\overbar c(\overbar \nu_k)c(\nu_k)\sigma=-\sigma,$$ 
	where $\overbar \nu_k$ is the unit inner normal vector of $\overbar F_k$ and $\nu_k$ is a Lipschitz unit-length section over $\overbar F_k$ with values in $f^*T\target$.	
	Let $\mathscr R$ and $\mathscr A$ be the endomorphisms on $E$ that appears in the Stokes formula for $D$ (see \eqref{eq:Stokes2}):
	$$\int_{\domain} |D\sigma|^2=\int_{\domain}|\nabla\sigma|^2+ \int_{\domain}\langle\mathscr R\sigma,\sigma\rangle+\sum_{k}\int_{\overbar F_k}\langle \mathscr A\sigma,\sigma\rangle$$
	for  $\sigma\in C^\infty_{00}(\domain,E;B)$. We say that the boundary condition $B$ is \emph{extremal} if we have $\mathscr R\geq 0$  on $\domain$ and $\mathscr A\geq 0$  on each $\overbar F_k$.
\end{definition}
		
We will use the approximation lemma only in the case where the target bundle is trivial. We therefore state the lemma in that form.

\begin{lemma}\label{lemma:approximation}
	Let $(\domain,\overbar g)$ be a $3$-dimensional polyhedral manifold and $D$ be the Dirac operator associated with $E=S(T\domain\oplus \underline{\mathbb{R}}^3)$ over $\domain$, where $\underline{\mathbb{R}}^3$ is the trivial flat bundle. Consider a family of boundary conditions $\{B_t\}_{t \in [0, \varepsilon]}$ defined on each codimension-one face $\overbar F_k$ by
	$$\mathscr E\overbar c(\overbar \nu_k)c(\nu_{k,t})\sigma=-\sigma,$$
	where $\overbar \nu_k$ is the unit inner normal vector of $\overbar F_k$ and $\nu_{k,t}$ is a Lipschitz section over $\overbar F_k$ taking values in $\underline{\mathbb{R}}^3$. Assume that $B_t$ depends continuously on $t$ in Lipschitz norm;  that is, for each codimension one  face $\overbar F_k$,  the maps 
	\[t \mapsto \nu_{k,t} \textup{ and } t \mapsto \nabla\nu_{k,t}\] are continuous with respect to the $L^\infty$-norm. If
	\begin{enumerate}
		\item $B_0$ is extremal in the sense of Definition \ref{def:extremalBoundaryCondition}, and
		\item for any $t\in(0,\varepsilon]$, there exists a non-zero $\sigma_t\in H^1(\domain,E;B_t)$ such that $D\sigma_t=0$,		
	\end{enumerate}
	then there exists a non-zero section $\sigma_0\in H^1(\domain,E;B_0)$ such that $\nabla\sigma_0=0$.
\end{lemma}
\begin{proof}
	Without loss of generality, assume that $\|\sigma_t\|_{L^2(\domain)} = 1$ for each $t \in (0, \varepsilon]$. Let $\mathscr R_t$ and $\mathscr A_t$ denote the endomorphisms appearing in the Stokes formula of $D$ (see \eqref{eq:Stokes2}) for the boundary condition $B_t$. Note that $\mathscr R_t \equiv \mathscr R$ is independent of $t$. Since $B_0$ is extremal, we have $\mathscr R \ge 0$ on $\domain$ and $\mathscr A_0 \ge 0$ on $\partial \domain$. Moreover, since $B_t \to B_0$  in Lipschitz norm, $\mathscr A_t$ converges to $\mathscr A_0$ in $L^\infty$ norm as $t \to 0$. Hence there exist constants $\alpha_t\geq0$, with $\alpha_t\to0$ as
	$t\to0$, such that, 
	$$ \langle \mathscr A_t \sigma(x), \sigma(x) \rangle \ge -\alpha_t |\sigma(x)|^2 \quad \text{for a.e. } x\in \overbar F_k \textup{ and }  \sigma \in C_{00}^\infty(\domain, E). $$
	
	Applying the Stokes formula (see \eqref{eq:Stokes2}) and using the fact that $D\sigma_t = 0$, we have
	$$ 0 = \int_{\domain} |D\sigma_t|^2 = \int_{\domain} |\nabla \sigma_t|^2 + \int_{\domain} \langle \mathscr R \sigma_t, \sigma_t \rangle + \int_{\partial \domain} \langle \mathscr A_t \sigma_t, \sigma_t \rangle. $$
	Using $\mathscr R\geq 0$ and the lower bound for $\mathscr A_t$, we obtain
	$$ \int_{\domain} |\nabla \sigma_t|^2 \le - \int_{\domain} \langle \mathscr R \sigma_t, \sigma_t \rangle - \int_{\partial \domain} \langle \mathscr A_t \sigma_t, \sigma_t \rangle \le \alpha_t \int_{\partial \domain} |\sigma_t|^2. $$
	Let $\tau \colon H^1(\domain, E) \to L^2(\partial \domain, E)$ be the trace operator. Since $\tau$ is bounded, there exists a constant $C > 0$ such that $\|\tau\sigma\|_{L^2(\partial\domain)}
	\leq
	C\|\sigma\|_{H^1(\domain)}$ for all $\sigma\in H^1(\domain, E)$. Therefore
	$$ \|\nabla \sigma_t\|^2_{L^2(\domain)} \le \alpha_t \|\tau \sigma_t\|^2_{L^2(\partial\domain)} \le C^2 \alpha_t \|\sigma_t\|^2_{H^1(\domain)} = C^2 \alpha_t (\|\sigma_t\|^2_{L^2(\domain)} + \|\nabla \sigma_t\|_{L^2(\domain)}^2). $$
	Rearranging the terms yields
	$$ (1 - C^2 \alpha_t) \|\nabla \sigma_t\|_{L^2(\domain)}^2 \le C^2 \alpha_t \|\sigma_t\|_{L^2(\domain)}^2. $$
	For sufficiently small $t$, we have $1 - C^2 \alpha_t > 0$. Since $\|\sigma_t\|_{L^2(\domain)} = 1$, it follows that
	$$ \|\nabla \sigma_t\|_{L^2(\domain)}^2 \le \frac{C^2 \alpha_t}{1 - C^2 \alpha_t}. $$
	This implies that $\|\nabla \sigma_t\|_{L^2(\domain)} \to 0$ as $t \to 0$.
	
	The family $\{\sigma_t\}$ is bounded in $H^1(\domain, E)$. By the Rellich compact embedding theorem, there exists a sequence $\{\sigma_{t_n}\}$ that converges strongly in  $L^2(\domain, E)$ to some limit $\sigma_0$, as $t_n\to 0$. Furthermore, since $\|\nabla \sigma_{t_n}\| \to 0$, the sequence $\{\sigma_{t_n}\}$ is a Cauchy sequence in $H^1(\domain, E)$, and consequently converges strongly to $\sigma_0$ in $H^1(\domain, E)$.
     In particular, we have $\|\sigma_0\| = 1$ (so $\sigma_0 \ne 0$) and $\nabla \sigma_0 = 0$. 
     
     Finally, by the boundedness  of the trace operator and the convergence of the boundary conditions $B_t \to B_0$, the limit section $\sigma_0$ satisfies the boundary condition $B_0$.
\end{proof}

We have the following immediate corollary of Lemma \ref{lemma:approximation}.

\begin{corollary}\label{coro:approximation}
		Let $(\domain,\overbar g)$ be a $3$-dimensional polyhedral manifold and $D$ be the Dirac operator associated with $E=S(T\domain\oplus \underline{\mathbb{R}}^3)$ over $\domain$. Consider a family of boundary conditions $\{B_t\}_{t \in [0, \varepsilon]}$ defined on each codimension-one face $\overbar F_k$ by
		$$\mathscr E\overbar c(\overbar \nu_k)c(\nu_{k,t})\sigma=-\sigma,$$
		where $\overbar \nu_k$ is the unit inner normal vector of $\overbar F_k$ and $\nu_{k,t}$ is a Lipschitz section over $\overbar F_k$ taking values in the trivial bundle $\underline{\mathbb{R}}^3$. Assume that $B_t$ depends continuously on $t$ in Lipschitz norm. If
	\begin{enumerate}
		\item $B_0$ is extremal in the sense of Definition \ref{def:extremalBoundaryCondition}, and
		\item there is no non-zero solution $\sigma\in H^1(\domain,E;B_0)$ of $D\sigma=0$,	
		\end{enumerate}
		then there exists $\delta>0$ such that for all $t\in[0,\delta]$, $D$ subject to the boundary condition $B_t$ does not admit a nonzero solution in $H^1(\domain,E;B_t)$.
\end{corollary}

We emphasize that in Lemma~\ref{lemma:approximation} and Corollary~\ref{coro:approximation}, the operator $D$ subject to the boundary condition $B_t$ is \emph{not} required to be essentially self-adjoint. In the
applications below, $D$ with boundary condition $B_t$ will typically be
essentially self-adjoint for $t>0$, but essential self-adjointness of $D$ with boundary condition $B_0$ may be unknown and, in any case, is not needed.

The following lemma and its corollary provide concrete examples of extremal boundary conditions for which the associated Dirac operator is invertible. 

\begin{lemma}\label{lemma:modelMixedSmooth}
	Let $\mathbb G \subset \R^3$ be a three dimensional compact convex polyhedral manifold.  Assume that the boundary of $\mathbb G$ consists of codimension one 
	faces
	\[
	F_0,F_1,\ldots,F_K,
	\]
	where the second fundamental form of each face is
	nonnegative, $F_0$ meets all other faces orthogonally, and moreover the second fundamental form of $F_0$ is strictly positive in
	a neighborhood of some point $x_0\in F_0$. Let $B_{\textup{mix}}$ be the boundary condition  on $\Bigwedge^*T\mathbb G$ given by
	$$ \mathscr E \overbar c(\nu_k)c(\nu_k)\omega = -\omega \quad \textup{on } F_k \textup{ for } k > 0, $$
	and
	$$ \mathscr E \overbar c(\nu_0)c(\nu_0)\omega = \omega \quad \textup{on } F_0, $$
	where $\nu_k$ is the unit inner normal vector field of $F_k$ for $k \geq 0$. Then $B_{\textup{mix}}$ is extremal, and  the de Rham operator $D^\dR$ subject to $B_{\textup{mix}}$ is essentially self-adjoint and invertible. In particular, there is no non-zero parallel section in $H^1(\mathbb G, \Bigwedge^*T\mathbb G; B_{\textup{mix}})$.
\end{lemma}

\begin{proof}
It is clear that the inner-product comparison conditions as in Theorem \ref{thm:ess-saVector} are satisfied.	The essential self-adjointness of $D^\dR$ follows from Theorem \ref{thm:ess-saVector}. Since $\mathbb G$ is flat and all faces have nonnegative second fundamental form,  $B_{\textup{mix}}$ is extremal by Lemma \ref{lemma:secondff>=} and Proposition \ref{prop:smooth>=}.
	
	Let  $\omega \in H^1(\mathbb G, \Bigwedge^*T\mathbb G; B_{\textup{mix}})$ satisfy  $D^\dR\omega = 0$. It follows from Proposition \ref{prop:smooth>=}  that $\nabla \omega = 0$. Thus $\omega$ is a parallel, hence constant, differential form. On $F_0$, the boundary condition gives 
	$$  \mathscr E\overbar c(\nu_0) c(\nu_0)\omega = \omega,$$
	equivalently,  
	$$  \overbar c(\nu_0)\omega = \mathscr E c(\nu_0)\omega. $$ Since the second fundamental form of $F_0$ is strictly positive near $x_0$, the Gauss map of $F_0$ has image containing an open subset of the unit sphere.
	Thus the preceding identity holds for all unit vectors in an open subset of
	$\mathbb S^2$. By linearity, it follows that
	\begin{equation}\label{eq:mixboundary}
			\overbar c(v)\omega
		=
		\mathscr E\,c(v)\omega
	\end{equation}
	for every $v\in\mathbb R^3$.
	
	On the other hand, on any face $F_k$ with $k>0$, the boundary condition is
	\[
	\mathscr E\,\overbar c(\nu_k)c(\nu_k)\omega=-\omega,
	\]
	which is equivalent to
	\[
	\overbar c(\nu_k)\omega
	=
	-\mathscr E\,c(\nu_k)\omega .
	\]
	Combining this with the identity \eqref{eq:mixboundary} with $v=\nu_k$ gives
	\[
	\overbar c(\nu_k)\omega=0.
	\]
	Since the operator $c(\nu_k)$ is invertible, it follows that 
	$\omega=0$.

	 Thus the kernel of $D^{\dR}_{B_{\mathrm{mix}}}$ is trivial. Since the operator
	 is self-adjoint and Fredholm, it is invertible.
\end{proof}

\begin{example}\label{example:hypersurface}
	Let $\mathcal C \subset \mathbb{R}^3$ be a closed polyhedral cone with its vertex at the origin, bounded by  planes $H_1, \dots, H_\ell$ passing through the origin. For each plane $H_j$, let $\nu_j$ denote the  unit inner  normal vector. Let
	$w_1,\ldots,w_J$ be unit vectors along the edges of the cone, pointing away
	from the origin.
	
	We construct a smooth hypersurface $\Sigma \subset \mathcal C$ with the following conditions:
	\begin{enumerate}
		\item the second fundamental form of $\Sigma$ is nonnegative everywhere;
		\item $\Sigma$ is strictly convex away from neighborhoods of the edges of
		$\mathcal C$;
		\item near each edge, $\Sigma$ coincides with a flat plane;
		\item $\Sigma$ intersects each boundary face $H_j$ orthogonally.
	\end{enumerate}
	Choose $0<a\ll b\ll 1$ and let $m_a\colon \mathbb{R} \to \mathbb{R}$ be smooth and nonneaative with 
	\begin{enumerate}[label=$(\roman*)$]
		\item $m_a(t) = 0$ for $t \leq -a$,
		\item $m_a(t) = t$ for $t \geq a$,
		\item $0 \leq m_a'(t) \leq 1$ for all $t \in \mathbb{R},$
		\item $m_a''(t) > 0$ for $t \in (-a, a)$.\end{enumerate}
	For each edge vector $w_i$, we define a linear function $L_i \colon \mathbb{R}^3 \to \mathbb{R}$:$$L_i(x) = \frac{x \cdot w_i}{1 - b}.$$
	Define a function $\varphi \colon \mathbb{R}^3 \to \mathbb{R}$:$$\varphi(x) = \|x\| + \sum_{i=1}^J m_a\big(L_i(x) - \|x\|\big).$$ Set $$\Sigma \coloneqq  \{ x \in \mathcal C \mid \varphi (x) = 1 \}$$
	and $$ \mathcal R_i = \{ x \in \Sigma \mid L_i(x) - \|x\| \geq -a \}.$$
	We may choose $0<a\ll b\ll1$ so that the sets $\mathcal R_i$ are pairwise disjoint and so that each $\mathcal R_i$ meets only those faces which contain $w_i$. Then  near any point $x \in \Sigma$, at most one summand in the summation $\sum_{i=1}^J m_a \big(L_i(x) - \|x\|\big)$ is non-zero. 
	Now a direct computation shows that, for sufficiently small $a$ and $b$, the hypersurface $\Sigma$ satisfies the required properties (1) through (4).
\end{example}
We have the following immediate consequence of Lemma \ref{lemma:modelMixedSmooth}.
\begin{corollary}\label{lemma:modelMixedCorner}
	Let $\fiber \subset \R^3$ be the compact convex polyhedral manifold bounded by the planes $H_1, \cdots, H_\ell$ and  $\Sigma$ as in Example \ref{example:hypersurface}.  Let $\nu_j$ be the 
	unit inner normal to $H_j$, and let $\nu_0$ be the unit inner normal to $\Sigma$.
	Consider the boundary condition $B_{\mathrm{mix}}$ on
	$\Bigwedge^* T\fiber$ given by
	\[
	\mathscr E\,\overbar c(\nu_j)c(\nu_j)\omega=-\omega
	\qquad
	\text{on }H_j,
	\]
	and
	\[
	\mathscr E\,\overbar c(\nu_0)c(\nu_0)\omega=\omega
	\qquad
	\text{on }\Sigma .
	\]
	Then $B_{\textup{mix}}$ is extremal, and  the de Rham operator $D^\dR$ subject to $B_{\textup{mix}}$ is essentially self-adjoint and invertible. In particular, there is no non-zero parallel section in $H^1(\fiber, \Bigwedge^*T\fiber; B_{\textup{mix}})$.
\end{corollary}

\section{An index theorem for  manifolds with polyhedral boundary} \label{sec:index-poly}

In this section, we prove our main theorem, Theorem~\ref{thm:main2}. As
explained in the introduction, we do not work directly with the Dirac operator
$D_B$ arising from the geometric assumptions of Theorem~\ref{thm:main2}.
Instead, we consider a sequence of Dirac operators with approximating boundary
conditions. We show that each approximating operator is Fredholm with nonzero
Fredholm index, and hence has a nontrivial kernel. We then apply the
approximation lemma, Lemma~\ref{lemma:approximation}, to obtain a nontrivial
solution for the limiting boundary condition $B$. Finally, applying
Lemma~\ref{lemma:angleRigidity} to this limiting solution completes the proof
of Theorem~\ref{thm:main2}.

\subsection{An algebraic angle enlarging lemma}

 Our strategy for computing the Fredholm index combines continuous deformation
 with the gluing formula from Section~\ref{sec:gluing}. One of the main
 difficulties is to determine whether the polyhedral manifolds in
 Theorem~\ref{thm:main2} admit a deformation that preserves the dihedral-angle
 comparison, or equivalently the corresponding inner-product comparison. This
 preservation is crucial for Fredholmness, in view of
 Theorem~\ref{thm:ess-saVector}. See also Gromov's
 $\angle$-shrinking conjecture \cite[Section~7]{Gromov:2022tr}.
 
 Theorem~\ref{thm:ess-saVector} provides an additional flexibility: the sections $\nu_k$ appearing in the boundary condition need not be the inward unit
 normals of the faces of the target polyhedral manifold $\target$. Instead,
 they may be prescribed independently as unit vector fields in the relevant
 auxiliary bundle, provided that the required inner-product comparisons are
 satisfied. This observation removes the geometric and combinatorial
 restrictions imposed by actual polyhedral deformations and reduces the problem
 to an algebraic deformation theorem for vector fields, as stated below.
 
 \begin{remark}
 	The main reason for using the algebraic deformation approach, namely
 	Lemma~\ref{lemma:algAngle} and
 	Proposition~\ref{thm:algebraicDeformation}, is that Gromov's
 	$\angle$-shrinking conjecture remains open in higher dimensions. Since the
 	present paper focuses on the three-dimensional case, one could instead use a
 	more geometric deformation argument, taking advantage of the fact that
 	Gromov's $\angle$-shrinking conjecture is known in dimension three; see
 	\cite[Appendix~B]{Wang:2021tq}. However, our aim is also to highlight several
 	of the key ingredients in our proof \cite{Wang:2021tq} of Gromov's dihedral
 	rigidity conjecture in arbitrary dimensions. For this reason, we use the
 	algebraic deformation framework here, which is closer to the higher-dimensional
 	strategy.
 \end{remark}

\begin{lemma}\label{lemma:algAngle}
	Let $\{v_i\}_{1 \leq i \leq q} \subset \sph^n$ be a set of distinct points\footnote{In typical geometric settings, each $v_i$ arises as the unit inner normal vector of a hyperplane in $\mathbb{R}^{n+1}$.} in the open upper hemisphere, and let $\normal$ denote the north pole. Then, there exists a family of continuous paths $\{v_i(t)\}_{1 \leq i \leq q}$ for $t \in [0, 1]$ in the open upper hemisphere such that:
	\begin{itemize}
		\item $v_i(0) = v_i$ and $v_i(1) = \normal$ for all $i$,
		\item $\langle v_i(t), \normal \rangle > 0$ for all $i$ and $t \in [0, 1]$,
		\item $\langle v_i(t), v_j(t) \rangle$ is non-decreasing with respect to $t$ for all $i, j$, and
		\item $\langle v_i(t), v_j(t) \rangle > \langle v_i, v_j \rangle$ for all $i \neq j$ and $t \in (0, 1]$.
	\end{itemize}
\end{lemma}
\begin{proof}Let$$\theta_i=\dist(v_i,\normal)\in [0,\pi/2).$$For each $i$, let $\gamma_i:[0,\theta_i]\to \sph^n$ be the unit-speed geodesic from $\normal$ to $v_i$, so that$$\gamma_i(0)=\normal,\qquad \gamma_i(\theta_i)=v_i.$$Define$$v_i(t)=\gamma_i((1-t)\theta_i),\qquad t\in[0,1].$$Then $v_i(0)=v_i$, $v_i(1)=\normal$, and$$\langle v_i(t),\normal\rangle=\cos((1-t)\theta_i)>0,$$so each path stays in the open upper hemisphere. It remains to prove the monotonicity of the pairwise inner products. The case where $i=j$ is trivial. Now fix
	$i\neq j$. If one of the two points is already $\normal$, say
	$v_i=\normal$, then $v_i(t)=\normal$, and
	$$
	\langle v_i(t),v_j(t)\rangle
	=
	\cos((1-t)\theta_j),
	$$
	which is non-decreasing in $t$, and is strictly larger than
	$\cos\theta_j=\langle v_i,v_j\rangle$ for every $t>0$, since
	$\theta_j>0$.
	
	Now assume $\theta_i,\theta_j>0$. Let $\phi\in[0,\pi]$ be the angle at
	$\normal$ between the two geodesics from $\normal$ to $v_i$ and $v_j$.
	Put $\lambda=1-t$. By the spherical law of cosines,
	$$
	F(\lambda):=\langle v_i(t),v_j(t)\rangle
	=
	\cos(\lambda\theta_i)\cos(\lambda\theta_j)
	+
	\sin(\lambda\theta_i)\sin(\lambda\theta_j)\cos\phi.
	$$
	We show that $F(\lambda)$ is strictly decreasing for $\lambda\in (0, 1]$, which is equivalent 
	to proving $\langle v_i(t),v_j(t)\rangle$ is strictly increasing for  $t\in [0, 1)$.
	
	Differentiating gives
	$$
	\begin{aligned}
		F'(\lambda)
		&=
		-(\theta_i-\theta_j)\sin(\lambda(\theta_i-\theta_j))  \\
		&\quad
		-
		(1-\cos\phi)
		\left[
		\theta_i\cos(\lambda\theta_i)\sin(\lambda\theta_j)
		+
		\theta_j\sin(\lambda\theta_i)\cos(\lambda\theta_j)
		\right].
	\end{aligned}
	$$
	Since $\theta_i,\theta_j\in[0,\pi/2)$, we have
	$$
	(\theta_i-\theta_j)\sin(\lambda(\theta_i-\theta_j))\ge  0.
	$$
	Also $1-\cos\phi\ge 0$, and the term in the brackets is nonnegative. Hence
	$$
	F'(\lambda)\le 0.
	$$
	Moreover, for $\lambda \in (0, 1]$, equality can occur only if $\theta_i=\theta_j$ and
	$\phi=0$, which would imply $v_i=v_j$, contradicting the assumption that
	the points are distinct. Therefore $F'(\lambda)<0$ for $\lambda\in(0,1]$.
	
	Thus, for every $t\in(0,1)$,
	$$
	\langle v_i(t),v_j(t)\rangle>\langle v_i,v_j\rangle.
	$$
	At $t=1$, we have
	$$
	\langle v_i(1),v_j(1)\rangle=1>\langle v_i,v_j\rangle,
	$$
	since $v_i\neq v_j$. This proves both the monotonicity and the strict
	inequality.
\end{proof}

\begin{proposition}\label{thm:algebraicDeformation}
	Let $(\domain,\overbar g)$ be a three-dimensional polyhedral manifold. Let $\normal$ be a smooth\footnote{More precisely, $\normal$ is the restriction to $\partial \domain$ of a smooth section defined on a tubular neighborhood of $\partial \domain$.} unit length section of the trivial bundle $\mathbb R^3$ on $\partial \domain$. For each codimension one face $\overbar F_k$ of $\domain$, let $\nu_k$ be a smooth unit length section of the trivial bundle $\mathbb R^3$ on $\overbar F_k$ such that $
	\langle \normal,\nu_k\rangle>0.$
	Assume moreover that, for every $i\neq j$ and every $x\in \overbar F_i\cap \overbar F_j$, one has
	$$
	\nu_i(x)\neq \nu_j(x).
	$$
	Then there exists a smooth family $\nu_{k,t}$ of unit length sections of the trivial bundle $\mathbb R^3$ on $\overbar F_k$, with $t\in[0,1]$, such that
	\begin{itemize}
		\item $\nu_{k,0}=\nu_k$ and $\nu_{k,1}=\normal$;
		\item $\langle\nu_{k,t},\normal\rangle>0$ for all $t\in [0,1]$ and every $k$;
		\item along $\overbar F_i\cap \overbar F_j$, the function $
		t\mapsto \langle \nu_{i,t},\nu_{j,t}\rangle$ 
		is non-decreasing, and
		$$
		\langle \nu_{i,t},\nu_{j,t}\rangle>\langle \nu_i,\nu_j\rangle
		$$
		for every $i\neq j$ and every $t\in(0,1]$.
	\end{itemize}
\end{proposition}
\begin{proof}
	For each $x\in \overbar F_k$, let
	$$
	\theta_k(x)=\dist_{\sph^2}(\nu_k(x),\normal(x)).
	$$
	Since $\langle \nu_k(x),\normal(x)\rangle>0$, we have $\theta_k(x)<\pi/2$. Let $\eta_k(x,s)$ be the unit-speed minimizing geodesic in $\sph^2$ from $\nu_k(x)$ to $\normal(x)$, with $
	0\leq s\leq \theta_k(x).$
	Define
	$$
	\nu_{k,t}(x)=\eta_k(x,t\theta_k(x)).
	$$
 Hence $\nu_{k,t}$ is a smooth family of unit length sections with $\nu_{k,0}=\nu_k$ and $	\nu_{k,1}=\normal.$
The proposition now follows from  Lemma $\ref{lemma:algAngle}$.
\end{proof}

\subsection{Index theorem for polyhedral manifolds}

In this subsection, we prove the index theorem for three-dimensional
polyhedral manifolds. We isolate the index calculation from the scalar
curvature and mean curvature hypotheses used later in the rigidity argument.
The curvature assumptions will enter only after the index theorem has produced
a nonzero harmonic spinor.

Let $(\target, g)$ be a compact convex polyhedron in Euclidean space $\mathbb{R}^3$, and let $(\domain,\overline g)$ be a spin polyhedral manifold. Let $f \colon \domain \to \target$ be a polyhedral map of nonzero degree. We assume that the dihedral angles of $\domain$ and $\target$ satisfy
\[
\theta_{ij}(\overline g)\leq f^{*}\theta_{ij}(g)
\quad \text{on } \overline F_{ij} = \overline F_i \cap \overline F_j,
\]
for each pair of adjacent codimension-one faces $\overline F_i$ and $\overline F_j$ of $\domain$. 

We will use the following auxiliary vector field. Choose a point
$p_0\in\operatorname{int}(\target)$ and define
\begin{equation}\label{eq:auxiliaryvector}
	\normal(y)=\frac{p_0-y}{\|p_0-y\|},
	\qquad
	y\in\partial\target .
\end{equation}
Since $\target$ is convex, $\normal$ is a smooth unit vector field on each
face of $\partial\target$ and satisfies
\[
\langle\normal,\nu_k\rangle>0
\]
on every codimension-one face $F_k$ of $\target$, where $\nu_k$ is the unit inner normal vector field of $F_k$.

By Proposition~\ref{thm:algebraicDeformation},
$\nu_k$ can be deformed through unit vector fields $\nu_{k,t}$,
$t\in[0,1]$, with $
\nu_{k,0}=\nu_k$ and $\nu_{k,1}=\normal$, 
such that the inner products $\langle\nu_{i,t},\nu_{j,t}\rangle$ are
non-decreasing in $t$. Moreover, for every
$t>0$,
\[
\langle\overbar\nu_i,\overbar\nu_j\rangle
<
f^*\langle\nu_{i,t},\nu_{j,t}\rangle
\]
on each edge of $\domain$.

In what follows, we replace the original normals $\nu_k$ by the deformed vector fields $\nu_{k,t}$. This leads us to formulate the following theorem.

\begin{theorem}\label{thm:indexTheoremVector}
	Let $(\target, g)$ be a compact convex polyhedron in $\mathbb{R}^3$ with $g$ the Euclidean metric, and let $(\domain,\overline g)$ be a spin polyhedral manifold. Let $f \colon \domain \to \target$ be a polyhedral map of nonzero degree. Set
	\[
	E=S(T\domain\oplus f^*T\target)
	\simeq
	S(T\domain\oplus\underline{\mathbb R}^3),
	\]
	where $\underline{\mathbb R}^3$ is the trivial flat bundle,
	and let $D$ be the associated Dirac operator.
	
	Assume that all dihedral angles of $\domain$ are less than $\pi$. For each
	codimension-one face $\overbar F_k$ of $\domain$, let $\overbar\nu_k$ be
	its unit inner normal and $\nu_k$ a smooth unit length section of the trivial
	bundle $\underline{\mathbb R}^3$ over $\overbar F_k$.  Let $\normal$ be the vector field over $\partial\target$ in \eqref{eq:auxiliaryvector}. 	
	Assume that
	\[
	\langle f^*\normal,\nu_k\rangle>0
	\qquad\text{on }\overbar F_k,
	\]
	where $f^*\normal$ is the pull-back of $\normal$, viewed as a unit-length section of $\underline{\R}^3$ over $\partial \domain$.
	Let $B$ be the boundary condition
	\[
	\mathscr E\,\overbar c(\overbar\nu_k)c(\nu_k)\sigma=-\sigma
	\qquad\text{on }\overbar F_k .
	\]
	Assume that, for every adjacent pair $\overbar F_i,\overbar F_j$, we have 
	\begin{equation}\label{eq:strictInnerProductIndex}
		\langle\overbar\nu_i,\overbar\nu_j\rangle
		<
		\langle\nu_i,\nu_j\rangle
	\end{equation}
	on each connected component of $\overbar F_i\cap\overbar F_j$. Then $D$ with
	boundary condition $B$ is essentially self-adjoint and Fredholm. Moreover,
	\[
	\mathrm{ind}(D_B) = \deg(f),
	\]
	where $\deg(f)$ denotes the degree of $f$.
\end{theorem}
\begin{proof}
	The essential self-adjointness and Fredholmness of $D_B$ follows from Theorem \ref{thm:ess-saVector}. So it remains to compute the Fredholm index of $D_B$. The proof is
	by a sequence of index-preserving deformations and cut-and-paste operations,
	reducing the problem to the classical smooth-boundary index theorem.
	
	\textbf{Step 1. Deforming the vector fields $\nu_k$ to $\normal$. }
	
	Let $\normal$ be the smooth unit vector field over $\partial\target$ in \eqref{eq:auxiliaryvector}. 
	For simplicity, we still denote its pull-back via $f$ by $\normal$ , which is a unit-length section of the trivial bundle $\underline{\R}^3$ over $\partial\domain$.

	By Proposition~\ref{thm:algebraicDeformation}, there are smooth unit-length sections $\nu_{k,t}$, $t\in[0,1]$, of the trivial bundle $\underline{\mathbb R}^3$ over $\overbar F_k$ such that
	\[
	\nu_{k,0}=\nu_k,
	\qquad
	\nu_{k,1}=\normal
	\quad \textup{ and } \quad 
	\langle\nu_{k,t},\normal\rangle>0 .
	\]
	Let $B_t$ be the boundary condition
	\[
	\mathscr E\,\overbar c(\overbar\nu_k)c(\nu_{k,t})\sigma=-\sigma
	\qquad\text{on }\overbar F_k .
	\]
	The assumption \eqref{eq:strictInnerProductIndex} and the monotonicity from Proposition \ref{thm:algebraicDeformation} imply that, for every adjacent pair of codimension-one faces $\overbar F_i$ and
	$\overbar F_j$, 
	\[
	\langle\overbar\nu_i,\overbar\nu_j\rangle
	<
	\langle\nu_{i,t},\nu_{j,t}\rangle
	\]
	along $\overbar F_i\cap \overbar F_j$ for all $t\in[0,1]$. Hence $D_{B_t}$ is essentially self-adjoint and Fredholm
	for all $t\in [0, 1]$ by Theorem~\ref{thm:ess-saVector}.
	
	We now prove that the Fredholm index is independent of $t$. The idea is to
	construct a norm-continuous family of bundle isomorphisms
	\[
	\Theta_t\colon E\longrightarrow E
	\]
	such that $\Theta_t$ maps the boundary condition $B_0$ to $B_t$.
	
	We first construct local isomorphisms mapping $B_s$ to $B_t$, for $s,t$ close
	to each other. The construction is the same as the construction of the
	bundle isomorphism in the proof of Theorem~\ref{thm:reduction}.
	
	\begin{enumerate}
		\item At interior points of $\domain$, we take the local isomorphism to be
		the identity.
		
		\item At points in the interior of a codimension-one face, we choose a
		local unitary bundle map carrying the subbundle determined by $B_s$ to
		the corresponding subbundle determined by $B_t$, and carrying the
		orthogonal complement to the orthogonal complement.
		
		\item At points in the interior of an edge
		$\overbar F_i\cap\overbar F_j$, we use the explicit construction from the
		codimension-two case in the proof of Theorem~\ref{thm:reduction}, using the fact that
		the inner product comparison is strict for all $t\in[0,1]$.
		
		\item At a vertex, we use the same partition-of-unity construction as in
		the codimension-three case of Theorem~\ref{thm:reduction}.  Namely, we use an open cover $\{W_\alpha\}$ of the link and a subordinate partition of unity $\{\rho_\alpha\}$ to define the local bundle isomorphism near a vertex by:  
		\[
		\Theta_{s,t}^{\mathrm{loc}}(z)
		=
		\sum_\alpha \rho_\alpha(\sigma)\Theta_{\alpha,s,t}(z),
		\]
		where $\Theta_{\alpha,s,t}$ is a local bundle isomorphism constructed as in the codimension two case. See equation \eqref{eq:bundlegluemap} and its construction for more details. 
	\end{enumerate}

	To pass from local maps to a  bundle isomorphism over $\domain$, we use a partition of unity
	subordinate to a finite open cover of $\domain$. Since the local maps are
	uniformly close to the identity whenever $s$ and $t$ are sufficiently close,
	a linear combination  remains pointwise invertible. Thus, after
	subdividing $[0,1]$ into sufficiently small intervals
	\[
	0=t_0<t_1<\cdots<t_N=1,
	\]
	we obtain global bundle isomorphisms
	\[
	\Theta_{s,t}\colon E\longrightarrow E
	\]
	for $s,t$ in the same subinterval, with $\Theta_{s,t}$ mapping $B_s$ to
	$B_t$. For $t\in[t_\ell,t_{\ell+1}]$, define
	\[
	\Theta_t
	=
	\Theta_{t_\ell,t}\circ
	\Theta_{t_{\ell-1},t_\ell}\circ\cdots\circ
	\Theta_{t_0,t_1}.
	\]
	Then $\Theta_t$ maps $B_0$ to $B_t$.
	
	We next record the Sobolev estimates needed below. Away from the vertices of
	$\domain$, $\Theta_t$ is smooth and depends smoothly on $t$. Near a vertex
	$x$, write $z=r\sigma$ in polar coordinates centered at $x$, where $r = \dist(z, x)$. In the local
	construction above,
	\[
	\Theta_{s,t}^{\mathrm{loc}}(z)
	=
	\sum_\alpha \rho_\alpha(\sigma)\Theta_{\alpha,s,t}(z).
	\]
	Since$
	|\nabla \rho_\alpha(\sigma)|\leq Cr^{-1}$
	and  the maps $\Theta_{\alpha,s,t}$ depend smoothly on $s$ and $t$, we
	obtain, for $s,t$ in the same small subinterval,
	\begin{equation}\label{eq:ThetaDiffEstimate}
		|\Theta_{s,t}(z)-I|
		\leq
		C|s-t|,
		\qquad
		|\nabla\Theta_{s,t}(z)|
		\leq
		C|s-t|\,r^{-1}
	\end{equation}
	near every vertex. More generally, for the composed family $\Theta_t$, we have
	\[
	|\nabla\Theta_t(z)|\leq C' r^{-1}
	\]
	near each vertex.
	
	By Hardy's inequality, Lemma~\ref{lemma:1/rEmbedding}, multiplication by
	$r^{-1}$ maps $H^1$ boundedly into $L^2$ in dimension three. Therefore
	multiplication by $\nabla\Theta_t$ defines a bounded map
	\[
	\nabla\Theta_t\colon H^1(\domain,E)\longrightarrow L^2(\domain,E).
	\]
	Using the product rule $
	\nabla(\Theta_t\sigma)
	=
	(\nabla\Theta_t)\sigma+\Theta_t\nabla\sigma$,
	we conclude that
	\[
	\Theta_t\colon H^1(\domain,E;B_0)
	\longrightarrow
	H^1(\domain,E;B_t)
	\]
	is bounded. The same argument applied to $\Theta_t^{-1}$ shows that
	$\Theta_t$ is a bounded isomorphism between these Sobolev spaces.
	
	Moreover, \eqref{eq:ThetaDiffEstimate} gives, for $s,t$ in the same small
	subinterval,
	\[
	\|(\Theta_t-\Theta_s)\sigma\|_{H^1}
	\leq
	C|t-s|\,\|\sigma\|_{H^1},
	\qquad
	\sigma\in H^1(\domain,E;B_0).
	\]
	Since the interval $[0,1]$ is covered by finitely many such subintervals, the
	family $\Theta_t$ is continuous in operator norm as a map
	\[
	H^1(\domain,E;B_0)\longrightarrow H^1(\domain,E).
	\]
	
	Define
	\[
	\mathcal D_t
	=
	D_{B_t}\circ \Theta_t
	\colon
	H^1(\domain,E;B_0)
	\longrightarrow
	L^2(\domain,E).
	\]
	For $\sigma\in H^1(\domain,E;B_0)$,  we have 
	\[
	\begin{aligned}
		\|\mathcal D_t\sigma-\mathcal D_s\sigma\|_{L^2}
		&=
		\|D((\Theta_t-\Theta_s)\sigma)\|_{L^2}   
		&\leq
		C\|(\Theta_t-\Theta_s)\sigma\|_{H^1}.
	\end{aligned}
	\]
	Hence $\mathcal D_t$ is a norm-continuous family of Fredholm operators on the
	fixed domain $H^1(\domain,E;B_0)$. By homotopy invariance of the Fredholm
	index, $
	\ind(\mathcal D_t)$ 
	is independent of $t$. Since $\Theta_t$ is an isomorphism,
	\[
	\ind(\mathcal D_t)=\ind(D_{B_t}).
	\]
	Therefore
	$\ind(D_{B_t})$
	is constant for $t\in[0,1]$. We may therefore replace the original auxiliary fields
	$\nu_k$ by $\normal$ without changing the index.

	\textbf{Step 2. Metric deformation near vertices.}

We next deform the metric near the vertices. Let $h_0=\overbar g$. For each
vertex $x$ of $\domain$, choose a sufficiently small neighborhood of $x$ in
the ambient open manifold $X$ and choose a flat metric $h^x$ on that
neighborhood such that
\[
h^x(x)=h_0(x).
\]
Using a partition of unity, and choosing the vertex neighborhoods pairwise
disjoint, we glue these local flat metrics to $h_0$ away from the vertices.
This gives a smooth metric $h_1$ on $\domain$ which is flat in a small
neighborhood of every vertex and agrees with $h_0$ away from these
neighborhoods. Define
\[
h_t=(1-t)h_0+t h_1,
\qquad
0\leq t\leq1.
\]
 
Let $\overbar\nu_{k,t}$ denote the unit inner normal to $\overbar F_k$ with
respect to $h_t$. During this step, we have  the auxiliary vector field $\nu_k = \normal$. The corresponding $h_t$-dependent boundary condition $B^{h_t}$ on the spinor bundle $E_{h_t} = S(T\domain_{h_t} \oplus \mathbb R^3)$ is defined by:
\[
\mathscr E\,
\overbar c(\overbar\nu_{k,t})c(\normal)\sigma=-\sigma
\qquad
\text{on }\overbar F_k.
\]

For every adjacent pair of codimension-one faces $\overbar F_i$ and
$\overbar F_j$, the faces remain transverse for all $t$. Hence their
$h_t$-unit inner normals remain distinct, and therefore
\[
\langle\overbar\nu_{i,t},\overbar\nu_{j,t}\rangle_{h_t}
<
1
=
\langle\normal,\normal\rangle .
\]
Thus the strict inner-product comparison in
Theorem~\ref{thm:ess-saVector} holds for each $t\in[0,1]$. Consequently,
$D^{h_t}$ with boundary condition $B^{h_t}$ is essentially self-adjoint and
Fredholm for all $t$.

Because $h_t$ is a smooth family of metrics, there exists a smooth family of bundle isometries $\Phi_t \colon E_{h_t} \to E_{h_0}$ covering the identity map on $\domain$. Define 
\[ \widetilde D_t \coloneqq \Phi_t \circ D^{h_t} \circ \Phi_t^{-1}\] acting on sections of $E_{h_0}$. Similarly, the boundary condition $B^{h_t}$ pulls back to a boundary condition $\widetilde B_t$ on $E_{h_0}$. Then $\widetilde B_t$ and  the  coefficients of $\widetilde D_t$ depend continuously on $t$. 

Similar to 
\textbf{Step 1}, we construct bundle isomorphisms $\widetilde \Theta_t \colon E_{h_0} \to E_{h_0}$ mapping the initial boundary condition $B^{h_0}$ to the pulled-back boundary condition $\widetilde B_t$. Moreover, the same argument from \textbf{Step 1} shows that  $\ind(D^{h_t}_{B^{h_t}})$ is independent of $t\in [0, 1]$. 
	
	\textbf{Step 3. Flattening the codimension one faces near vertices.}
	
	The goal of this step is to flatten the codimension-one faces in small neighborhoods of the vertices of $(\domain, h_1)$. Let $x$ be a vertex of $\domain$. We will deform each codimension-one face $\overline{F}_j$ passing through $x$ to its tangent plane at $x$, within a small neighborhood of $x$. To streamline the discussion, we perform this deformation on one codimension-one face at a time. For each individual deformation, we will show that the index of the associated Dirac operator remains constant. Since there are only finitely many faces and vertices,
	performing these deformations successively gives a new polyhedral manifold
	$(\domain',\overbar h)$ for which the metric $\overbar h$ and all
	codimension-one faces are flat in sufficiently small neighborhoods of the
	vertices.
	
    Recall that $h_1$ is flat near each vertex. We identify a sufficiently small
    neighborhood of $x$ with a neighborhood of the origin in $\mathbb R^3$. Let
    $\overbar F_0$ be one codimension-one face passing through $x$, and let
    $\Sigma$ be the corresponding smooth surface near the origin. After a
    rotation on $\mathbb R^3$, we may assume that
    \[
    T_0\Sigma=\{z=0\}.
    \] Thus $\Sigma$ is the graph of a smooth function $f(x,y)$ satisfying
    \[
    f(0,0)=0,
    \qquad
    \nabla f(0,0)=0.
    \] Choose a smooth cutoff function $\rho$ supported in $B_{r_0}(0)$ and equal to
    $1$ on $B_{r_0/2}(0)$, where $B_{r_0}(0)$ is the ball of radius $r_0$ centered at the origin.  The gradient of $\rho$ satisfies $\|\nabla \rho(r)\| \leq C / r$ for some constant $C>0$. Define a family of surfaces $\Sigma_t$, $t\in[0,1]$, by
    \begin{equation}\label{eq:surfacedeformation}
    	\Sigma_t
    	=
    	\left\{
    	(x,y,z): z=(1-t\rho(x,y))f(x,y)
    	\right\}.
    \end{equation}
    Then $\Sigma_0=\Sigma$, the deformation is supported in $B_{r_0}(0)$, and
    $\Sigma_1$ agrees with the tangent plane $\{z=0\}$ on $B_{r_0/2}(0)$. Let us denote the corresponding surface at time $t$ under this deformation by $\Sigma_t$, and the corresponding new codimension-one face of the manifold by $\overbar F_0^t$. Let
    $\domain^{(t)}$ denote the space obtained from $\domain$ by
    replacing $\overbar F_0$ with the corresponding deformed face
    $\overbar F_0^{(t)}$.
    
  We first check that $\domain^{(t)}$ is a polyhedral manifold in the sense of Definition \ref{def:polymanifold}. 
 A normal to $\Sigma_t$ is
  \[
  \overbar N_t
  =
  \bigl(-(1-t\rho)\nabla f+t f\nabla\rho,\,1\bigr).
  \]
  Since $
  f=O(r^2),  \nabla f=O(r)$
  and 
  \[
  |f\nabla\rho|=O(r)
  \]
  on the support of $\nabla\rho$, we have
  \[
  \overbar N_t=(0,0,1)+O(r)
  \]
  uniformly in $t$. Therefore, after shrinking $r_0$ if necessary, the normal to
  $\overbar F_0^{(t)}$ remains linearly independent from the normals to all
  adjacent fixed faces along their intersections. At the vertex itself, the
  normal is fixed and equal to the normal of $T_0\Sigma$. It follows that $\domain^{(t)}$ satisfies conditions of Definition \ref{def:polymanifold}, hence is a
  polyhedral manifold, for every $t\in[0,1]$. Moreover, after shrinking $r_0$ if necessary,
  all dihedral angles of $\domain^{(t)}$ also remain less than $\pi$. 
   
   To show the invariance of the index of the associated Dirac operator, We construct a
   polyhedral $C^1$-diffeomorphism
   \[
   \Phi_t\colon \domain^{(1)}\longrightarrow \domain^{(t)}
   \]
   by adapting the constructions in Examples~\ref{ex:codimtwo} and
   \ref{ex:codimthree}.
   
    We first construct local maps near the vertex $x$ (which has been identified with the origin of $\mathbb R^3$).  Let
    $\Sigma_t$ be the deformed face, and let $\Sigma_k$ be an adjacent face that
    is kept fixed. We choose Euclidean coordinates near the vertex $x$, with $x$ identified
    with the origin in $\mathbb R^3$, so that
    \[
    T_0\Sigma_t=\Pi_1=\{z=0\},
    \qquad
    T_0\Sigma_k=\Pi_k=\{y\sin\theta-z\cos\theta=0\}.
    \]
    Let
    \[
    g_{t,1}(x,y,z)
    =
    z-(1-t\rho(x,y))f(x,y)
    \]
    be a defining function for $\Sigma_t$, and let $g_k$ be a defining function
    for $\Sigma_k$. We normalize these defining functions so that
    \[
    \nabla g_{t,1}(0)=(0,0,1),
    \qquad
    \nabla g_k(0)=(0,\sin\theta,-\cos\theta).
    \]
    Define
    \[
    \widetilde\Phi_{t,k}(x,y,z)
    =
    \left(
    x,\,
    \frac{g_k(x,y,z)+g_{t,1}(x,y,z)\cos\theta}{\sin\theta},\,
    g_{t,1}(x,y,z)
    \right).
    \]
    Then $\widetilde\Phi_{t,k}$ maps $\Sigma_t$ to $\Pi_1$, maps $\Sigma_k$ to
    $\Pi_k$, and satisfies
    \[
    d\widetilde\Phi_{t,k}|_0=I.
    \]
    Consequently,
    \[
    \Psi_{t,k}
    =
    \widetilde\Phi_{t,k}^{-1}\circ\widetilde\Phi_{1,k}
    \]
    is a local $C^1$-diffeomorphism which maps the flattened face $\Sigma_1$ to
    $\Sigma_t$, preserves the fixed face $\Sigma_k$, satisfies
    \[
    d\Psi_{t,k}|_0=I,
    \]
    and depends smoothly on $t$. 
    
    For each pair of adjacent faces $\Sigma_i$ and $\Sigma_j$ through the vertex,
    we define the corresponding local $C^1$-diffeomorphism as follows:
    \begin{enumerate}
    	\item if one of the two faces is the deforming face $\Sigma_1$ and the
    	other is a fixed face $\Sigma_k$, we use the map $
    	\Psi_{t,k};$
    	
    	\item if neither face is the deforming face $\Sigma_1$, so that both faces are fixed,
    	we use the identity map.
    \end{enumerate}
   Using these local $C^1$-diffeomorphisms, 
   the construction in Example \ref{ex:codimthree} gives a local
  $C^1$-diffeomorphism $
   \Phi_t^{\mathrm{loc}}$   defined in a sufficiently small neighborhood of the vertex, such that
   \[
   \Phi_t^{\mathrm{loc}}(\Sigma_1)=\Sigma_t
   \]
   for the deforming face, while every fixed face is mapped to itself. Moreover,
   \[
   d\Phi_t^{\mathrm{loc}}|_0=I,
   \]
   and the family $\Phi_t^{\mathrm{loc}}$ depends continuously on $t$.
   
   Finally, we extend $\Phi_t^{\mathrm{loc}}$ to a global polyhedral
   $C^1$-diffeomorphism
   \[
   \Phi_t\colon \domain^{(1)}\longrightarrow\domain^{(t)}.
   \]
   Choose the vertex neighborhood slightly larger than the support of the
   surface deformation. Since the deformation of the face is supported in
   $B_{r_0}(0)$, and the faces are fixed outside this ball, 
   $\Phi_t^{\mathrm{loc}}$ may be chosen to be the identity near the boundary of
   a slightly larger ball. We then extend it by the identity outside that ball.
   
   Let $
   E_t=S(T\domain^{(t)}\oplus\underline{\mathbb R}^3).$
   Because $\Phi_t$ is the identity map outside of the flat ball $B_{r_0}(0)$, the pullback bundle $\Phi_t^* E_t$ and  $E_1 = S(T\domain^{(1)} \oplus \underline{\mathbb R}^3)$ are naturally identified. Let $U_t$ be the induced unitary map 
   \[
   U_t\colon L^2(\domain^{(t)},E_t)
   \longrightarrow
   L^2(\domain^{(1)},E_1).
   \]
   Let $D_t$ be the Dirac operator on $\domain^{(t)}$, and set
   \[
   \widetilde D_t=U_tD_tU_t^{-1}.
   \]
   Then $\widetilde D_t$ is a first-order elliptic operator on 
   $E_1$ over $\domain^{(1)}$. Its coefficients are continuous at the vertex,
   smooth away from the vertex, and depend continuously on $t$.

   It remains to track the boundary condition. Extend the auxiliary unit-length section $\normal$ with value in $\underline{\R}^3$ to a smooth unit-length section $\widetilde\normal$ on a tubular
   neighborhood of $\partial\domain$ in the ambient manifold. After shrinking
   $r_0$ if necessary, all deformed faces remain in this tubular neighborhood. Define
   \begin{equation}\label{eq:deformnormal}
   	\normal_t=\widetilde\normal|_{\partial\domain^{(t)}} .
   \end{equation}
   Let $B_t$ be the boundary condition on $\domain^{(t)}$ determined by the
    unit inner normals to the faces of $\domain^{(t)}$ and the auxiliary
   field $\normal_t$. Let $\widetilde B_t$ be its pullback to $\domain^{(1)}$
   under $\Phi_t$. The boundary conditions $\widetilde B_t$ vary continuously in $t$, and the
   construction from Theorem~\ref{thm:reduction} gives Lipschitz bundle
   isomorphisms
   \[
   \Theta_t\colon E_1\longrightarrow E_1
   \]
   mapping $B_1$ to $\widetilde B_t$. By the same argument from \textbf{Step 1} and \textbf{Step 2},  $\ind((\widetilde D_t)_{\widetilde B_t})$ is independent of $t$; equivalently, $\ind((D_t)_{B_t})$ is is independent of $t$.
  
  Applying this deformation successively to all codimension-one faces through
  all vertices gives a polyhedral manifold $(\domain',\overbar h)$ whose metric
  and codimension-one faces are flat near the vertices, and whose associated
  Dirac operator has the same Fredholm index as the original one at the
  beginning of this step. See Figure \ref{fig:deformation-near-vertices} in Section \ref{sec:intro}.

	\textbf{Step 4: Cutting off small neighborhoods of the vertices.}  
	
	At this point, we have reduced to the case in which the metric $\overbar g$ of $(\domain, \overbar g)$ is
	flat in a small neighborhood of each vertex, and all codimension-one faces
	meeting a given vertex are flat in that neighborhood.
	
	Let $x$ be a vertex of $\domain$, and choose a sufficiently small flat
	neighborhood
	\[
	U_x=B_\delta(x)\cap\domain .
	\]
	Using the flat metric on $U_x$, we identify
	\[
	T\domain|_{U_x}\simeq U_x\times\mathbb R^3.
	\]
	Thus the auxiliary trivial bundle $\underline{\mathbb R}^3$ may be identified
	with $T\domain|_{U_x}$ over $U_x$. Under this identification,  the section $\normal$ of the auxiliary bundle $\underline{\mathbb R}^3$ can be
	equivalently viewed as a tangent vector field on $U_x$.
	
	Choose a hypersurface $\Sigma_\varepsilon$ as in
	Example~\ref{example:hypersurface}, contained in
	\[
	B_\varepsilon(x)\cap\domain,
	\qquad
	0<\varepsilon<\delta .
	\]
	Before cutting along $\Sigma_\varepsilon$, we make a local deformation of the
	auxiliary unit-length section $\normal$ near $\Sigma_\varepsilon$. We choose this
	deformation so that:
	\begin{enumerate}[label=$(\roman*)$]
		\item $\normal$ is tangent to $\Sigma_\varepsilon$ along each curve
		$\Sigma_\varepsilon\cap\overbar F_k$;
		
		\item  and $
		\langle\normal,\overbar\nu_k\rangle>0$
		along each flat face $\overbar H_k = \overbar F_k\cap B_\delta(x)$, where $\overbar\nu_k$ is the unit inner normal of $\overbar F_k$.
	\end{enumerate}
	 This deformation can be chosen through smooth unit-length sections
	supported in $U_x$. By the same argument as in \textbf{Step 1}, this local deformation does not change the Fredholm
	index.
	
	We now cut $\domain$ along $\Sigma_\varepsilon$. This gives two pieces,
	\[
	\domain=\domain_1\cup_{\Sigma_\varepsilon}\domain_2.
	\]
	Here $\domain_1$ is the piece containing the vertex, while $\domain_2$ is the truncated remainder.
	
	Let $\overbar \nu_\Sigma$ be the unit normal vector of $\Sigma_\varepsilon$ pointing out of
	$\domain_1$. Equivalently, $\overbar \nu_\Sigma$ is the  unit inner normal vector of
	$\Sigma_\varepsilon$ when $\Sigma_\varepsilon$ is viewed as a boundary face of
	$\domain_2$. We use the same vector $\nu_\Sigma = \overbar \nu_\Sigma$, under the Euclidean
	identification above, as the auxiliary vector of $\underline{\mathbb R}^3$ on $\Sigma_\varepsilon$.
	Because $\normal$ is tangent to $\Sigma_\varepsilon$ along
	$\Sigma_\varepsilon\cap\overbar F_k$, we have
	\[
	\langle\normal,\nu_\Sigma\rangle=0
	\qquad
	\text{along }\Sigma_\varepsilon\cap\overbar F_k .
	\]
	Thus the hypotheses of the gluing theorem, Theorem~\ref{thm:gluing}, are
	satisfied. Hence
	\[
	\ind(D_B^{\domain})
	=
	\ind(D_{B_1}^{\domain_1})
	+
	\ind(D_{B_2}^{\domain_2}),
	\]
	where $B_1$ and $B_2$ are the boundary conditions defined by
	Theorem~\ref{thm:gluing}. In particular, along $\Sigma_\varepsilon$, the
	condition on $\domain_2$ is the absolute condition, while the condition on
	$\domain_1$ is the complementary condition.
	
	\begin{claim}\label{claim:indexzero}
	 We have $
		\ind(D_{B_1}^{\domain_1})=0.$
	\end{claim}
	
	\begin{proof}
		The codimension-one faces of $\domain_1$ consist of the flat faces
		\[
		\overbar H_1,\ldots,\overbar H_\ell
		\]
		together with $\Sigma_\varepsilon$. Let $\overbar\nu_k$ be the  unit inner 
		normal to $\overbar H_k$. We also regard $\overbar\nu_k$ as a section of the
		auxiliary trivial bundle $\underline{\mathbb R}^3$, and denote this auxiliary section by
		\[
		\nu_k=\overbar\nu_k .
		\]
		
		By the local deformation arranged above, the auxiliary field $\normal$ satisfies
		\[
		\langle\normal,\nu_k\rangle>0
		\]
		on each flat face $\overbar H_k$. Therefore Proposition~\ref{thm:algebraicDeformation}
		applies to the collection $\{\nu_k\}$ with north pole $\normal$. We obtain
		smooth unit-length sections $\nu_{k,t}$, $t\in[0,1]$, such that
		\[
		\nu_{k,0}=\nu_k,
		\qquad
		\nu_{k,1}=\normal,
		\]
		and, for $t>0$,
		\[
		\langle\nu_{i,t},\nu_{j,t}\rangle
		>
		\langle\nu_i,\nu_j\rangle
		=
		\langle\overbar\nu_i,\overbar\nu_j\rangle
		\]
		along every intersection $\overbar H_i\cap\overbar H_j$.
		
		We keep the auxiliary vector $\nu_\Sigma$ on $\Sigma_\varepsilon$ fixed during
		this deformation. Since $\Sigma_\varepsilon$ meets each flat face
		$\overbar H_k$ orthogonally, the vector $\nu_k=\overbar\nu_k$ is tangent to
		$\Sigma_\varepsilon$ along $\Sigma_\varepsilon\cap\overbar H_k$. By construction,
		$\normal$ is also tangent to $\Sigma_\varepsilon$ along the same curve. The
		explicit formula in Proposition~\ref{thm:algebraicDeformation}
		therefore shows that $\nu_{k,t}$ remains tangent to $\Sigma_\varepsilon$ along
		$\Sigma_\varepsilon\cap\overbar H_k$ for all $t$. Hence
		\[
		\langle\nu_{k,t},\nu_\Sigma\rangle=0
		=
		\langle\overbar\nu_k,\overbar\nu_\Sigma\rangle
		\]
		along $\Sigma_\varepsilon\cap\overbar H_k$.
		
		Let $\widetilde B_t$ be the boundary condition on $\domain_1$ determined by
		the  normals $\overbar\nu_k$, the auxiliary fields $\nu_{k,t}$ on the
		flat faces, and $\overbar\nu_\Sigma$ and $\nu_\Sigma$ along $\Sigma_\varepsilon$. For
		$t>0$, the strict inner product comparison holds along intersections of flat faces $\overbar H_i \cap \overbar H_j$, while
		the equality case holds along intersections  $\Sigma_\varepsilon\overbar H_k$.
		Thus Theorem~\ref{thm:ess-saVector} applies, and
		$D_{\widetilde B_t}^{\domain_1}$ is Fredholm for every $t>0$.
		
		Moreover, the same boundary-condition deformation argument used in \textbf{Step~1}
		shows that
		$\ind(D_{\widetilde B_t}^{\domain_1})$
		is constant for $t\in(0,1]$.
		
		At $t=0$, the boundary condition $\widetilde B_0$ is the mixed boundary
		condition of Lemma~\ref{lemma:modelMixedCorner}. Therefore $D_{\widetilde B_0}^{\domain_1}$ is invertible,
		and $\widetilde B_0$ is extremal in the sense of
		Definition~\ref{def:extremalBoundaryCondition}. Since
		$\widetilde B_t\to \widetilde B_0$ in Lipschitz norm, Corollary~\ref{coro:approximation}
		implies that $D_{\widetilde B_t}^{\domain_1}$ is invertible for all sufficiently
		small $t>0$. Hence
		\[
		\ind(D_{\widetilde B_t}^{\domain_1})=0
		\]
		for all sufficiently small $t>0$. Since the index is constant on $(0,1]$, we
		obtain
		\[
		\ind(D_{B_1}^{\domain_1})
		=
		\ind(D_{\widetilde B_1}^{\domain_1})
		=
		0.
		\]
		This proves the claim.
	\end{proof}
	
	Applying this argument at every vertex of $\domain$, we obtain a truncated
	manifold $\domain'$ and a boundary condition $B'$ such that
	\[
	\ind(D_B^{\domain})=\ind(D_{B'}^{\domain'}).
	\]
	See Figure \ref{fig:cutting-and-pasting-vertices} and \ref{fig:cutting-and-pasting-vertices-2} in Section \ref{sec:intro}. The codimension-one faces of $\domain'$ are of two types:
	\begin{enumerate}
		\item If a face $\overbar F_k'$ is contained in an original face
		$\overbar F_k$ of $\domain$, then the boundary condition is
		\[
		\mathscr E\,\overbar c(\overbar\nu_k)c(\normal)\sigma=-\sigma,
		\]
		where $\overbar\nu_k$ is the unit inner normal to $\overbar F_k'$.
		
		\item\label{item:bdalongcut}
		If a face is one of the cutting hypersurfaces
		$\Sigma_\varepsilon(x_j)$, then the boundary condition is
		\[
		\mathscr E\,\overbar c(\overbar\nu_\Sigma)c(\nu_\Sigma)\sigma=-\sigma,
		\]
		where, using the local Euclidean identification,
		\[
		\overbar\nu_\Sigma=\nu_\Sigma
		\]
		is the  unit normal inner of the cutting hypersurface  $\domain'$.
	\end{enumerate}
	
	\textbf{Step 5. Product deformation along the original edges of $\domain$.}
	
	Cutting off the vertex neighborhoods in \textbf{Step~4} creates new edges and truncates
	the original edges of $\domain$. To avoid ambiguity, let
	$\mathcal E_{\domain}$ denote the collection of edge segments in the truncated
	manifold $\domain'$ in \textbf{Step~4} that arise from the original edges of $\domain$. The new
	edges contained in the cutting hypersurfaces
	$\Sigma_\varepsilon(x_j)$ are not included in $\mathcal E_{\domain}$.
	
	Let $\Gamma\in\mathcal E_{\domain}$. We deform a sufficiently small tubular
	neighborhood of $\Gamma$ into a genuine Riemannian product
	\[
	I\times\sector,
	\]
	where $I$ is a compact interval and $\sector\subset\mathbb R^2$ is a flat
	sector. At the end of the deformation, $\Gamma$ corresponds to
	$I\times\{0\}$, the two original faces meeting along $\Gamma$ correspond to
	the two side faces $I\times\partial\sector$, and the cutting hypersurfaces
	at the endpoints correspond to the end faces
	$\partial I\times\sector$.
	
	Suppose that
	\[
	\Gamma=\overbar F_1\cap\overbar F_2
	\]
	and that its endpoints lie on the cutting hypersurfaces
	$\Sigma_\varepsilon(x_i)$ and $\Sigma_\varepsilon(x_j)$. We perform the
	deformation in two stages.
	
	\begin{enumerate}
		\item We first deform the metric in a sufficiently small tubular
		neighborhood of $\Gamma$ to a flat product metric. The deformation is
		chosen to be fixed near the two end faces, where the metric is already
		flat. The faces remain transverse throughout the deformation, and the
		auxiliary section on both side faces is the same unit-length section $\normal$.
		 The deformation argument from
		\textbf{Step~2} therefore shows that the Fredholm index remains unchanged.
		
		\item We next deform the faces $\overbar F_1$ and $\overbar F_2$ so that
		the neighborhood becomes a genuine product $I\times \mathbb G$.
		
		By Example~\ref{example:hypersurface}, the end faces
		$\Sigma_\varepsilon(x_i)$ and $\Sigma_\varepsilon(x_j)$ are flat near
		their intersections with $\Gamma$ and are orthogonal to $\Gamma$. Let
		$\alpha_i$ and $\alpha_j$ denote the dihedral angles between
		$\overbar F_1$ and $\overbar F_2$ near the two endpoints of $\Gamma$. If
		$\alpha_i\neq\alpha_j$, we first deform one of the side faces in a small
		collar of one endpoint so that the two end angles agree. This deformation
		may be realized by rotating the corresponding cross-sectional ray about
		$\Gamma$. It fixes $\Gamma$ and the end face, preserves the
		orthogonality to the end face, and may be chosen so that the dihedral
		angle remains in $(0,\pi)$ throughout the deformation.
		
		Once the two end angles agree, we flatten $\overbar F_1$ and
		$\overbar F_2$ within a small  tubular neighborhood of $\Gamma$ by the same  procedure 
		used in \textbf{Step~3}. The resulting neighborhood is isometric to
		$I\times\sector$.
		
		During this face deformation, the auxiliary section $\normal$ is
		deformed as in \textbf{Step~3}; see \eqref{eq:deformnormal}.
		The deformation may be chosen  so that
		$\normal_t$ remains tangent to
		$\Sigma_\varepsilon(x_i)$ and $\Sigma_\varepsilon(x_j)$ along their
		intersections with the side faces. Thus, for $a=1,2$ and
		$\ell\in\{i,j\}$,
		\[
		\langle\overbar\nu_{a,t},
		\overbar\nu_{\Sigma_\varepsilon(x_\ell)}\rangle
		=
		0
		=
		\langle\normal_t,\nu_{\Sigma_\varepsilon(x_\ell)}\rangle
		\]
		along
		$\overbar F_a\cap\Sigma_\varepsilon(x_\ell)$.
		Consequently, the strict inner product comparison holds along $\Gamma$, while
		the equality case holds along the edges $\overbar F_a\cap\Sigma_\varepsilon(x_\ell)$.
		Theorem~\ref{thm:ess-saVector} therefore applies throughout the
		deformation.
		
		Finally, the same argument from \textbf{Step~3} shows that  the Fredholm index remains unchanged during this
		face-flattening deformation.
	\end{enumerate}
	
	Choosing the tubular neighborhoods sufficiently small, and performing the
	construction successively for all $\Gamma\in\mathcal E_{\domain}$, we obtain
	a manifold for which a neighborhood of every edge in $\mathcal E_{\domain}$ is a genuine
	product $I\times\sector$, without changing the Fredholm index.

\textbf{Step 6. Smoothing the edges of $\domain$}

Let $\Gamma\in \mathcal E_{\domain}$ be one of the  edges treated in \textbf{Step~5}. A sufficiently
small neighborhood of $\Gamma$ is now a genuine Riemannian product $
I\times\sector,$
where $\Gamma=I\times\{0\}$ and $\sector\subset\mathbb R^2$ is a flat sector. Choose $\varepsilon>0$ sufficiently small, and set
\[
K_\varepsilon=B_\varepsilon(0)\cap\sector,
\qquad
R_\varepsilon=\partial B_\varepsilon(0)\cap\sector,
\qquad
C_\varepsilon=I\times R_\varepsilon .
\]
We will cut along the hypersurface $C_\varepsilon$, following the same
strategy as in \textbf{Step~4}.

Let $\overbar F_1$ and $\overbar F_2$ be the two side faces meeting along
$\Gamma$. Before cutting, we make a local deformation of the auxiliary field
$\normal$ on these two faces. The deformation is supported in small
neighborhoods of
\[
C_\varepsilon\cap\overbar F_1
\qquad\text{and}\qquad
C_\varepsilon\cap\overbar F_2
\]
and is disjoint from the edge $\Gamma$. We arrange that, near these
intersections, $\normal$ is deformed to become the unit inner normal $\overbar\nu_k$ of  $\overbar F_k$. Moreover, the deformation is chosen relative to the two end faces
\[
\Sigma_\varepsilon(x_i)=\{0\}\times K_\varepsilon,
\qquad
\Sigma_\varepsilon(x_j)=\{1\}\times K_\varepsilon,
\]
so that $\normal$ remains tangent to these end faces along their intersections
with $\overbar F_1$ and $\overbar F_2$. The auxiliary vector $\nu_{\Sigma}$ (see item \eqref{item:bdalongcut} at the end of \textbf{Step 4})on each of the end faces $\Sigma_\varepsilon(x_i)$ and $\Sigma_\varepsilon(x_j)$ is kept fixed. The same  argument from \textbf{Step 1} and \textbf{Step 2} shows that the index of the associated Dirac operator (subject to the corresponding boundary condition) remains  constant throughout the deformation.

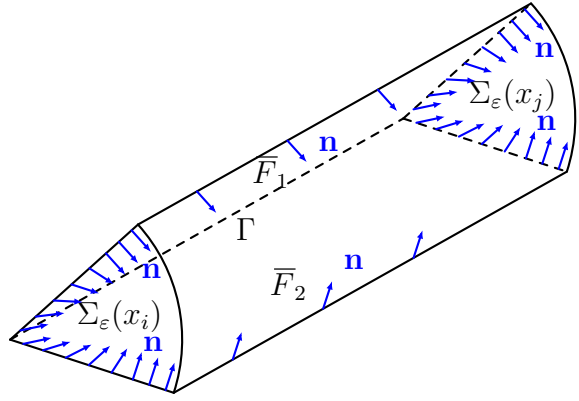
\begin{figure}

\begin{tikzpicture}[line cap=round, line join=round, thick, scale=0.65,  >={Latex[scale=0.5]}]
	
	\def\R{3.5}          
	\def\angTop{42}      
	\def\angBot{-18}     
	\def\vx{8}           
	\def\vy{4.5}         
	
	\coordinate (H) at (0,0);
	\coordinate (T) at (\angTop:\R);
	\coordinate (B) at (\angBot:\R);
	
	\coordinate (Hp) at (\vx, \vy);
	\coordinate (Tp) at ($(T) + (\vx, \vy)$);
	\coordinate (Bp) at ($(B) + (\vx, \vy)$);
	
	\draw[dashed] (Hp) -- (Tp);
	\draw[dashed] (Hp) -- (Bp);
	\draw (Tp) arc (\angTop:\angBot:\R);
	
	\draw[dashed] (H) -- (Hp) node[pos=0.6, below] {$\Gamma$};   
	\draw (T) -- (Tp);   
	\draw (B) -- (Bp);   
	
	\draw (H) -- (T);
	\draw (H) -- (B);
	\draw (T) arc (\angTop:\angBot:\R);
	
	
	\coordinate (CenterF1) at ($0.25*(H) + 0.25*(T) + 0.25*(Hp) + 0.25*(Tp)$);
	\node at (CenterF1) {$\overbar F_1$};
	
	\coordinate (CenterF2) at ($0.25*(H) + 0.25*(B) + 0.25*(Hp) + 0.25*(Bp)$);
	\node[below] at (CenterF2) {$\overbar F_2$};
	
	\pgfmathsetmacro{\midAng}{(\angTop+\angBot)/2}
	\node at ($(H) + (\midAng:0.65*\R)$) {$\Sigma_\varepsilon(x_i)$};
	\node at ($(Hp) + (\midAng:0.65*\R)$) {$\Sigma_\varepsilon(x_j)$};
	
	
	\foreach \frac/\angShift in {0.95/0, 0.85/0, 0.75/0, 0.62/15, 0.48/30, 0.34/45, 0.20/55, 0.08/60} {
		\coordinate (PF1) at (\angTop:\frac*\R);
		\draw[->, blue, thick] (PF1) -- ++(\angTop-90+\angShift:0.6);
		
		\coordinate (PF2) at (\angBot:\frac*\R);
		\draw[->, blue, thick] (PF2) -- ++(\angBot+90-\angShift:0.6);
	}
	\node[blue, below] at ($(\angTop:0.95*\R) + (\angTop-90:0.6)$) {$\normal$};
	\node[blue, above left] at ($(\angBot:0.95*\R) + (\angBot+90:0.6)$) {$\normal$};
	
	\foreach \frac/\angShift in {0.95/0, 0.85/0, 0.75/0, 0.62/15, 0.48/30, 0.34/45, 0.20/55, 0.08/60} {
		\coordinate (PF1p) at ($(Hp) + (\angTop:\frac*\R)$);
		\draw[->, blue, thick] (PF1p) -- ++(\angTop-90+\angShift:0.6);
		
		\coordinate (PF2p) at ($(Hp) + (\angBot:\frac*\R)$);
		\draw[->, blue, thick] (PF2p) -- ++(\angBot+90-\angShift:0.6);
	}
	\node[blue, below] at ($(Hp) + (\angTop:0.95*\R) + (\angTop-90:0.6)$) {$\normal$};
	\node[blue, above left] at ($(Hp) + (\angBot:0.95*\R) + (\angBot+90:0.6)$) {$\normal$};
	
	\foreach \frac in {0.15, 0.38, 0.61} {
		\coordinate (PT) at ($(T)!\frac!(Tp)$);
		\draw[->, blue, thick] (PT) -- ++(\angTop-90:0.6);
		
		\coordinate (PB) at ($(B)!\frac!(Bp)$);
		\draw[->, blue, thick] (PB) -- ++(\angBot+90:0.6);
	}
	
	\node[blue, above right] at ($(T)!0.38!(Tp) + (\angTop-90:0.6)$) {$\normal$};
	\node[blue, above right] at ($(B)!0.38!(Bp) + (\angBot+90:0.6)$) {$\normal$};
	
\end{tikzpicture}
\caption{A product neighborhood $I\times\sector$ of an edge $\Gamma$.
	The blue arrows indicate the auxiliary vector $\protect\normal$ along
	the side faces $\protect\overbar F_1$ and $\protect\overbar F_2$ after the
	local deformation near $I\times R_\varepsilon$.}
\label{fig:normal}
\end{figure}

We now cut $\domain'$ along $C_\varepsilon$. This gives $\domain' = \domain'_1 \cup_{C_\varepsilon} \domain'_2$, where 
\[
\domain'_1=I\times K_\varepsilon 
\textup{ and } 
\domain'_2
=
\domain'\setminus\operatorname{int}(\domain'_1).
\]
Let $\overbar \nu_R$ denote the unit normal to $C_\varepsilon$ pointing out of
$\domain'_1$. Equivalently, $\overbar \nu_R$ is the  unit inner normal to
$C_\varepsilon$ as a boundary face of $\domain'_2$. Using the natural identification of the auxiliary trivial bundle $\underline{\mathbb R}^3$ with
the flat tangent bundle $T\domain'_1$, we use the same vector $\nu_R = \overbar \nu_R$  as the auxiliary vector on $C_\varepsilon$. In particular, we have 
\[
\langle\normal,\nu_R\rangle=0
\qquad
\textup{ along }C_\varepsilon\cap\overbar F_k.
\]
Moreover, the auxiliary vector $\nu_\Sigma$ on each end face is its geometric inner normal, which is also orthogonal to $\nu_R$.  Hence all hypotheses of
Theorem~\ref{thm:gluing} are satisfied, and
\[
\ind(D_{B'}^{\domain'})
=
\ind(D_{B'_1}^{\domain'_1})
+
\ind(D_{B'_2}^{\domain'_2}).
\]

\begin{claim}\label{claim:indexzero-cylinder}
	$\ind(D^{\domain'_1}_{B'_1}) = 0$.
\end{claim}
\begin{proof}
	The codimension-one faces of $\domain'_1=I\times K_\varepsilon$ are the two
	side faces $\overbar F_1$ and $\overbar F_2$, the two end faces
	\[
	\{0\}\times K_\varepsilon,
	\qquad
\{1\}\times K_\varepsilon,
	\]
	and the cylindrical face $C_\varepsilon=I\times R_\varepsilon$.
	Let $\overbar \nu_k$ denote the unit inner normal vector of $\overbar F_k$. Let $\nu_k$ also denote the unit inner normal vector of $\overbar F_k$, but viewed as a section of the auxiliary trivial bundle $\underline{\mathbb{R}}^3$ (which has been identified with $T\domain'_1$). By the arrangement made in \textbf{Steps~4} and \textbf{~5}, and preserved by the preliminary
	deformation above, the current auxiliary field $\normal$ satisfies
	\[
	\langle\normal,\overbar\nu_k\rangle>0
	\]
	on each side face. Proposition  \ref{thm:algebraicDeformation} therefore gives smooth unit vector fields $\nu_{k, t}$, $t \in [0, 1]$ on $\overbar{F}_k$ such that:\begin{itemize}\item $\nu_{k,0}=\nu_k$ and $\nu_{k,1}=\normal$,\item $\langle\nu_{k,t},\normal\rangle>0$ for all $t\in [0,1]$ and every $k \in \{1, 2\}$, and\item $\langle \nu_{1,t},\nu_{2,t}\rangle>\langle \nu_{1},\nu_{2}\rangle$ along $\overbar F_1\cap  \overbar F_2$ for all $t\in (0,1]$.\end{itemize}
	The auxiliary vectors on the end faces $\{0\}\times K_\varepsilon$ and 
	$\{1\}\times K_\varepsilon$,  and $C_\varepsilon$ are kept fixed. By construction, for every $t>0$, the strict inner product comparison holds along
	the edge $\Gamma = \overbar F_1\cap \overbar F_2$, while the equality case holds along all remaining edges of $\domain'_1$.  By Theorem~\ref{thm:ess-saVector}, the corresponding
	operators $D_{\widetilde B_t}^{\domain'_1}$ are essentially self-adjoint and
	Fredholm for $t>0$. Now the same argument used in the proof of Claim \ref{claim:indexzero} shows that $\ind(D^{\domain'_1}_{B'_1}) = 0$.
\end{proof}

Repeating this construction for every edge segment in
$\mathcal E_{\domain}$ gives a manifold $\domain''$, obtained by deleting the
interiors of the product neighborhoods $I\times K_\varepsilon$, with boundary
condition $B''$, such that
\[
\ind(D_{B'}^{\domain'})
=
\ind(D_{B''}^{\domain''}).
\]

We now fill each newly created cylindrical boundary component $I \times R_\varepsilon$  as follows. Choose a
smooth convex domain $U\subset\sector$ whose boundary consists of the circular
arc $R_\varepsilon$ and a smooth convex curve $\Lambda$. See the lower left picture in Figure \ref{fig:hownormalchanges} or the right picture in  Figure \ref{fig:edgecut}.  We require that
 $R_\varepsilon$ and
$\Lambda$ meet orthogonally at their endpoints. The curve $\Lambda$ replaces
the corner at the origin of $\mathbb G$ by a smooth convex arc.

Form the product $I\times U$ and glue it to $\domain''$ along the common
boundary hypersurface $I\times R_\varepsilon$. Let $\domain^U$ denote the
resulting manifold. Extend the auxiliary field $\normal$ from the adjacent side faces
to a smooth unit vector field $\widetilde\normal$ on the new face
$I\times\Lambda$. We choose this extension so that (see Figure \ref{fig:hownormalchanges}):
\begin{enumerate}[label=$(\roman*)$]
	\item it is tangent to the end faces $\{0\}\times U$ and
	$\{1\}\times U$ along $\{0\}\times \Lambda  $ and $\{1\}\times \Lambda  $;
	\item it has positive inner product with the unit inner normal of
	$I\times\Lambda$.
\end{enumerate}

\begin{figure}[h]
	\begin{tikzpicture}[scale=2]
		
		\begin{scope}[shift={(0,0)}]
			\draw[thick] (0,0) -- ({2*cos(40)},{2*sin(40)});
			\draw[thick] (0,0) -- ({-2*cos(40)},{2*sin(40)});
			
			\draw[dashed, thick] ({0.9*cos(40)},{0.9*sin(40)}) arc (40:140:0.9);
			
			\node[above] at (0,0.9) {$R_\varepsilon$};
			
			\draw[->,blue] (0,0) -- (0,0.3);
			
			\foreach \r in {0.15, 0.3, 0.45, 0.6, 0.75} {
				\pgfmathsetmacro{\ang}{90 - (\r/0.9)*40}
				\draw[->,blue] ({\r*cos(40)},{\r*sin(40)}) -- ({\r*cos(40)-0.3*cos(\ang)},{\r*sin(40)+0.3*sin(\ang)});
				\draw[->,blue] ({-\r*cos(40)},{\r*sin(40)}) -- ({-\r*cos(40)+0.3*cos(\ang)},{\r*sin(40)+0.3*sin(\ang)});
			}
			
			\foreach \r in {0.9, 1.05, 1.2, 1.35, 1.5, 1.65, 1.8} {
				\draw[->,blue] ({\r*cos(40)},{\r*sin(40)}) -- ({\r*cos(40)-0.3*cos(50)},{\r*sin(40)+0.3*sin(50)});
				\draw[->,blue] ({-\r*cos(40)},{\r*sin(40)}) -- ({-\r*cos(40)+0.3*cos(50)},{\r*sin(40)+0.3*sin(50)});
			}
		\end{scope}

		\begin{scope}[shift={(4.5,0)}]
			\def\sep{0.1}
			
			\begin{scope}[shift={(0,\sep)}]
				\draw[thick] ({0.9*cos(40)},{0.9*sin(40)}) -- ({2*cos(40)},{2*sin(40)});
				\draw[thick] ({-0.9*cos(40)},{0.9*sin(40)}) -- ({-2*cos(40)},{2*sin(40)});
				\draw[thick] ({0.9*cos(40)},{0.9*sin(40)}) arc (40:140:0.9);
				
				\foreach \r in {0.9, 1.05, 1.2, 1.35, 1.5, 1.65, 1.8} {
					\draw[->,blue] ({\r*cos(40)},{\r*sin(40)}) -- ({\r*cos(40)-0.3*cos(50)},{\r*sin(40)+0.3*sin(50)});
					\draw[->,blue] ({-\r*cos(40)},{\r*sin(40)}) -- ({-\r*cos(40)+0.3*cos(50)},{\r*sin(40)+0.3*sin(50)});
				}
			\end{scope}
			
			\begin{scope}[shift={(0,-\sep)}]
				\draw[thick] (0,0) -- ({0.9*cos(40)},{0.9*sin(40)});
				\draw[thick] (0,0) -- ({-0.9*cos(40)},{0.9*sin(40)});
				\draw[thick] ({0.9*cos(40)},{0.9*sin(40)}) arc (40:140:0.9);
				
				\draw[->,blue] (0,0) -- (0,0.3);
				
				\foreach \r in {0.15, 0.3, 0.45, 0.6, 0.75, 0.9} {
					\pgfmathsetmacro{\ang}{90 - (\r/0.9)*40}
					\draw[->,blue] ({\r*cos(40)},{\r*sin(40)}) -- ({\r*cos(40)-0.3*cos(\ang)},{\r*sin(40)+0.3*sin(\ang)});
					\draw[->,blue] ({-\r*cos(40)},{\r*sin(40)}) -- ({-\r*cos(40)+0.3*cos(\ang)},{\r*sin(40)+0.3*sin(\ang)});
				}
			\end{scope}
		\end{scope}

		\begin{scope}[shift={(0,-2.5)}]
			\def\sep{0.1}
			
			\begin{scope}[shift={(0,\sep)}]
				\draw[thick] ({0.9*cos(40)},{0.9*sin(40)}) -- ({2*cos(40)},{2*sin(40)});
				\draw[thick] ({-0.9*cos(40)},{0.9*sin(40)}) -- ({-2*cos(40)},{2*sin(40)});
				\draw[thick] ({0.9*cos(40)},{0.9*sin(40)}) arc (40:140:0.9);
				
				\foreach \r in {0.9, 1.05, 1.2, 1.35, 1.5, 1.65, 1.8} {
					\draw[->,blue] ({\r*cos(40)},{\r*sin(40)}) -- ({\r*cos(40)-0.3*cos(50)},{\r*sin(40)+0.3*sin(50)});
					\draw[->,blue] ({-\r*cos(40)},{\r*sin(40)}) -- ({-\r*cos(40)+0.3*cos(50)},{\r*sin(40)+0.3*sin(50)});
				}
			\end{scope}
			
			\begin{scope}[shift={(0,-\sep)}]
				\pgfmathsetmacro{\R}{0.9/tan(40)}
				\pgfmathsetmacro{\Yc}{0.9/sin(40)}
				
				\draw[thick] ({0.9*cos(40)},{0.9*sin(40)}) arc (40:140:0.9);
				\draw[thick] ({-0.9*cos(40)},{0.9*sin(40)}) arc (-130:-50:\R);
				
				\node[below, yshift=-1pt] at (0, {\Yc - \R}) {$\Lambda$};
				
				\foreach \ang in {-130, -120, -110, -100, -90, -80, -70, -60, -50} {
					\draw[->,blue] ({\R*cos(\ang)}, {\Yc + \R*sin(\ang)}) 
					-- ({\R*cos(\ang) + 0.3*cos(\ang+180)}, {\Yc + \R*sin(\ang) + 0.3*sin(\ang+180)});
				}
			\end{scope}
		\end{scope}

		\begin{scope}[shift={(4.5,-2.5)}]
			\pgfmathsetmacro{\R}{0.9/tan(40)}
			\pgfmathsetmacro{\Yc}{0.9/sin(40)}
			
			\draw[thick] ({0.9*cos(40)},{0.9*sin(40)}) -- ({2*cos(40)},{2*sin(40)});
			\draw[thick] ({-0.9*cos(40)},{0.9*sin(40)}) -- ({-2*cos(40)},{2*sin(40)});
			
			\draw[thick] ({-0.9*cos(40)},{0.9*sin(40)}) arc (-130:-50:\R);
			\draw[dashed, thick] ({0.9*cos(40)},{0.9*sin(40)}) arc (40:140:0.9);
			
			\foreach \r in {0.9, 1.05, 1.2, 1.35, 1.5, 1.65, 1.8} {
				\draw[->,blue] ({\r*cos(40)},{\r*sin(40)}) -- ({\r*cos(40)-0.3*cos(50)},{\r*sin(40)+0.3*sin(50)});
				\draw[->,blue] ({-\r*cos(40)},{\r*sin(40)}) -- ({-\r*cos(40)+0.3*cos(50)},{\r*sin(40)+0.3*sin(50)});
			}
			
			\foreach \ang in {-130, -120, -110, -100, -90, -80, -70, -60, -50} {
				\draw[->,blue] ({\R*cos(\ang)}, {\Yc + \R*sin(\ang)}) 
				-- ({\R*cos(\ang) + 0.3*cos(\ang+180)}, {\Yc + \R*sin(\ang) + 0.3*sin(\ang+180)});
			}
		\end{scope}
		
	\end{tikzpicture}
	\caption{Cross-sections of the cutting and pasting construction.
		Upper left: the product sector with cutting arc $R_\varepsilon$. Upper right:
		the two pieces obtained after cutting. Lower left: the smoothing domain $U$,
		bounded by $R_\varepsilon$ and $\Lambda$, together with the extension
		$\protect\widetilde\normal$. Lower right: the cross-section after gluing
		$I\times U$ to the truncated manifold.}
	\label{fig:hownormalchanges}
\end{figure}
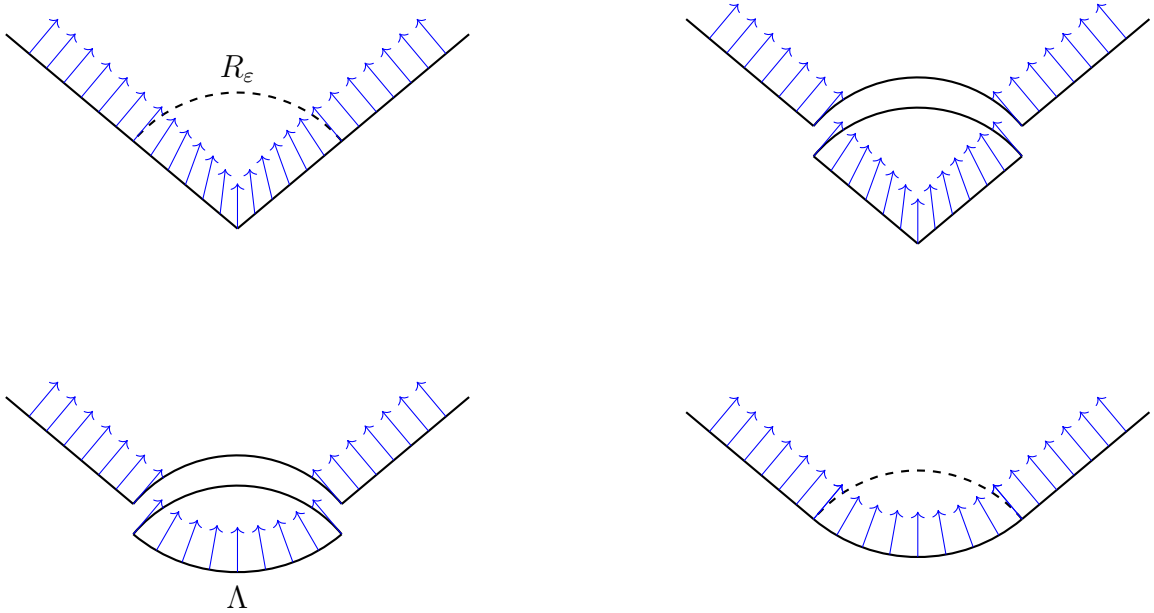

The auxiliary vector fields on the end faces are simply chosen to be the unit inner normal of the end faces, by using the natural identification of the auxiliary trivial bundle $\underline{\mathbb R}^3$ with the flat tangent bundle $T(I\times U)$. The gluing formula (Theorem \ref{thm:gluing}) gives 
$$\ind(D^{\domain^U}_{B}) = \ind(D^{I\times U}_{B_U}) + \ind(D^{\domain''}_{B''}),$$
where $B$, $B_U$, and $B''$ are the corresponding boundary conditions.

The same argument from the proof of Claim \ref{claim:indexzero-cylinder} shows that $\ind(D^{I\times U}_{B_U}) = 0$. Therefore, we have$$\ind(D^{\domain^U}_{B}) =  \ind(D^{\domain''}_{B''}).$$
Performing this gluing construction at every cylindrical boundary component
produces a manifold $\domain'''$ with boundary condition $B'''$ such that
\[
\ind(D_{B'''}^{\domain'''})
=
\ind(D_{B''}^{\domain''}).
\]

 \textbf{Step 7. Resolving the remaining singularities.}
 
 By construction, $\domain'''$ has no codimension-three strata. Consequently,
 each connected component of its codimension-two faces is a compact
 one-dimensional manifold without boundary, and hence is a circle.
 
 Let $\gamma$ be one such component. After choosing a trivialization of its
 oriented normal two-plane bundle, a sufficiently small tubular neighborhood of
 $\gamma$ may be identified with a sector bundle over $\mathbb S^1$. We first
 apply the metric and face deformations from \textbf{Step~5}, with $\mathbb S^1$ in
 place of the interval $I$, to deform this neighborhood into a genuine product
 \[
 \mathbb S^1\times\sector .
 \]
 We then repeat the cutting and pasting construction of \textbf{Step~6}. The auxiliary
 unit-length section $\normal$ is deformed and extended similarly as in that step.
 
 The same argument as in \textbf{Step~6} shows that this construction does not change the Fredholm index. Performing the construction
 in pairwise disjoint tubular neighborhoods of all remaining codimension-two
 faces  produces a manifold $\domain^s$ with smooth boundary, together with an
 auxiliary unit-length field $\normal^s$, such that
 \[
 \ind(D_{B'''}^{\domain'''})
 =
 \ind(D_{B^s}^{\domain^s}).
 \]

 \textbf{Conclusion.}
 Throughout \textbf{Steps 1--7}, the auxiliary vector field $\normal$ is deformed, extended, and modified in a manner that preserves the degree of the associated boundary map $\normal\colon \partial \domain \to \sph^2$. Consequently,
 \[
 \deg(\normal^s)=\deg(\normal).
 \]
 Recall that $\normal$ is the pullback under $f$ of the unit vector field defined in \eqref{eq:auxiliaryvector}, regarded as a section of the trivial bundle
 \[
 \partial\domain\times\mathbb{R}^3.
 \]
 By construction,
 \[
 \deg(-\normal)=\deg(f).
 \]
 Hence,
 \[
 \deg(-\normal^s)=\deg(-\normal)=\deg(f).
 \]
 
 Since $\domain^s$ has smooth boundary, the classical index theorem for manifolds with smooth boundary gives
 \[
 \ind\bigl(D^{\domain^s}_{B^s}\bigr)=\deg(-\normal^s).
 \]
 Therefore,
 \[
 \ind(D_B)
 =
 \ind\bigl(D^{\domain^s}_{B^s}\bigr)
 =
 \deg(-\normal^s)
 =
 \deg(-\normal)
 =
 \deg(f).
 \]
 This completes the proof of the theorem.
\end{proof}

With all the necessary ingredients now established, we are ready to prove Theorem \ref{thm:main2}.

\begin{theorem}[Theorem \ref{thm:main2}]\label{thm:main-rigidity}
	Let $(\target, g)$ be a convex polyhedron in the Euclidean space $\R^3$, where $g$ is the Euclidean metric. Let $(\domain,\overbar g)$ be a spin polyhedral manifold and $f\colon \domain \to \target$ be a polyhedral map with non-zero degree, such that the scalar curvature, mean curvature, and dihedral angles satisfies$$\Sc_{\overbar g}\geq 0,~H_{\overbar g}\geq 0,~\theta_{\overbar g}\leq f^*\theta_{g},$$  then $\Sc_{\overbar g}=0$, $H_{\overbar g}=0$, and $\theta_{\overbar g}=f^*\theta_{g}$. Moreover, $(\domain,\overbar g)$ is flat.
\end{theorem}
\begin{proof}
	Let $E=S(T\domain\oplus f^*T\target)=S(T\overbar \target\oplus \underline{\R}^3)$ and let $D$ be the Dirac operator associated with $E$. Let $B$ be the boundary condition given by$$\mathscr E \overbar c(\overbar \nu_k)c(\nu_k)\sigma=-\sigma \quad \textup{on } \overbar F_k,$$ where $\overbar \nu_k$ is the unit inner normal vector field of each codimension-one face $\overbar F_k$ of $\domain$, and $\nu_k$ is the unit inner normal vector field of the corresponding face $F_k$ of $\target$. 
	
	Let $\normal$ the be auxiliary vector field given in \eqref{eq:auxiliaryvector}. By Proposition~\ref{thm:algebraicDeformation}, there exists, on each face
	$F_k$, a smooth family of unit-length sections $\nu_{k,t}$ taking value in $\underline{\R}^3$ for $t\in[0,1]$, such
	that
	\[
	\nu_{k,0}=\nu_k,
	\qquad
	\nu_{k,1}=\normal,
	\]
	and the pairwise inner products
	\[
	\langle\nu_{i,t},\nu_{j,t}\rangle
	\]
	are non-decreasing in $t$. Pulling these sections back by $f$, define the
	boundary condition $B_t$ on $\domain$ by
	\[
	\mathscr E\,
	\overbar c(\overbar\nu_k)c(\nu_{k,t})\sigma
	=
	-\sigma
	\qquad
	\text{on }\overbar F_k.
	\]

	The assumption that $\theta_{\overbar g}\leq f^*\theta_{g}$ implies that
	$$\langle \overbar \nu_i(x),\overbar \nu_j(x)\rangle\leq\langle \nu_{i}(f(x)),\nu_{j}(f(x))\rangle$$
	at every $x \in \overbar F_i\cap\overbar F_j$ for each pair of adjacent codimension-one faces $\overbar F_i$ and $\overbar F_j$. Proposition~\ref{thm:algebraicDeformation} therefore imply that, for every
	$t>0$,
	\[
	\langle\overbar\nu_i(x),\overbar\nu_j(x)\rangle
	<
	\langle\nu_{i,t}(f(x)),\nu_{j,t}(f(x))\rangle
	\]
	along each  edge $\overbar F_i\cap\overbar F_j$.   
	
	It follows from Theorem~\ref{thm:ess-saVector} that, for every $t>0$, the
	operator $D$ with boundary condition $B_t$ is essentially self-adjoint and
	Fredholm. Moreover, Theorem~\ref{thm:indexTheoremVector} gives
	\[
	\ind(D_{B_t})=\deg(f).
	\]
   Since $\deg(f)\neq 0$, it follows that $\ind(D_{B_t})$ is non-zero for all $t\in (0, 1]$. Therefore, for each $t\in (0, 1]$, there exists a non-zero $\sigma_t\in H^1(\domain,E;B_t)$ such that $D\sigma_t=0$. 
	
	The family $B_t$ converges to $B_0=B$ in Lipschitz norm as $t\to0$. Moreover, by Proposition \ref{prop:smooth>=}, the boundary condition $B=B_0$ is extremal in the sense of Definition \ref{def:extremalBoundaryCondition}. By Lemma \ref{lemma:approximation}, there exists a non-zero $\sigma\in H^1(\domain,E;B)$ such that $\nabla\sigma=0$. In particular, $D\sigma = 0$. The conclusion of the theorem now follows immediately from Lemma \ref{lemma:angleRigidity}.
\end{proof}

%

\end{document}